\documentclass{amsart}
\usepackage[table]{xcolor}
\usepackage{graphicx,fullpage,comment,bm}
\usepackage{amsmath,amsfonts}
\usepackage[english]{babel}
\usepackage{amsthm}

\usepackage{algorithm}
\usepackage{textcomp}
\usepackage{algpseudocode}
\usepackage{bbm}
\usepackage{url}
\usepackage[shortlabels]{enumitem}
\usepackage{hyperref}

\graphicspath{{images/}}

\newtheorem{definition}{Definition}
\newtheorem{theorem}{Theorem}
\newtheorem{lemma}{Lemma}
\newtheorem{remark}{Remark}
\newtheorem{fact}{Fact}

\newtheorem{assumption}{Assumption}
\newcommand{\E}{\mathbb{E}}

\newcommand{\tens}[1]{\bm{\mathcal{#1}}}
\newcommand{\mat}[1]{\bm{#1}}
\def\tA{{\tens{A}}}  
\def\tB{{\tens{B}}}  

\def\tE{{\tens{E}}}

\def\tH{{\tens{H}}}
\def\tI{{\tens{I}}}

\def\tM{{\tens{M}}}

\def\tR{{\tens{R}}}
\def\tS{{\tens{S}}}
\def\tT{{\tens{T}}}
\def\tU{{\tens{U}}}
\def\tV{{\tens{V}}}
\def\tW{{\tens{W}}}
\def\tX{{\tens{X}}}  
\def\tY{{\tens{Y}}}
\def\tZ{{\tens{Z}}}

\def\vzero{{\bm{0}}}

\def\vb{{\bm{b}}}

\def\vh{{\bm{h}}}

\def\vu{{\bm{u}}}
\def\vv{{\bm{v}}}

\def\vx{{\bm{x}}}

\def\vz{{\bm{z}}}

\def\R{{\mathbb{R}}} 
\def\E{{\mathbb{E}}} 

\def\diag{{\math{diag}}}
\def\bcirc{{\mathrm{bcirc}}}
\def\unfold{{\mathrm{unfold}}}
\def\fold{{\mathrm{fold}}}

\def\bcirc{{\mathrm{bcirc}}}

\begin{document}

\title{Randomized Kaczmarz methods for t-product tensor linear systems with factorized operators}
\author{Alejandra Castillo, Jamie Haddock, Iryna Hartsock, Paulina Hoyos, Lara Kassab, Alona Kryshchenko, Kamila Larripa, Deanna Needell, Shambhavi Suryanarayanan, Karamatou Yacoubou-Djima}

\begin{abstract}
    Randomized iterative algorithms, such as the randomized Kaczmarz method, have gained considerable popularity due to their efficacy in solving matrix-vector and matrix-matrix regression problems. Our present work leverages the insights gained from studying such algorithms to develop regression methods for tensors, which are the natural setting for many application problems, e.g., image deblurring. In particular, we extend the randomized Kaczmarz method to solve a tensor system of the form $\tA\tX = \tB$, where $\tX$ can be factorized as $\tX = \tU\tV$, and all products are calculated using the t-product. We develop variants of the randomized factorized Kaczmarz method for matrices that approximately solve tensor systems in both the consistent and inconsistent regimes. We provide theoretical guarantees of the exponential convergence rate of our algorithms, accompanied by illustrative numerical simulations. Furthermore, we situate our method within a broader context by linking our novel approaches to earlier randomized Kaczmarz methods.
\end{abstract}
\maketitle
\section{Introduction}

Solving large-scale systems of linear equations or performing linear regressions is one of the most frequently encountered problems across the data-rich sciences. These problems arise in machine learning as subroutines of several optimization methods~\cite{boyd2004convex}, in medical imaging~\cite{Gordon1970,herman1993algebraic}, in sensor networks~\cite{savvides2001dynamic}, and in statistical analysis, to name only a few. In the matrix-vector and matrix-matrix regimes, these problems are well-understood; there exist highly efficient iterative methods with provable guarantees in the literature, e.g., Kaczmarz methods.
Recently, Kaczmarz-type methods have been proposed for a variety of linear systems and regression problems defined over \emph{tensors}~\cite{ma2022randomized,chen2021regularized,tang2023sketch}; these problems have the form
\vspace{-0.1em}
\begin{equation}
    \tA \tX = \tB,
\end{equation}
where $\tA \in \mathbb{R}^{p \times t \times q}$ is the measurement operator or dictionary, $\tB \in \mathbb{R}^{p \times s \times q}$ represents the measurements or data, $\tX \in \mathfrak{X} \subset \mathbb{R}^{s \times t \times q}$ is the signal of interest, and $\tA \tX$ is the t-product between $\tA$ and $\tX$ (see Definition~\ref{def:t-product}).

In some imaging applications (e.g., deblurring with known blurring operator), the measurement operator tensor $\tA$, is given in a factored or factorizable form~\cite{kilmer2013third}. For instance, in settings where multiple known blurring operators have sequentially degraded an image, the deblurring problem can be represented as a regression defined by an operator tensor $\tA = \tU \tV$ where $\tU$ and $\tV$ are the sequential deblurring operators. We consider the problem of solving a tensor linear system defined with a tensor operator given in factorized form, $\tA = \tU \tV$, using randomized Kaczmarz (RK) methods.
 In particular, given the system in the form
\begin{align}
    \tU  \tV  \tX = \tY \label{eq: full system}
\end{align}
where $\tU \in \R^{m \times m_1 \times p}, \tV \in \R^{m_1 \times n \times p}$, $\tX \in \R^{n \times l \times p}$ and $\tY \in \R^{m \times l \times p},$ we iteratively produce approximate solutions to the systems
\begin{align}
     \tU \tZ &= \tY, \text{ and }\label{eq:outer_system}\\
     \tV \tX &= \tZ,\label{eq:inner_system}
\end{align}
using variants of the tensor randomized Kaczmarz method~\cite{ma2022randomized}.

\subsection{Notation and Definitions}
We use regular lower-case Latin and Greek letters (e.g., $a$ and $\alpha$) to denote scalars, boldfaced lower-case Latin letters (e.g., $\vx$) to denote vectors, boldfaced upper-case Latin letters (e.g., $\mat A$) to denote matrices, and boldfaced upper-case calligraphic Latin letters (e.g., $\tA$) to denote higher-order tensors.
We use ``MATLAB" notation, e.g., $\mat A_{i :}$ is the $i$th row of the matrix $\mat A$ and $\tA_{: j :}$ is the $j$th column-slice of the tensor $\tA$. In addition, we let $\tA_k$ denote the $k$th frontal slice of $\tA$, i.e., $\tA_k := \tA_{::k}$.

Throughout this document, $\|\vv\|$ denotes the Euclidean norm of a vector $\vv$ and $\|\tA\|_F$ denotes the Frobenius norm of a tensor $\tA$.
Given a matrix $\mat A$, we denote its smallest  singular value by $\sigma_{\min}{(\mat A)}$ and its smallest \textit{positive} singular value by $\sigma_{\min +}(\mat A)$.

Note that \begin{align*}
    \langle \tA, \tB \rangle = \sum_{i,j,k} \tA_{ijk}\tB_{ijk}
\end{align*}
defines an inner product on spaces of real-valued tensors whose dimensions agree.  This inner product induces the Frobenious norm of a tensor $\tA$, i.e., $\langle \tA, \tA \rangle = \|\tA\|_F^2$.

We now provide background on the tensor-tensor t-product \cite{kilmer2011factorization}.
\begin{definition}
\label{def:t-product}
The \emph{tensor-tensor t-product} between $\tA \in \R^{m \times n \times p}$ and $\tB \in \R^{n \times l \times p}$ is defined as
\begin{equation*}
\tA \tB = \fold (\bcirc(\tA) \unfold (\tB)) \in \R^{m \times l \times p},
\end{equation*}
where $\bcirc (\tA)$ denotes the block-circulant matrix
\[\bcirc(\tA) = \begin{pmatrix}
\tA_{1} & \tA_{p} & \tA_{p-1} & \dots & \tA_{2} \\
\tA_{2} & \tA_{1} & \tA_{p} & \dots & \tA_{3} \\
\vdots & \vdots & \vdots & \dots & \vdots \\
\tA_{p} & \tA_{p-1} & \tA_{p-2} & \dots & \tA_{1} \\
\end{pmatrix} \in \R^{mp \times np},\]
and $\unfold(\tB)$ denotes the operation defined as
\[\unfold(\tB) = \begin{pmatrix}
 \tB_{1}\\
 \tB_{2}\\
 \vdots \\
 \tB_{p}
\end{pmatrix} \in \R^{np \times l}.\]
and $\fold(\unfold (\tB))) = \tB$.
\end{definition}

We use $\tA^*$ to denote the conjugate transpose of the tensor $\tA \in \mathbb{R}^{m \times n \times p}$, which is obtained by taking the conjugate transpose of each of the frontal slices $\tA_k$ and then reversing the order of the transposed frontal slices from 2 through $p$. The identity tensor $\tI \in \R^{m\times m \times p}$ is the tensor whose first frontal slice is the $m\times m$ identity matrix, and all other slices are the $m \times m$ zero matrix.

Finally, we note that the definition of the t-product implies that a tensor linear system may be reformulated as an equivalent (matrix) linear system.
\begin{fact}
   The tensor linear system $$\tA \tX = \tB$$ is equivalent to the matrix-matrix linear system $$\bcirc(\tA) \unfold(\tX) = \unfold(\tB);$$ that is, solutions to the tensor linear system, $\tX$, after unfolding, $\unfold(\tX)$, are solutions to the matrix linear system, and vice versa. \label{fact:equivalent_systems}
\end{fact}
We will exploit this fact to compare iterative methods for tensor linear systems to their counterpart iterative methods for matrix linear systems.

\begin{definition} Given a tensor $\tM$, we define the tensor $\mathcal{P}_\tM = \tM^* (\tM\tM^*)^{-1} \tM$.
\end{definition}

This is a projection operator, that is, $\mathcal{P}_\tM^2= \mathcal{P}_\tM$. It is an orthogonal projection onto the range of $\tM$.\\

We let $[m]$ denote the set $\{1, 2, \hdots, m\}$ and  $2^S$ denote the power set of a given set $S$.
\begin{definition}
\label{def:block_sampling_definitions}
  Given a tensor $\tA \in \R^{m \times n \times p}$, we let $T_{\tA} \subset 2^{[m]}$ be the index sets of the allowable (i.e., able to be selected) blocks of rows of $\tA$ and let $d_{T_{\tA}}$ be the size of the largest set in $T_{\tA}$.  We let
  $c_{\max}(T_{\tA}) = \text{argmax}_{i \in [m]} |\{\tau \in T_{\tA} : i \in \tau\}|$ denote the maximum number of blocks in which any one row slice appears. Let $\mathcal{D}(T_{\tA})$ be the sampling probability distribution over the set $T_{\tA}$ and let $prob(\tau)$ denote the probability with which a block is sampled.  We assume that all $\mu \in T_{\tA}$ are sets with no repeated elements; that is, each block contains any one row at most once, even if rows can appear in multiple blocks.
\end{definition}

\begin{assumption}
    We assume throughout that if $(\tM_{\gamma ::} \tM_{\gamma ::}^*)^{-1}$ is referred to, then $\tM_{\gamma ::} \tM_{\gamma ::}^*$ is invertible.  In particular, we require that the block row subsets of $\tU \in \R^{m \times m_1 \times p} $ and $\tV \in \R^{m_1 \times n \times p}$ all satisfy this required property, that is, the tensors $\tU_{\mu::}\tU_{\mu::}^*$ and $\tV_{\nu::}\tV_{\nu::}^*$ are invertible for all $\mu \in T_{\tU}$ and $\nu \in T_{\tV}$; and that the column subsets of $\tU$ all satisfy this required property, that is, $\tU_{:j:}^*\tU_{:j:}$ is invertible for all $j \in [m_1].$\label{assump:invertibility}
\end{assumption}
We note that for Assumption~\ref{assump:invertibility} to hold, we must have that the sampled blocks of $\tU$ and $\tV$ are wide tensors; that is, that $|\mu| \le m_1$ and $|\nu| \le n$.

\subsection{Related Work}

The Kaczmarz method is an iterative algorithm for solving linear systems of equations. It consists of iteratively updating an approximation by choosing one equation in the system and projecting the current approximation onto the hyperplane defined by that equation~\cite{Kaczmarz1937Angen}. More precisely, let $\mathbf{A}$ be an $m \times n$ matrix, $\mathbf{x}$ be a vector of unknowns, and $\mathbf{b}$ be a vector of constants. An update in the Kaczmarz algorithm is defined by
\begin{equation}\label{eqn:basicKaczmarz}
    \mathbf{x}^{(t)}
    = \mathbf{x}^{(t - 1)} + \frac{\mathbf{b}_{i_t} - \mathbf{A}_{i_t:}\mathbf{x}^{(t-1)}}{\Vert \mathbf{A}_{i_t:}\Vert^2}\mathbf{A}_{i_t:}^T,
\end{equation}
where $\mathbf{A}_{i_t:}$ is the selected $i_t$th row of $\mathbf{A}$ and $b_{i_t}$ is the selected $i_t$th entry of $\mathbf{b}$ at iteration $t$.

The Kaczmarz method is a simple yet highly efficient method, particularly for large, sparse systems of equations characterized by matrices where most entries are zero. It was later rediscovered for use in computerized tomography as the Algebraic Reconstruction Technique~\cite{herman1993algebraic}.
Additionally, this method has generated a plethora of extensions. The randomized Kaczmarz (RK) method, introduced in~\cite{strohmer2009randomized}, was the first variant shown to yield exponential convergence in expectation when selecting the rows at random according to a specified distribution, instead of selecting the rows cyclically as done in classical Kaczmarz. In this variant, the $i$th row of $\mathbf{A}$ is selected in the $t$th equation with probability $\Vert \mathbf{A}_i \Vert_2^2/\Vert \mathbf{A} \Vert_F^2$.
Other extensions include popular iterative algorithms \cite{Ma2015convergence, NSWjournal, Pop01:Fast-Kaczmarz-Kovarik, Pop04:Kaczmarz-Kovarik-Algorithm, frek}, acceleration and parallelization strategies~\cite{EN11:Acceleration-Randomized, liu2014asynchronous, morshed2019accelerated, moorman2020randomized}, and techniques for reducing noise and corruption~\cite{zouzias2013randomized, haddock2019randomized}.

\subsubsection{Extended Methods}
For inconsistent systems, that is, systems with no exact solution, the RK method does not converge to the ordinary least-squares solution
\[\mat{x_{LS}} = \underset{\mat{x} \in \R^n }{\text{argmin}}
 \|\mat{Ax} - \mat{b}\|_2^2, \]
as demonstrated in~\cite{needell2010randomized}. To overcome this issue, the Randomized Extended Kaczmarz (REK) method includes a random projection to iteratively reduce the component of $\mat{b}$ that is orthogonal to the range of $\mat{A}$. The updates in the REK algorithm are
\begin{equation}\label{eqn:ExtendedKaczmarz}
    \mathbf{x}^{(t)}
    = \mathbf{x}^{(t - 1)} + \frac{\mathbf{b}_{i_t} - \mat{z}_{i_t}^{(t-1)}-\mathbf{A}_{i_t:}\mathbf{x}^{(t-1)}}{\Vert \mathbf{A}_{i_t:}\Vert^2}\mathbf{A}_{i_t:}^T, \quad \mat{z}^{(t)} = \mat{z}^{(t-1)} - \frac{\langle \mat{A}_{: j_t}, z^{(t)} \rangle}{\|\mat{A}_{: j_k}\|^2}\mat{A}_{: j_t}.
\end{equation}
The iterate $\mat{z}^{(t)}$, initialized as $\mat{z}^{(0)} = \mat{y}$, approximates the component
of $\mat{y}$ that is orthogonal to the range of $\mat{A}$, allowing $\mat{x}^{(t)}$
to converge to the least-squares solution of the system $\mat{x_{LS}}$. In \cite{zouzias2013randomized}, the authors prove that the REK method
converges linearly in expectation to the solution $\mat{x_{LS}}$. Related works include an average block variant of the REK method \cite{Du2020averageblockREK}, and techniques robust to sparse noise and corruption~\cite{haddock2018randomized,haddock2019randomized,quantHNRS20,steinerberger2021quantile}.

\subsubsection{Block Methods}
In~\cite{needell2013paved, popa2012kaczmarz}, the authors present a block variant of the Kaczmarz method to solve the system $\mat{A}\vx = \vb$. In \cite{needell2013paved} the authors let $T_{\mat{A}}$ be a partition $\{\tau_1, \dots, \tau_k\}$ of the row indices $[m]$ of $\mat A$ and select the uniform probability distribution for $\mathcal{D}(T_{\mat A})$. At iteration $t$, the current estimate $\vx^{(t -1)}$ is projected onto the solution space of $\mat{A}_{\tau_t}\vx = \vb_{\tau_t}$ using the pseudo-inverse of the matrix $\mat A_{\tau_t}$:
\[ \vx^{(t)} = \vx^{(t-1)} - \mat{A}_{\tau_t}^\dagger( \mat{A}_{\tau_t}\vx^{(t-1)} - \vb_{\tau_t}). \]
The process continues until it converges. The authors prove the algorithm has an expected linear rate of convergence until it reaches a convergence horizon about the true solution. This result is generalized in \cite{HJY21}  to the case when the least-squares problem is rank-deficient, and it relaxes the requirement that the row blocks be sampled from a matrix paving. Moreover, the authors prove that the convergence horizon depends upon the minimum nonzero singular value of the blocks $\mat A_\tau$ rather than the absolute minimum singular value, which is often 0. Article \cite{zhang2024blockmatrix} presents a block RK method to solve linear feasibility problems without requiring the calculation of the pseudo-inverse of $\mat A_\tau$. It is shown that the method converges linearly in expectation without reliance on any special row paving.

\subsubsection{Methods for Factorized Systems}
In \cite{ma2018iterative}, the authors produce variants of RK to tackle the problem of solving a system of linear equations where the measurement matrix is given as a product of two matrices. Given a matrix $\mat{A} \in \mathbb{R}^{m \times n}$ with factorization $\mat{A} = \mat{U} \mat{V}$, where $\mat{U} \in \mathbb{R}^{m \times r}$ and $\mat{V} \in \mathbb{R}^{r \times n}$, the authors solve the system $\mat{A} \vx = \vb$ in the form
\begin{equation}\label{eqn:matricfactorized}
    \mat{U} \mat{V} \vx = \vb,
\end{equation}
by applying the randomized Kaczmarz method to the individual systems
\begin{equation}
    \mat{U} \vz = \vb, \text{ and } \label{eqn:subsystems1}
\end{equation}
\begin{equation}
    \mat{V} \vx = \vz\label{eqn:subsystems2}.
\end{equation}
The systems \eqref{eqn:subsystems1} and~\eqref{eqn:subsystems2} are solved by interlacing Kaczmarz steps, i.e., finding updates for both systems in each iteration of the algorithm such that the most recent update for $\vz^{(t)}$ is used to compute the next update $\vx^{(t)}$. The authors of~\cite{ma2018iterative} develop two algorithms, RK-RK and REK-RK, that efficiently compute the optimal solution for (\ref{eqn:matricfactorized}) using the factors of $\mat{A}$ in the consistent and inconsistent settings, respectively.

We note that the REK and interlaced RK-RK approaches for factorized systems are quite similar.  In each case, one interlaces RK steps to an updating linear system with updates to the system measurement vector which bring it closer to consistency.  In the REK case, the updates to the system learn the portion of the measurement vector $\vb$ orthogonal to the column-space of $\mat{A}$ and remove this component.  In the factorized system case, the updates to the measurement vector of the inner system $\vz$ are produced by RK steps on the outer system, $\mat{U}\vz = \vb$, which bring the inner system nearer to consistency under the assumptions of~\cite{ma2018iterative}.

\subsubsection{Tensor Methods}

Tensor regression problems arise organically in settings in which the model inputs or outputs are naturally formulated as a multidimensional tensor array, and the tensor product governs the dependence structure between input and output tensors. Examples include weather and climate forecasting, age estimation from medical imaging data or other biomarker information, and many others; see~\cite{liu2022tensor} for more details and an excellent survey of tensor regression models.  In this paper, we will be concerned with tensor regression under the \emph{tensor t-product}.
The closest tensor method related to our work is the Tensor Randomized Kaczmarz (TRK) method proposed in~\cite{ma2022randomized}, a generalization of the RK method for tensor linear systems defined under the t-product. The TRK method finds an approximation to the solution $\tX^*$ of the tensor linear system $\tA \tX = \tB$ by iteratively sampling a row slice of the system defined by $\tA_{i_t::}$ and $\tB_{i_t::}$ and projects the previous iterate onto the space of solutions to this sampled subsystem. The authors prove that this method converges at least linearly in expectation to the unique solution $\tX^*$ of the system. As a result of the equivalence noted in Fact \ref{fact:equivalent_systems}, the TRK method can analogously be viewed as a block Kaczmarz method applied to a matrix-matrix system for a particular choice of blocks. However, as expounded by the authors in \cite{ma2022randomized}, the theoretical guarantees afforded by TRK are stronger than the guarantees from a naive application of the block Kaczmarz method in the equivalent matrix-matrix setting. This gain can be attributed to restriction of the choice of blocks due to the nature of the TRK iterates. This also motivates the advantages of developing methods that utilize the natural structure of the data in such tensor problems.

\begin{algorithm}
\caption{Tensor Randomized Kaczmarz (TRK) \cite{ma2022randomized} }\label{alg:TRK}
	\begin{algorithmic}[1]
		\Procedure{TRK}{$\tA,\tB,\tX^{(0)}, K$}
		\For{$t = 1, \ldots, T$}
		      \State{Sample $i_t \in [m]$.}
                \State{$\tX^{(t)} = \tX^{(t-1)} - \tA_{i_t : :}^*(\tA_{i_t : :} \tA_{i_t : :}^*)^{-1}(\tA_{i_t : :}\tX^{(t-1)} - \tB_{i_t : :})$}
		\EndFor{} \\
		\Return{$\tX^{(T)}$}
		\EndProcedure
	\end{algorithmic}
\end{algorithm}

In \cite{Bao2022RandAvg}, the authors compute the least Frobenius-norm solution for consistent tensor linear systems of the form
        \begin{equation}\label{eq: tensor linear system AX=B}
            \tA \tX = \tB,
        \end{equation}
        where $\tA \in \R^{m \times n \times p}$, $\tX \in \R^{n \times l \times p}$, and $\tB \in \R^{m \times l \times p}$.
The authors propose the tensor randomized average Kaczmarz (TRAK) method, which is pseudoinverse- and inverse-free and offers a speed-up over the TRK method for solving consistent tensor linear systems~\eqref{eq: tensor linear system AX=B}. In~\cite{huang2023tensor}, the authors proposed an extended variant of TRK for the case of inconsistent tensor systems, which they refer to as the tensor randomized extended Kaczmarz (TREK) method.  The authors prove that the iterates converge at least linearly in expectation to the solution of the unperturbed tensor system.

\subsection{Contributions}

In this paper, we propose methods to solve problem \eqref{eq: full system} in both the consistent and inconsistent system regime. Our first method, named FacTBRK obtains an optimal solution for \eqref{eq: full system} when both the outer system and inner systems are consistent, and is a hybrid of TRK~\cite{ma2022randomized}, FacRK~\cite{ma2018iterative}, and block approaches~\cite{needell2013paved, popa2012kaczmarz}.

\begin{algorithm}
	\caption{Factorized Tensor Block Randomized Kaczmarz (FacTBRK)}\label{alg:facTBRK}
	\begin{algorithmic}[1]
		\Procedure{FacTBRK}{$\tU,\tV,\tY,
  T$}
            \State{$\tZ^{(0)} = \vzero$}
            \State{$\tX^{(0)} = \vzero$}
		\For{$t = 1, \ldots, T$}
		      \State{Sample $\mu_t \sim \mathcal{D}(T_{\tU})$.}
                \State{$\tZ^{(t)} = \tZ^{(t-1)} - \tU_{\mu_t : :}^*(\tU_{\mu_t : :} \tU_{\mu_t: :}^*)^{-1}(\tU_{\mu_t::}\tZ^{(t-1)} - \tY_{\mu_t: :})$}
                \State{Sample $\nu_t \sim \mathcal{D}(T_{\tV})$.}
                \State{$\tX^{(t)} = \tX^{(t-1)} - \tV_{\nu_t : :}^*(\tV_{\nu_t : :} \tV_{\nu_t : :}^*)^{-1}(\tV_{\nu_t : :}\tX^{(t-1)} - \tZ_{\nu_t : :}^{(t)})$}
		\EndFor{}
		\Return{$\tX^{(T)}$}
		\EndProcedure
	\end{algorithmic}
\end{algorithm}

The key features of Algorithm (\ref{alg:facTBRK}) are the random sampling of tensor blocks on which each update is projected and the interlaced Kaczmarz steps (5-6) and (7-8). The first part of our main result, Theorem~\ref{theorem:factorized}~\eqref{thm:facTBRK}, illustrates that this method converges at least linearly in expectation on consistent tensor linear systems under some mild assumptions.

The second method, named FacTBREK, tackles the same problem as before, except that the outer system \eqref{eq:outer_system} can now be inconsistent. To develop this method, we use the same tools as for FacTBRK combined with the tensor analogs of the ideas used in the analysis of the randomized extended Kaczmarz method~\cite{zouzias2013randomized}. We prove that each of these methods converges at least linearly in expectation.

\begin{algorithm}
	\caption{Factorized Tensor Block Randomized Extended Kaczmarz (FacTBREK)}\label{alg:facTBREK}
	\begin{algorithmic}[1]
		\Procedure{facTBREK}{$\tU,\tV,\tY,
  T$}
            \State{$\tW^{(0)} = \tY$}
            \State{$\tZ^{(0)} = \vzero$}
            \State{$\tX^{(0)} = \vzero$}
		\For{$t = 1, \ldots, T$}
                \State{Sample $l_t \in [m_1]$}
                \State{$\tW^{(t)} = \tW^{(t-1)} - \tU_{:l_t:}(\tU^\ast_{:l_t :}\tU_{:l_t :})^{-1}(\tU^\ast_{:l_t:}\tW^{(t-1)})$}
		      \State{Sample $\mu_t \sim \mathcal{D}(T_\tU)$.}
                \State{$\tZ^{(t)} = \tZ^{(t-1)} - \tU_{\mu_t : :}^*(\tU_{\mu_t : :} \tU_{\mu_t: :}^*)^{-1}(\tU_{\mu_t: :}\tZ^{(t-1)} - \tY_{\mu_t: :} +\tW^{(t)}_{\mu_t : :})$}
                \State{Sample $\nu_t \sim \mathcal{D}(T_\tV)$.}
                \State{$\tX^{(t)} = \tX^{(t-1)} - \tV_{\nu_t : :}^*(\tV_{\nu_t : :} \tV_{\nu_t : :}^*)^{-1}(\tV_{\nu_t : :}\tX^{(t-1)} - \tZ_{\nu_t : :}^{(t)})$}
		\EndFor{}
		\Return{$\tX^{(T)}$}
		\EndProcedure
	\end{algorithmic}
\end{algorithm}

Here, it is interesting to observe that while the $\tW^{(t)}$ update is reminiscent of a column action update, it is equivalent to performing a TRK update on the t-linear system $\tU^*\tW = 0$. Also, we would like to note that this is different from a Gauss-Seidel-like column update as it cannot be viewed as a coordinate minimization.

The second part of our main result, Theorem~\ref{theorem:factorized}~\eqref{thm:facTBREK}, illustrates that this method converges at least linearly in expectation to the least-squares solution of an inconsistent tensor linear system under some mild assumptions.

\begin{definition}
    Let $\tU \in \R^{m \times m_1 \times p}$ and $\tV \in \R^{m_1 \times n \times p}.$  Define $\alpha_\tV = 1 - \sigma_{\min} (\E[\text{bcirc}(\mathcal{P}_{\tV_{\nu::}})])$, $\alpha_\tU = 1 - \sigma_{\min} (\E[\text{bcirc}(\mathcal{P}_{\tU_{\mu::}})])$, $\beta_\tU = 1 - \sigma_{\min} \left(\E[\text{bcirc}(\mathcal{P}_{\tU^\star_{:l:}})]\right)$, $\alpha_{\max} = \max\{\alpha_\tU, \alpha_\tV\}$, $\alpha_{\min} = \min\left\{ \frac{\alpha_\tV}{\alpha_\tU} , \frac{\alpha_\tU}{\alpha_\tV} \right\}$, $\phi_{\max} := \max\{\alpha_\tU, \beta_\tU\}$, $\phi_{\min} := \min\left\{ \frac{\alpha_{\tU}}{\beta_{\tU}}, \frac{\beta_{\tU}}{\alpha_{\tU}}\right\}$,
\[ \theta_\tV = \begin{cases} \frac{d_{T_{\tV}}p c_{\max}(T_\tV)}{\sum_{\nu \in T_\tV} \sigma_{\min +}^2(\bcirc(\tV_{\nu::}))} & \text{if } prob(\nu_t) = \frac{\sigma_{\min +}^2(\bcirc(\tV_{\nu_t::}))}{\sum_{\nu \in T_\tV} \sigma_{\min +}^2(\bcirc(\tV_{\nu::}))}\\ d_{T_{\tV}}p \E \left[\sigma_{\min +}^{-2}(\bcirc(\tV_{\nu : :}))\right] & \text{else} \end{cases},
\]
and
\[ \theta_\tU = \begin{cases} \frac{d_{T_{\tU}}p c_{\max}(T_\tU)}{\sum_{\mu \in T_\tU} \sigma_{\min +}^2(\bcirc(\tU_{\mu::}))} & \text{if } prob(\mu_t) = \frac{\sigma_{\min +}^2(\bcirc(\tU_{\mu_t::}))}{\sum_{\mu \in T_\tU} \sigma_{\min +}^2(\bcirc(\tU_{\mu::}))}\\ d_{T_{\tU}}p \E \left[\sigma_{\min +}^{-2}(\bcirc(\tU_{\mu : :}))\right] & \text{else} \end{cases}.
\]\label{def:convergence_constants}
Here, the constants $d_{T_\tV}, d_{T_\tU}, c_{max}(T_\tV)$ and $c_{max}(T_\tU)$ are as defined in Definition~\ref{def:block_sampling_definitions}.
\end{definition}

\begin{theorem}
\label{theorem:factorized}
Let $\tU \in \R^{m \times m_1 \times p}, \tV \in \R^{m_1 \times n \times p}$, and $\tY \in \R^{m \times l \times p}.$  Let $\alpha_{\tV}, \alpha_{\tU}, \beta_\tU, \alpha_{\max}, \alpha_{\min}, \phi_{\min}, \theta_{\tV},$ and $\theta_{\tU}$ be as defined in Definition~\ref{def:convergence_constants}.

\begin{enumerate}
    \item Suppose $\tU\tV\tX = \tY$ is a consistent tensor linear system with the outer system $\tU\tZ = \tY$ consistent with unique solution $\tZ^\ddagger$ and the inner system $\tV\tX = \tZ^\ddagger$ consistent.
    Let $\tX^\ddagger$ be the tensor of minimal Frobenius norm such that $\tU \tV \tX^\ddagger = \tY$ and $\tX^{(t)}$ the $t$-th approximation of $\tX^\ddagger$ given by the updates of Algorithm~\ref{alg:facTBRK}. Then, the expected error at the $t$-th iteration satisfies
    \[\mathbb{E}\| \tX^{(t)} - \tX^\ddagger \|_F^2 \le \begin{cases} \left(\alpha_\tV^t + \theta_\tV  \frac{\alpha_{\max}^t \alpha_{\min}}{1 - \alpha_{\min}}\sigma_{\max}^2(\bcirc(\tV))\right)\|\tX^\ddagger\|_F^2 & \text{if } 0 \not= \alpha_\tU \neq \alpha_\tV \not= 0 \\
     \left(\alpha_\tV^t + \theta_\tV t \alpha_{\max}^t\sigma_{\max}^2(\bcirc(\tV))\right) \|\tX^\ddagger\|_F^2 & \text{ else}.
    \end{cases}\]\label{thm:facTBRK}

 \item Suppose $\tZ^\ddagger$ is the unique minimizer of $\|\tU\tZ - \tY\|_F^2$ and the inner system $\tV\tX = \tZ^\ddagger$ is consistent.
    Let $\tX^\ddagger$ be the tensor of minimal Frobenius norm such that $\tV \tX^\ddagger = \tZ^\ddagger$ and $\tX^{(t)}$ the $t$-th approximation of $\tX^\ddagger$ given by the updates of Algorithm~\ref{alg:facTBREK}. Then, if $\alpha_\tV \not= 0$, the expected error at the $t$-th iteration satisfies
    \begin{align*}
    &\mathbb{E}\| \tX^{(t)} - \tX^\ddagger \|_F^2 \\
    &\le  \begin{cases}
      \left(\alpha_\tV^t + \theta_\tV \sigma_{\max}^2(\bcirc(\tV)) (t+1) \alpha_\tV^{\lfloor t/2 \rfloor}\left(\alpha_\tU^{\lfloor t/2 \rfloor}  + \theta_\tU  \frac{\phi_{\max}^{\lfloor t/2 \rfloor} \phi_{\min}}{1 - \phi_{\min}} \sigma_{\max}^2(\bcirc(\tU))\right)\right)\|\tX^\ddagger\|_F^2 & \text{ if } 0 \not= \alpha_\tU \not= \beta_\tU \not= 0 \\
     \left(\alpha_\tV^t + \theta_\tV \sigma_{\max}^2(\bcirc(\tV)) (t+1) \alpha_\tV^{\lfloor t/2 \rfloor}\left(\alpha_\tU^{\lfloor t/2 \rfloor}  + \theta_\tU t \phi_{\max}^{\lfloor t/2 \rfloor}\sigma_{\max}^2(\bcirc(\tU))\right)\right)\|\tX^\ddagger\|_F^2 & \text{ else}.
  \end{cases}
  \end{align*}\label{thm:facTBREK}
\end{enumerate}
\end{theorem}

\begin{remark}
    We note that when the blocks are of size one, we have $d_{T_{\tU}} = 1$ and $d_{T_{\tV}} = 1$.  If the set of allowable blocks are $T_{\tU} = [m]$ and $T_{\tV} = [m_1]$, then $\alpha_\tU$ and $\alpha_\tV$ are equivalent to the parameter $\rho$ defined in~\cite{ma2022randomized} (Section 4) for  the tensors $\tU$ and $\tV$.
\end{remark}

\begin{remark}
     We note that these results hold for matrix-matrix and matrix-vector systems when the third tensor dimension is $p = 1$.  These results generalize those of~\cite{ma2018iterative, zouzias2013randomized} for matrix-vector systems, and provide new results for the matrix-matrix system regime.
\end{remark}

\begin{remark}
 When $\tV = \tI$, $|\nu_t| = m_1$ and $|\mu_t| = 1$, then FacTBRK coincides with the TRK method~\cite{ma2022randomized} and our result bounds the  expected error at the t-th iteration as
     \[\mathbb{E}\| \tX^{(t)} - \tX^\ddagger \|_F^2 \le  m_{1}pt \alpha_{\tU}^t \| \tX^{(0)} - \tX^\ddagger \|_F^2,\]
which matches the bound derived in~\cite{ma2022randomized} up to a constant. \end{remark}

\subsection{Organization of the paper} In Section \ref{section:convergence}, we provide theoretical convergence guarantees for all our proposed algorithms.  Some fundamental lemmas that will be used in the main proofs of the paper are detailed in Subsection~\ref{sub:fundamental_lemmas}. In Subsection~\ref{subsec:tbrk}, we provide theoretical guarantees for TBRK, a block variant of the TRK algorithm, and Subsection~\ref{subsec:TBREK} details the convergence analysis of TBREK, an extended block version of the TBREK algorithm. With these results in place, in Subsection~\ref{subsec:factbrk_and_factbrek}, we then provide convergence guarantees for the FacTBRK and FacTBREK algorithms for solving factorized t-product regression problems.
In Subsection~\ref{sec:Synthetic_Data}, we provide a suite of numerical experiments on synthetic data that analyze the performance of our proposed algorithms in different settings. Subsequently, in Subsection~\ref{subsec:video_deblurring}, we present the results of applying our proposed methods to video deblurring, specifically, for recovering frames from an MRI video that have been sequentially blurred twice.

\section{Convergence Analysis}
\label{section:convergence}
In this section, we prove that Algorithm~\ref{alg:facTBRK} and Algorithm~\ref{alg:facTBREK} converge at least linearly in expectation, as stated in Theorem~\ref{theorem:factorized}.  Our results generalize that of~\cite{ma2018iterative}, Theorem 4.1.

Note that the $t^{th}$ iterations of FacTBRK, Algorithm~\ref{alg:facTBRK}, and FacTBREK, Algorithm~\ref{alg:facTBREK}, involve sampling of index sets $\mu_t \in T_{\tU}$, $\nu_t \in T_{\tV}$, and $l_t \in [m_1]$.  These samples are independent, so the joint density function of these variables factorizes into the products of the marginal density functions, and the expectation with respect to the joint distribution, $\mathbb{E}$, may be calculated by sequential expectation with respect to marginal distributions of the individual samples.  We denote expectation with respect to the marginal distribution of $\mu_t$ as $\mathbb{E}_{T_{\tU}}$ and expectation with respect to the marginal distribution of $\nu_t$ as $\mathbb{E}_{T_{\tV}}$. We will use this fact to simplify calculations in the sections that follow.

\subsection{Fundamental Lemmas}
\label{sub:fundamental_lemmas}
We begin by establishing three fundamental lemmas. The first lemma establishes orthogonality between two components of the error vectors we will be bounding.  This result allows us to apply a ``Pythagorean theorem''-like result, which simplifies our analysis.

\begin{lemma}\label{lem:orthogonality}
We have that
\begin{equation*}
    \langle (\tI - \mathcal{P}_{\tM})\tX, \tM^* \tY \rangle = 0, \text{ and thus, } \|(\tI - \mathcal{P}_{\tM})\tX + \tM^* \tY\|_F^2 = \|(\tI - \mathcal{P}_{\tM})\tX\|_F^2 + \|\tM^* \tY\|_F^2\label{eq:governing_orthogonality}
\end{equation*}
for any $\tX, \tY$ and $\tM$ such that $(\tM\tM^*)^{-1}$ exists.
\end{lemma}

\begin{proof}
We prove this result by noting that
\begin{align*}
    \langle (\tI - \mathcal{P}_\tM) \tX, \tM^* \tY \rangle &= (\tI - \tM^*(\tM\tM^*)^{-1} \tM) \tX, \tM^* \tY \rangle \\
    & = (\tM - \tM\tM^* (\tM\tM^*)^{-1} \tM) \tX, \tY \rangle = 0,
\end{align*}
where the second equation follows from \cite[Lemma 2.7]{chen2021regularized}.
The Pythagorean theorem follows from using orthogonality in the inner product expansion of the error norm.
\end{proof}

The second fundamental lemma that we prove gives us an upper bound on the Frobenius norm of a product of two tensors.  It additionally relates the expectation of the residual error $\| \mathcal{P}_{\tV_{\nu_t::}}(\tX^\ddagger - \tX^{(t-1)} )\|_F^2$ to that of the approximation error $\|  \tX^\ddagger - \tX^{(t-1)} \|_F^2$.  This type of theoretical bound is common among analyses of Kaczmarz-type methods and generally hinges on the minimum singular value of the projector matrix.  In this case, as the projector operator is a tensor, we utilize the minimum singular value of the expectation of a matricization of the tensor, $ \E_t\left[\text{bcirc}\left(\mathcal{P}_{\tV_{\nu_t::}}\right)\right]$.

\begin{lemma}
\label{lem:residual_error_bound}
Suppose the t-product of $\tU$ and $\tX$ is defined, then $$\|\tU \tX\|_F^2 \le \sigma_{\max}^2(\bcirc(\tU)) \|\tX\|_F^2.$$
Suppose $\tM$ is a random tensor variable such that $\tM\tM^*$ is invertible and that $\tX$ and $\tM$ are independent, then
$$\mathbb{E}_{\tM}\left[  \| \mathcal{P}_{\tM}\tX\|_F^2  \right]
        \ge \sigma_{min}(\mathbb{E}_\tM[\bcirc(\mathcal{P}_{\tM})]) \|  \tX \|_F^2.$$
\end{lemma}
\begin{proof}
    We first note that for any $\tU$, $\|\tU\tX\|_F^2$ may be rewritten and bounded as
\begin{align}
    \|\tU\tX\|_F^2 &= \| \text{bcirc}(\tU) \text{unfold}(\tX) \|_F^2 \nonumber\\&= \sum_{i} \| \text{bcirc}(\tU) \text{unfold}(\tX)_{: i}\|_2^2 \label{eq:column_decoupling}\\
    &\le \sigma_{\max}^2(\bcirc(\tU)) \sum_{i} \|\text{unfold}(\tX)_{: i}\|_2^2 = \sigma_{\max}^2(\bcirc(\tU)) \|\tX\|_F^2\nonumber.
\end{align}

Now, we replace $\tU$ with $\mathcal{P}_\tM$ and begin from~\eqref{eq:column_decoupling} by expanding the norm using the inner product, we have
\begin{align*}
    \mathbb{E}_{\tM}\left[  \| \mathcal{P}_{\tM}\tX\|_F^2  \right]
    &= \sum_i \mathbb{E}_{\tM} \left[ \langle \bcirc(\mathcal{P}_{\tM}) \text{unfold}(\tX)_{:i}, \bcirc(\mathcal{P}_{\tM}) \text{unfold}(\tX)_{:i}  \rangle \right] \\
    &= \sum_i \langle \mathbb{E}_{\tM}[\bcirc(\mathcal{P}_{\tM})] \text{unfold}(\tX)_{:i}, \unfold(\tX)_{:i} \rangle,
\end{align*}
where the last equation follows since $\bcirc(\mathcal{P}_{\tM})$ is an orthogonal projection.
Since $\mathbb{E}_{\tM}[\bcirc(\mathcal{P}_{\tM})] $ is symmetric, we have
\begin{align*}
        \mathbb{E}_{\tM}\left[\|\mathcal{P}_{\tM}(\tX)\|_F^2  \right]
        &\geq \sigma_{min}(\mathbb{E}_{\tM}[\bcirc(\mathcal{P}_{\tM})]) \sum_{i} \|  \text{unfold}(\tX)_{:i} \|^2 \\
        &= \sigma_{min}(\mathbb{E}_{\tM}[\bcirc(\mathcal{P}_{\tM})]) \sum_{i} \|\tX_{:i:}\|_F^2\\
        &= \sigma_{min}(\mathbb{E}_{\tM}[\bcirc(\mathcal{P}_{\tM})]) \|\tX\|_F^2.
\end{align*}
\end{proof}

The third fundamental lemma provides an upper bound on the Frobenius norm of matrices of the form $\bcirc(\tM^*(\tM\tM^*)^{-1})$ such that $(\tM\tM^*)^{-1}$ exists.

\begin{lemma}
    If $\tM \in \mathbb{R}^{l \times n \times p}$ then $\|\mathrm{bcirc}(\tM^*(\tM \tM^*)^{-1})\|_F^2 \le lp \sigma_{min +}^{-2}(\bcirc(\tM))$. \label{lem:almost_projector_bound}
\end{lemma}
\begin{proof}
    Note that
    \begin{align*}
        \|\bcirc(\tM^*(\tM \tM^*)^{-1})\|_F^2 &= \langle \bcirc(\tM)^* (\bcirc(\tM)\bcirc(\tM^*))^{-1}, \bcirc(\tM)^* (\bcirc(\tM)\bcirc(\tM^*))^{-1} \rangle \\
        &= \langle \tI_{lp},  (\bcirc(\tM)\bcirc(\tM^*))^{-1} \rangle\\
        &\le \|\tI_{lp}\|_F \|(\bcirc(\tM)\bcirc(\tM^*))^{-1}\|_F \\
        &\le \sqrt{lp} \sqrt{lp} \|(\bcirc(\tM)\bcirc(\tM^*))^{-1}\|_2,
    \end{align*}
    where the second equation follows from applying the adjoint operator in the inner product and simplifying, the first inequality follows from Cauchy-Schwartz, and the last from usual Frobenius-spectral norm inequalities.
\end{proof}

\subsection{Tensor block randomized Kaczmarz (TBRK)}
\label{subsec:tbrk}
FacTBRK, Algorithm~\ref{alg:facTBRK}, utilizes interlaced steps of tensor block randomized Kaczmarz, whose pseudocode we include below in Algorithm~\ref{alg:TBRK} for convenience.

\begin{remark}
    Note that TBRK, Algorithm~\ref{alg:TBRK}, is recovered as a special case of FacTBRK, Algorithm~\ref{alg:facTBRK}, when $\tV = \tI$ (so $m_1 = n$) and $T_{\tV} = \{[n]\}$.
\end{remark}

\begin{algorithm}
	\caption{Tensor Block Randomized Kaczmarz (TBRK)}\label{alg:TBRK}
	\begin{algorithmic}[1]
		\Procedure{TBRK}{$\tU,\tY,T$}
            \State{$\tX^{(0)} = \vzero$}
		\For{$t = 1, \ldots, T$}
		      \State{Sample $\mu_t \sim \mathcal{D}(T_\tU)$.}
                \State{$\tX^{(t)} = \tX^{(t-1)} - \tU_{\mu_t : :}^*(\tU_{\mu_t : :} \tU_{\mu_t: :}^*)^{-1}(\tU_{\mu_t::}\tX^{(t-1)} - \tY_{\mu_t: :})$}
		\EndFor{}
		\Return{$\tX^{(T)}$}
		\EndProcedure
	\end{algorithmic}
\end{algorithm}

We begin by proving that when applied to an inconsistent system, TBRK, Algorithm~\ref{alg:TBRK} converges at least linearly in expectation up to a convergence horizon which depends on the error in the system measurements.

\begin{lemma}[Convergence Horizon of TBRK]
\label{lemma:trk_convergence horizon}Consider tensors $\tU \in \R^{m\times m_1 \times p}$ and $\tY$, $\tE \in \R^{m \times l \times p}$. Let $\tX^\ddagger$  be the solution to the consistent system $\tU\tX = \tY$ that minimizes $\|\tX^{(0)} - \tX\|_F$.  Then, the iterates, $\tX^{(t)},$ of Algorithm~\ref{alg:TBRK} applied to the system defined by $\tU$ and measurements $\tY + \tE$ satisfies
 $$ \E^{(t-1)}\| \tX^{(t)} - \tX^{\ddagger}\|^2 \leq \alpha_\tU \| \tX^{(t-1)} - \tX^{\ddagger} \|^2 + d_{T_\tU} p \E^{(t-1)} \left[\sigma_{\min +}^{-2} \left(\text{bcirc}(\tU_{\mu::})\right)\|\tE_{\mu::}\|_F^2 \right]$$
 \end{lemma}

 \begin{proof}
 Note that
 \begin{align*}
     \tX^{(t)} - \tX^\ddagger  &= \tX^{(t-1)} - \tX^{\ddagger} - \tU_{\mu_t::}^* (\tU_{\mu_t::} \tU_{\mu_t::}^*)^{-1} \left( \tU_{\mu_t::} \tX^{(t-1)} - \tY_{\mu_t::}-\tE_{\mu_t::} \right)\\
     &= \tX^{(t-1)} - \tX^{\ddagger} - \mathcal{P}_{\tU_{\mu_t::}} \left(  \tX^{(t-1)} - \tX^{\ddagger}\right) + \tU_{\mu_t::}^* (\tU_{\mu_t::} \tU_{\mu_t::}^*)^{-1}\tE_{\mu_t::} \\
     &= (\tI - \mathcal{P}_{\tU_{\mu_t::}})(\tX^{(t-1)} - \tX^{\ddagger}) +  \tU_{\mu_t::}^* (\tU_{\mu_t::} \tU_{\mu_t::}^*)^{-1}\tE_{\mu_t::}.
 \end{align*}

By Lemma~\ref{lem:orthogonality} eq.~\eqref{eq:governing_orthogonality}, we have $$\|\tX^{(t)} - \tX^\ddagger\|_F^2 = \|(\tI - \mathcal{P}_{\tU_{\mu_t::}})(\tX^{(t-1)} - \tX^{\ddagger})\|_F^2 +  \|\tU_{\mu_t::}^* (\tU_{\mu_t::} \tU_{\mu_t::}^*)^{-1}\tE_{\mu_t::}\|_F^2.$$
This yields
\begin{align} \E^{(t-1)} \| \tX^{(t)} - \tX^{\ddagger}\|^2 &\leq (1 - \sigma_{\min} (\E[\text{bcirc}(\mathcal{P}_{\tU_{\mu::}})])) \| \tX^{(t-1)} - \tX^{\ddagger} \|^2 + \E^{(t-1)}\|\tU_{\mu::}^* (\tU_{\mu::} \tU_{\mu::}^*)^{-1} \tE_{\mu::}\|^2 \nonumber \\
\label{eqn:Convergence horizon}
&\leq (1 - \sigma_{\min} (\E[\text{bcirc}(\mathcal{P}_{\tU_{\mu::}})])) \| \tX^{(t-1)} - \tX^{\ddagger} \|^2 + d_{T_\tU} p \E^{(t-1)} \left[\sigma_{\min +}^{-2} \left(\text{bcirc}(\tU_{\mu::})\right)\|\tE_{\mu::}\|_F^2 \right],
\end{align}
where the first inequality follows from Lemma~\ref{lem:residual_error_bound} and the second from Lemma~\ref{lem:almost_projector_bound}.
 \end{proof}

\subsection{Tensor block randomized extended Kaczmarz (TBREK)}
\label{subsec:TBREK}
FacTBREK, Algorithm~\ref{alg:facTBREK}, utilizes interlaced steps of tensor block randomized Kaczmarz, Algorithm~\ref{alg:TBRK}, and tensor block randomized extended Kaczmarz, whose pseudocode we include below in Algorithm~\ref{alg:TBREK} for convenience. This method is a block extension of the TREK method proposed in~\cite{du2021randomized}. Note that TBREK, Algorithm~\ref{alg:TBREK}, is recovered as a special case of FacTBREK, Algorithm~\ref{alg:facTBREK}, when $\tV = \tI$ (so $m_1 = n$) and the $T_{\tV} = \{[n]\}$.  We include an convergence analysis of TBREK in Theorem~\ref{thm:tbrek_conv}.

\begin{algorithm}
	\caption{Tensor Block Randomized Extended Kaczmarz (TBREK)}\label{alg:TBREK}
	\begin{algorithmic}[1]
		\Procedure{TBREK}{$\tU,\tY,T$}
            \State{$\tX^{(0)} = \vzero$}
            \State{$\tW^{(0)} = \tY$}
		\For{$t = 1, \ldots, T$}
		      \State{Sample $l_t \in [m_1]$.}
\State{$\tW^{(t)} = \tW^{(t-1)} - \tU_{:l_t:}(\tU^\ast_{:l_t :}\tU_{:l_t :})^{-1}(\tU^\ast_{:l_t:}\tW^{(t-1)})$}
		      \State{Sample $\mu_t \sim \mathcal{D}(T_\tU)$.}
                \State{$\tX^{(t)} = \tX^{(t-1)} - \tU_{\mu_t : :}^*(\tU_{\mu_t : :} \tU_{\mu_t: :}^*)^{-1}(\tU_{\mu_t::}\tX^{(t-1)} - \tY_{\mu_t: :} +\tW^{(t)}_{\mu_t : :})$}
		\EndFor{}
		\Return{$\tX^{(T)}$}
		\EndProcedure
	\end{algorithmic}
\end{algorithm}

Before we can prove Theorem~\ref{thm:tbrek_conv}, we prove Lemma~\ref{lem:usual_least_squares_theory} which illustrates that t-product least-squares problems enjoy the same orthogonal subspace properties as least-squares defined with the usual matrix product.

\begin{lemma}\label{lem:usual_least_squares_theory}
    Let $\tY_{\tR(\tU)^\perp}$ be the solution of $\tU^* \tW = \mathbf{0}$ that minimizes $\|\tY - \tW\|_F$.  Define $\tY_{\tR(\tU)} = \tY - \tY_{\tR(\tU)^\perp}.$  Let $\tX^\ddagger$ be the least-norm minimizer of $\|\tU \tX - \tY\|_F^2$.  Then $$\tY_{\tR(\tU)} = \tU \tX^\ddagger.$$
\end{lemma}

\begin{proof}
    We simply appeal here to the usual least-squares theory of linear algebra.  We note that $\tY_{\tR(\tU)^\perp}$ can be equivalently computed by folding the solution of $\bcirc(\tU)^* \unfold(\tW) = \unfold(\mathbf{0})$ that minimizes $\|\unfold(\tY) - \unfold(\tW)\|_F$.  Similarly, $\tX^{\ddagger}$ can be equivalently computed by folding $\mat{X}'$, where $$ \mat{X}' =  \underset{\mat{X} }{\text{argmin}} \ \|\bcirc(\tU)\mat{X} - \unfold(\tY)\|_F^2.$$
    Thus, from usual linear algebraic facts, we have $\unfold(\tY) - \unfold(\tY_{\tR(\tU)^\perp}) = \bcirc(\tU) \unfold(\tX^\ddagger)$.  We finally note that $\unfold(\tY) - \unfold(\tY_{\tR(\tU)^\perp}) = \unfold(\tY - \tY_{\tR(\tU)^\perp}) = \unfold(\tY_{\tR(\tU)})$.  The result then follows from folding.
\end{proof}

Now we turn to Theorem~\ref{thm:tbrek_conv}.  This result illustrates that Algorithm~\ref{alg:TBREK} converges at least linearly in expectation to the least-squares minimizer.

\begin{theorem}
\label{thm:tbrek_conv}[Convergence of TBREK]
Let $\tU \in \R^{m \times m_1 \times p}, \tV \in \R^{m_1 \times n \times p}$, and $\tY \in \R^{m \times l \times p}.$  Let $\alpha_{\tV}$, $\alpha_{\tU},$ $\beta_\tU,$  $\alpha_{\max},$ $\alpha_{\min},$ $\phi_{\min},$ $\theta_{\tV},$ and $\theta_{\tU}$ be as defined in Definition~\ref{def:convergence_constants}.
Let $\tX^\ddagger$ be the unique minimizer of the least-squares objective, $\|\tU \tX - \tY\|_F^2$.   Given the iterates $\{\tX^{(t)}\}$ of TBREK, Algorithm \ref{alg:TBREK}, applied to $\tU$ and $\tY$ satisfy
\begin{equation*}
\mathbb{E}\| \tX^{(t)} - \tX^\ddagger \|_F^2 \le \begin{cases}
    \left(\alpha_\tU^t  + \theta_\tU  \frac{\phi_{\max}^t \phi_{\min}}{1 - \phi_{\min}} \sigma_{\max}^2(\bcirc(\tU))\right)\|\tX^\ddagger\|_F^2 & \text{ if } \alpha_\tU \not= \beta_\tU \\
    \left(\alpha_\tU^t  + \theta_\tU t \phi_{\max}^t\sigma_{\max}^2(\bcirc(\tU))\right) \|\tX^\ddagger\|_F^2 & \text{ else}
\end{cases}
\end{equation*}
\end{theorem}

\begin{proof}
Consider the TBREK extended updates
$$\tW^{(t)} = \tW^{(t-1)} - \tU_{:l_t:}(\tU^\ast_{:l_t :}\tU_{:l_t :})^{-1}\tU^\ast_{:l_t:}\tW^{(t-1)} = \left( \tI - \mathcal{P}_{\tU^*_{:l_t:}}\right) \tW^{(t-1)}.$$
These iterates coincide with those of tensor randomized Kaczmarz applied to solve the system
$$ \tU^{*}\tW = \bm{0} \ \textrm{ with }  \ \tW^{(0)} = \tY.$$
By Lemma~\ref{lemma:trk_convergence horizon} and the definition of $\tY_{\tR(\tU)^\perp}$ given in Lemma~\ref{lem:usual_least_squares_theory}, we have
\begin{equation}
\label{eqn:convergence_to_range}\E^{(t-1)} \left[\|\tW^{(t)} - \tY_{\tR(\tU)^\perp} \|^2\right] \leq \beta_\tU \|\tW^{(t-1)} - \tY_{\tR(\tU)^\perp}\|^2
\end{equation}
where $ \beta_\tU = 1 - \sigma_{\min} \left(\E[\text{bcirc}(\mathcal{P}_{\tU^\star_{:l_t:}})]\right) < 1.$

Now, consider the update for $\tX^{(t)}$ in Algorithm \ref{alg:TBREK},
$$\tX^{(t)} = \tX^{(t-1)} - \tU_{\mu_t : :}^*(\tU_{\mu_t : :} \tU_{\mu_t: :}^*)^{-1}(\tU_{\mu_t::}\tX^{(t-1)} - \tY_{\mu_t: :} +\tW^{(t)}_{\mu_t : :})$$
This is equivalent to the TBRK update with system $\tU\tX = \tY -\tW^{(t)} = \tY_{\tR(\tU)} + \tY_{\tR(\tU)^\perp} + \tW^{(t)} $.
Applying Lemma \ref{lemma:trk_convergence horizon} on this system, we have
\begin{align*}
    \E^{(t-1)} \left[\|\tX^{(t)} - \tX^{\ddagger}\|\right]^2 &\le \alpha_\tU \left\|\tX^{(t-1)} - \tX^\ddagger\right\|^2 + d_{T_{\tU}} p \E^{(t-1)} \left[ \sigma_{\min}^{-2}(\bcirc(\tU_{\mu_t::}))\left\| \tW^{(t)}_{\mu_t::} - \left[ \tY_{\tR(\tU)^\perp} \right]_{\mu_t::}\right\|^2\right]
\end{align*}

Now, we consider two cases.  In the first case, we make no assumptions on $\mathcal{D}(T_{\tU})$.  In this case, we have
\begin{align*}
    d_{T_{\tU}} p \E^{(t-1)} \left[ \sigma_{\min}^{-2}(\bcirc(\tU_{\mu_t::}))\left\| \tW^{(t)}_{\mu_t::} - \left[ \tY_{\tR(\tU)^\perp} \right]_{\mu_t::}\right\|^2\right] &\le d_{T_{\tU}} p \E^{(t-1)} \left[ \sigma_{\min}^{-2}(\bcirc(\tU_{\mu_t::}))\left\| \tW^{(t)} -  \tY_{\tR(\tU)^\perp} \right\|^2\right] \\
    &= d_{T_{\tU}} p \E^{(t-1)} \left[ \sigma_{\min}^{-2}(\bcirc(\tU_{\mu_t::}))\right] \E^{(t-1)}\left[\left\| \tW^{(t)} -  \tY_{\tR(\tU)^\perp} \right\|^2\right].
\end{align*}

In the second case, we assume that the probability that block $\mu$ is selected is proportional to $\sigma_{\min}^2(\bcirc(\tU_{\mu::}))$, and have
\begin{align*}
    d_{T_{\tU}} p \E^{(t-1)} \left[ \sigma_{\min}^{-2}(\bcirc(\tU_{\mu_t::}))\left\| \tW^{(t)}_{\mu_t::} - \left[ \tY_{\tR(\tU)^\perp} \right]_{\mu_t::}\right\|^2\right] &\le \frac{d_{T_{\tU}}p c_{\max}(T_\tU)}{\sum_{\mu \in T_\tU} \sigma_{\min +}^2(\bcirc(\tU_{\mu::}))} \E^{(t-1)}\left[\left\| \tW^{(t)} -  \tY_{\tR(\tU)^\perp} \right\|^2\right].
\end{align*}

Thus, combining this with~\eqref{eqn:convergence_to_range}, we have
\begin{align*}
    \E^{(t-1)} \left[\|\tX^{(t)} - \tX^{\ddagger}\|\right]^2 &\le \alpha_\tU \left\|\tX^{(t-1)} - \tX^\ddagger\right\|^2 + \theta_\tU \beta_\tU \left\| \tW^{(t-1)} -  \tY_{\tR(\tU)^\perp} \right\|^2,
\end{align*}
and by recursion, we have
\begin{align*}
    \E \left[\|\tX^{(t)} - \tX^{\ddagger}\|\right]^2 &\le \alpha_\tU^t \left\|\tX^{(0)} - \tX^\ddagger\right\|^2 + \theta_\tU\sum_{s=1}^t  \beta_\tU^s \alpha_\tU^{t-s} \left\| \tW^{(0)} -  \tY_{\tR(\tU)^\perp} \right\|^2 \\
    &= \alpha_\tU^t \left\|\tX^{(0)} - \tX^\ddagger\right\|^2 + \theta_\tU\sum_{s=1}^t  \beta_\tU^s \alpha_\tU^{t-s} \left\|\tY_{\tR(\tU)}\right\|^2.
\end{align*}

Let  $\phi_{\max} := \max\{\alpha_\tU, \beta_\tU\}$ and $\phi_{\min} := \min\left\{ \frac{\alpha_{\tU}}{\beta_{\tU}}, \frac{\beta_{\tU}}{\alpha_{\tU}}\right\}$. When $\alpha_\tU \neq \beta_{\tU}$, by recursion we have
\begin{align*}
    \mathbb{E}\| \tX^{(t)} - \tX^\ddagger \|_F^2 &\le \alpha_\tU^t \|\tX^{(0)} - \tX^\ddagger\|_F^2 + \theta_\tU\sum_{s=1}^t  \beta_\tU^s \alpha_\tU^{t-s} \left\|\tY_{\tR(\tU)}\right\|^2
    \\&\le \alpha_\tU^t \|\tX^\ddagger\|_F^2 + \theta_\tU  \phi_{\max}^t \sum_{s=1}^{t} \phi^s_{\min}\|\tY_{\tR(\tU)}\|_F^2\\
    &\le \alpha_\tU^t \|\tX^\ddagger\|_F^2 + \theta_\tU  \frac{\phi_{\max}^t \phi_{\min}}{1 - \phi_{\min}}\|\tY_{\tR(\tU)}\|_F^2.\end{align*}
When $\alpha_\tU = \beta_\tU$, we have
\begin{align*}
    \mathbb{E}\| \tX^{(t)} - \tX^\ddagger \|_F^2 &\le \alpha_\tU^t \|\tX^{(0)} - \tX^\ddagger\|_F^2 + \theta_\tU\sum_{s=1}^t  \beta_\tU^s \alpha_\tU^{t-s} \left\|\tY_{\tR(\tU)}\right\|^2
    \\&=  \alpha_\tU^t \|\tX^\ddagger\|_F^2 + \theta_\tU t \phi_{\max}^t \left\|\tY_{\tR(\tU)}\right\|^2.
\end{align*}
The result follows now from Lemma~\ref{lem:usual_least_squares_theory} and Lemma~\ref{lem:residual_error_bound}, which yields $\|\tY_{\tR(\tU)}\|_F^2 = \|\tU \tX^\ddagger\|_F^2 \le \sigma_{\max}^2(\bcirc(\tU)) \|\tX^\ddagger\|_F^2$.
\end{proof}

\subsection{FacTBRK and FacTBREK}
\label{subsec:factbrk_and_factbrek}
We now present the proof of Theorem~\ref{theorem:factorized}.  Our final lemma analyzes the single iteration convergence of the iterates of Algorithm~\ref{alg:facTBRK} and Algorithm~\ref{alg:facTBREK} defined on the inner system~\ref{eq:inner_system}, $\tX^{(t)}$, in terms of the expected error of the current values of the iterates defined on the outer system~\ref{eq:outer_system}, $\tZ^{(t)}$.  This lemma allows us to interlace expectation with respect to the sampled block for the inner system, $\nu_t$ and the sampled block for the outer system, $\mu_t$.

\begin{lemma}
With the assumptions of Theorem~\ref{theorem:factorized}, we have
$$\E^{(t-1)}\| \tX^{(t)} - \tX^\ddagger \|^2
  \leq  (1- \sigma_{\min}( {\E[\text{bcirc}( \mathcal{P}_{\tV_{\nu_t::}})]) }) \| {\tX^{(t-1)}} - \tX^\ddagger\|^2 + d_{T_{\tV}}p \E \left[\sigma_{\min +}^{-2}(\bcirc(\tV_{\nu_t : :}))\right] \E_{T_\tU}^{(t-1)}\left[\|( {\tZ^{\ddagger}} - {\tZ^{(t)}}) \|^2\right].$$
  If in particular we have that the sampling distribution $\mathcal{D}(T_{\tV})$ is such that the probability of sampling $\nu_t$ is $\sigma_{\min +}^2(\bcirc(\tV_{\nu_t::}))/\sum_{\nu \in T_{\tV}} \sigma_{\min +}^2(\bcirc(\tV_{\nu::}))$ then
  $$\E^{(t-1)}\| \tX^{(t)} - \tX^\ddagger \|^2
  \leq (1- \sigma_{\min}( {\E[\text{bcirc}( \mathcal{P}_{\tV_{j_t::}})]) }) \| {\tX^{(t-1)}} - \tX^\ddagger\|^2 + \frac{d_{T_{\tV}}p c_{\max}(T_\tV)}{\sum_{\nu \in T_\tV} \sigma_{\min +}^2(\bcirc(\tV_{\nu::}))} \E_{T_\tU}^{(t-1)}\|( {\tZ^{\ddagger}} - {\tZ^{(t)}})\|^2.$$ \label{lem:expectation_decoupling}
\end{lemma}

\begin{proof}
By Lemma~\ref{lemma:trk_convergence horizon}, we have
\begin{align*}
    \E^{(t-1)}&\| \tX^{(t)} - \tX^\ddagger \|^2      \\
  \leq& d_{T_{\tV}}p \E_{T_\tU}^{(t-1)} \E_{T_\tV}^{(t-1)}  \left[\sigma_{\min +}^{-2}(\bcirc(\tV_{\nu_t : :})) \|( {\tZ_{\nu_t : :}^{\ddagger}} - {\tZ^{(t)}_{\nu_t : :}}) \|^2\right]  + (1- \sigma_{\min}( {\E_{T_\tV}^{(t-1)}[\text{bcirc}( \mathcal{P}_{\tV_{\nu_t::}})]) }) \| {\tX^{(t-1)}} - \tX^\ddagger\|^2
\end{align*}

Now, we consider two cases.  In the first case, we make no assumptions on $\mathcal{D}(T_{\tV})$.  In this case, we have
\begin{align*}
    \E^{(t-1)}&\| \tX^{(t)} - \tX^\ddagger \|^2    \\
  \leq&~   d_{T_{\tV}}p \E_{T_\tU}^{(t-1)} \E_{T_\tV}^{(t-1)}  \left[\sigma_{\min +}^{-2}(\bcirc(\tV_{\nu_t : :})) \|( {\tZ_{\nu_t : :}^{\ddagger}} - {\tZ^{(t)}_{\nu_t : :}}) \|^2\right]  + (1- \sigma_{\min}( {\E[\text{bcirc}( \mathcal{P}_{\tV_{\nu_t::}})]) }) \| {\tX^{(t-1)}} - \tX^\ddagger\|^2 \\
  \leq& ~   d_{T_{\tV}}p \E_{T_\tV}^{(t-1)}  \left[\sigma_{\min +}^{-2}(\bcirc(\tV_{\nu_t : :}))\right] \E_{T_\tU}^{(t-1)}\left[\|( {\tZ^{\ddagger}} - {\tZ^{(t)}}) \|^2\right]  + (1- \sigma_{\min}( {\E[\text{bcirc}( \mathcal{P}_{\tV_{\nu_t::}})]) }) \| {\tX^{(t-1)}} - \tX^\ddagger\|^2 \\
  =& ~   d_{T_{\tV}}p \E \left[\sigma_{\min +}^{-2}(\bcirc(\tV_{\nu_t : :}))\right] \E_{T_\tU}^{(t-1)}\left[\|( {\tZ^{\ddagger}} - {\tZ^{(t)}}) \|^2\right]  + (1- \sigma_{\min}( {\E[\text{bcirc}( \mathcal{P}_{\tV_{\nu_t::}})]) }) \| {\tX^{(t-1)}} - \tX^\ddagger\|^2
\end{align*}
where the second inequality follow from the coarse bound $\|( {\tZ_{j_t : :}^{\ddagger}} - {\tZ^{(t)}_{j_t : :}}) \| \le \|( {\tZ^{\ddagger}} - {\tZ^{(t)}}) \|$ and the independence of $\tZ^{(t)}$ and $\nu_t$.

In the second case, we assume that the probability that block $\nu$ is selected is proportional to $\sigma_{\min +}^2(\bcirc(\tV_{\nu::}))$, and have
\begin{align*}
  \E^{(t-1)}&\| \tX^{(t)} - \tX^\ddagger \|^2      \\
  \leq& ~   d_{T_{\tV}}p \E_{T_\tU}^{(t-1)} \E_{T_{\tV}}^{(t-1)}  \left[\sigma_{\min +}^{-2}(\bcirc(\tV_{\nu_t : :})) \|( {\tZ_{\nu_t : :}^{\ddagger}} - {\tZ^{(t)}_{\nu_t : :}}) \|^2\right]  + (1- \sigma_{\min}( {\E[\text{bcirc}( \mathcal{P}_{\tV_{\nu_t::}})]) }) \| {\tX^{(t-1)}} - \tX^\ddagger\|^2 \\
        \leq&  d_{T_{\tV}}p \E_{T_\tU}^{(t-1)}\frac{c_{\max}(T_\tV) \|( {\tZ^{\ddagger}} - {\tZ^{(t)}})\|^2 }{\sum_{\nu \in T_\tV} \sigma_{\min +}^2(\bcirc(\tV_{\nu::}))}  + (1- \sigma_{\min}( {\E[\text{bcirc}( \mathcal{P}_{\tV_{j_t::}})]) }) \| {\tX^{(t-1)}} - \tX^\ddagger\|^2.
\end{align*}
\end{proof}

We may now prove Theorem~\ref{theorem:factorized}.
\begin{proof}[Proof of Theorem~\ref{theorem:factorized}]
We first prove Part~\eqref{thm:facTBRK}.  With \[ \theta_\tV = \begin{cases} \frac{d_{T_{\tV}}p c_{\max}(T_\tV)}{\sum_{\nu \in T_\tV} \sigma_{\min +}^2(\bcirc(\tV_{\nu::}))} & \text{if } p(\nu_t) = \frac{\sigma_{\min +}^2(\bcirc(\tV_{\nu_t::}))}{\sum_{\nu \in T_\tV} \sigma_{\min +}^2(\bcirc(\tV_{\nu::}))}\\ d_{T_{\tV}}p \E \left[\sigma_{\min +}^{-2}(\bcirc(\tV_{\nu_t : :}))\right] & \text{else} \end{cases},
\]
we have by Lemma~\ref{lem:expectation_decoupling} that
\[\E^{(t-1)}\| \tX^{(t)} - \tX^\ddagger \|^2
  \leq  \alpha_\tV \| {\tX^{(t-1)}} - \tX^\ddagger\|^2 + \theta_\tV \E_{T_\tU}^{(t-1)}\left[\|( {\tZ^{\ddagger}} - {\tZ^{(t)}}) \|^2\right]\]
Note that the iterations updating the $\tZ$ iterates are independent of the $\tX$ iterates. This gives us
\[\mathbb{E}^{t-1}\| \tX^{(t)} - \tX^\ddagger \|^2 \leq \alpha_\tV\|\tX^{(t-1)} - \tX^\ddagger \|^2 + \theta_{\tV} \alpha_\tU {\|\tZ^{(t-1)} - \tZ^\ddagger\|^2}\]

Let  $\alpha_{\max} := \max\{\alpha_\tU, \alpha_\tV\}$ and $\alpha_{\min} := \min\left\{ \frac{\alpha_{\tU}}{\alpha_{\tV}}, \frac{\alpha_{\tV}}{\alpha_{\tU}}\right\}$. When $0 \neq \alpha_\tU \neq \alpha_{\tV} \neq 0$, by recursion we have
\begin{align*}
    \mathbb{E}\| \tX^{(t)} - \tX^\ddagger \|_F^2 &\le \alpha_\tV^t \|\tX^{(0)} - \tX^\ddagger\|_F^2 + \theta_\tV \left(\sum_{s = 1}^t \alpha_\tU^s \alpha_\tV^{t-s}\right) \|\tZ^{(0)} - \tZ^\ddagger\|_F^2
    \\&\le \alpha_\tV^t \|\tX^\ddagger\|_F^2 + \theta_\tV  \alpha_{\max}^t \sum_{s=1}^{t} \alpha^s_{\min}\|\tZ^\ddagger\|_F^2\\
    &\le \alpha_\tV^t \|\tX^\ddagger\|_F^2 + \theta_\tV  \frac{\alpha_{\max}^t \alpha_{\min}}{1 - \alpha_{\min}}\|\tZ^\ddagger\|_F^2.\end{align*}
When $\alpha_\tU = \alpha_\tV$, we have
\begin{align*}
    \mathbb{E}\| \tX^{(t)} - \tX^\ddagger \|_F^2 &\le \alpha_\tV^t \|\tX^{(0)} - \tX^\ddagger\|_F^2 + \theta_\tV \left(\sum_{s = 1}^t \alpha_\tU^s \alpha_\tV^{t-s}\right) \|\tZ^{(0)} - \tZ^\ddagger\|_F^2
    \\&\le \alpha_\tV^t \|\tX^\ddagger\|_F^2 + \theta_\tV t \alpha_{\max}^t \|\tZ^\ddagger\|_F^2.
\end{align*}
The result then follows from Lemma~\ref{lem:residual_error_bound} and the fact that $\tV \tX^\ddagger = \tZ^\ddagger$.

 We now prove part~\eqref{thm:facTBREK}.  From Lemma~\ref{lem:expectation_decoupling}, we know that
\begin{align*}  \E^{(t-1)}\| \tX^{(t)} - \tX^\ddagger \|^2
  \leq&    \alpha_\tV \| {\tX^{(t-1)}} - \tX^\ddagger\|^2  + \theta_\tV \E^{(t-1)}\left[\|( {\tZ^{\ddagger}} - {\tZ^{(t)}}) \|^2\right],
\end{align*}
and thus, we have
\begin{align*}  \E\| \tX^{(t)} - \tX^\ddagger \|^2
  \leq&    \alpha_\tV \E\| {\tX^{(t-1)}} - \tX^\ddagger\|^2  + \theta_\tV \E\left[\|( {\tZ^{\ddagger}} - {\tZ^{(t)}}) \|^2\right]\\
  \leq&    \alpha_\tV \E\| {\tX^{(t-1)}} - \tX^\ddagger\|^2  + \theta_\tV \gamma(t) \|{\tZ^{\ddagger}} \|_F^2\\
  \leq&    \alpha_\tV \E\| {\tX^{(t-1)}} - \tX^\ddagger\|^2  + \theta_\tV \gamma(t) \sigma_{\max}^2(\bcirc(\tV))\|{\tX^{\ddagger}} \|_F^2,
\end{align*}
where the last inequality follows from Lemma~\ref{lem:residual_error_bound} and the second inequality follows from Theorem~\ref{thm:tbrek_conv} with
$$\gamma(t) := \begin{cases}
    \alpha_\tU^t  + \theta_\tU  \frac{\phi_{\max}^t \phi_{\min}}{1 - \phi_{\min}} \sigma_{\max}^2(\bcirc(\tU)) & \text{ if } \alpha_\tU \not= \beta_\tU \\
     \alpha_\tU^t  + \theta_\tU t \phi_{\max}^t\sigma_{\max}^2(\bcirc(\tU)) & \text{ else}.
\end{cases}.$$
Recursively then, we have
\begin{align*}  \E\| \tX^{(t)} &- \tX^\ddagger \|^2
  \leq    \left(\alpha_\tV^t + \theta_\tV \sigma_{\max}^2(\bcirc(\tV)) \sum_{t=0}^t \alpha_\tV^s \gamma(t-s)\right) \|\tX^\ddagger\|_F^2\\
  \le& \left(\alpha_\tV^t + \theta_\tV \sigma_{\max}^2(\bcirc(\tV)) (t+1) \alpha_\tV^{\lfloor t/2 \rfloor} \gamma(\lfloor t/2 \rfloor)\right) \|\tX^\ddagger\|_F^2\\
  \leq& \begin{cases}
      \left(\alpha_\tV^t + \theta_\tV \sigma_{\max}^2(\bcirc(\tV)) (t+1) \alpha_\tV^{\lfloor t/2 \rfloor}\left(\alpha_\tU^{\lfloor t/2 \rfloor}  + \theta_\tU  \frac{\phi_{\max}^{\lfloor t/2 \rfloor} \phi_{\min}}{1 - \phi_{\min}} \sigma_{\max}^2(\bcirc(\tU))\right)\right)\|\tX^\ddagger\|_F^2 & \text{ if } 0 \neq \alpha_\tU \not= \beta_\tU \neq 0 \\
     \left(\alpha_\tV^t + \theta_\tV \sigma_{\max}^2(\bcirc(\tV)) (t+1) \alpha_\tV^{\lfloor t/2 \rfloor}\left(\alpha_\tU^{\lfloor t/2 \rfloor}  + \theta_\tU t \phi_{\max}^{\lfloor t/2 \rfloor}\sigma_{\max}^2(\bcirc(\tU))\right)\right)\|\tX^\ddagger\|_F^2 & \text{ else}.
  \end{cases}
\end{align*}
\end{proof}

\section{Numerical Experiments}\label{sec:numerical_experiments}
In this section, we provide numerical experiments which illustrate our theoretical convergence results for FacTBRK and FacTBREK, and explore the behavior of these methods for scenarios where our theoretical results do not hold.  We additionally explore the effect of different types of block sampling, and compare these methods to those applied to the equivalent matricized systems described in Fact~\ref{fact:equivalent_systems}.  Finally, we explore the application of these methods to video and image deblurring.

\subsection{Synthetic Data}
\label{sec:Synthetic_Data}
In these experiments, we generate synthetic consistent and inconsistent tensor linear systems of various sizes.  In each experiment described below, we generate the tensor linear system according to the process described here to ensure that systems of appropriate sizes will satisfy the assumptions of Theorem~\ref{theorem:factorized}.

\textbf{Consistent system.}  We generate tensors $\tU \in \mathbb{R}^{m \times m_1 \times p}$ and $\tV \in \mathbb{R}^{m_1 \times n \times p}$ with each entry sampled i.i.d.\ from the standard normal distribution.  We generate a solution to our consistent system $\tX_{\text{gen}} \in \mathbb{R}^{n \times l \times p}$ with each entry sampled i.i.d.\ from the standard normal distribution and the measurement tensor for our consistent system $\tY \in \mathbb{R}^{m \times l \times p}$ as $\tY = \tU \tV \tX_{\text{gen}}$.  Our theoretical convergence result, Theorem~\ref{theorem:factorized} tells us that when the outer systems $\tU \tZ = \tY$ has a unique solution $\tZ^\ddagger$ and the inner system $\tV \tX = \tZ^\ddagger$ is consistent, then FacTBRK will converge to the solution of minimal norm, $\tX^\ddagger = \tV^\dagger \tZ^\ddagger$.  Our numerical experiments corroborate this result.

\textbf{Inconsistent system.} We generate tensors $\tU \in \mathbb{R}^{m \times m_1 \times p}$ and $\tV \in \mathbb{R}^{m_1 \times n \times p}$ with each entry sampled i.i.d.\ from the standard normal distribution.  We generate our inconsistent system by first generating $\tX_{\text{gen}} \in \mathbb{R}^{n \times l \times p}$ with each entry sampled i.i.d.\ from the standard normal distribution.  We then generate $\tY = \tU \tV \tX_{\text{gen}} + 10^{-4}\tY^\perp$ where we generate $\tY^\perp$ to satisfy $\langle \tY^{\perp}, \tU \tV \tX_{\text{gen}} \rangle = 0$. To do this, we generate $\tilde{\tY} \in \mathbb{R}^{m \times l \times p}$ with each entry sampled i.i.d.\ from the standard normal distribution, $\tilde{\tX} = \tU^\dagger \tilde{\tY}$, and $\tY^\perp = \tilde{\tY} - \tU \tilde{\tX}.$  Our theoretical convergence result, Theorem~\ref{theorem:factorized} tells us that when the outer least-squares problem  $\|\tU \tZ - \tY\|_F^2$ has a unique minimizer $\tZ^\ddagger$ and the inner system $\tV \tX = \tZ^\ddagger$ is consistent, then FacTBREK will converge to the system of minimal norm, $\tX^\ddagger = \tV^\dagger \tZ^\ddagger$.  Our numerical experiments support this result.

\subsubsection{FacTBRK vs FacTBREK}
In these experiments, we compare the behavior of FacTBRK and FacTBREK to one another on consistent and inconsistent systems.  In every experiment in this section, we generate $\tU \in \mathbb{R}^{40 \times 10 \times 7}, \tV \in \mathbb{R}^{10 \times 5 \times 7}$, and $\tX_{\text{gen}} \in \mathbb{R}^{5 \times 5 \times 7}$.  We generate consistent and inconsistent systems according to the process described above.

\begin{figure}
    \includegraphics[width=0.495\textwidth]{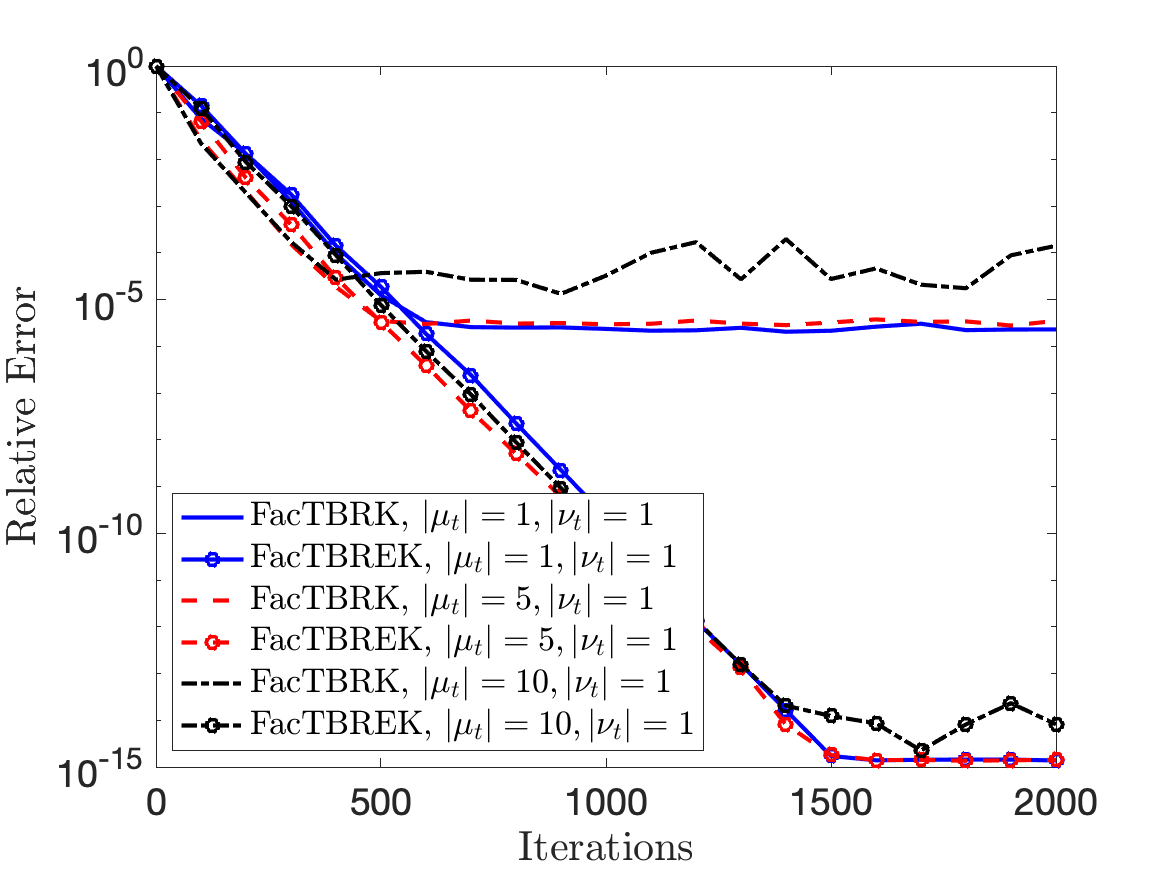}\hfill%
    \includegraphics[width=0.495\textwidth]{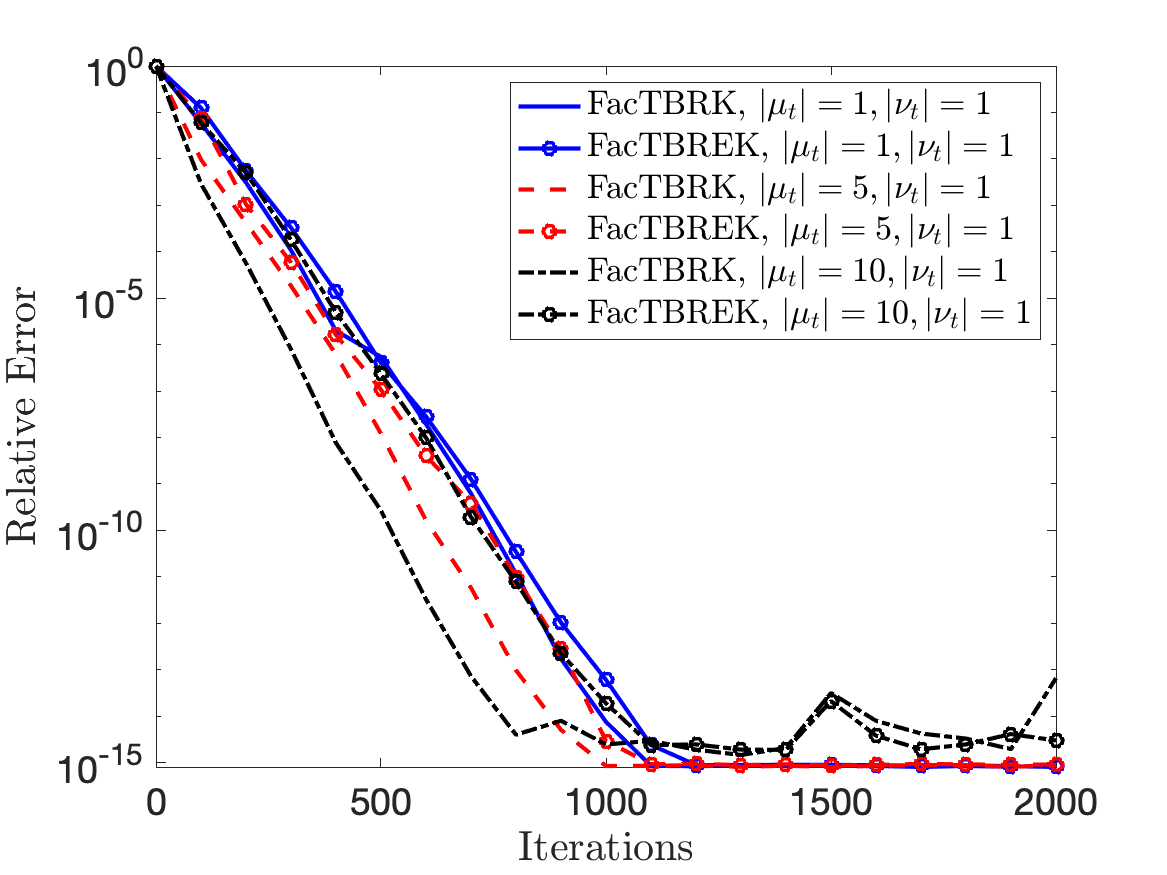}
    \caption{Relative error $\|\tX^{(t)} - \tX^\ddagger\|_F/\|\tX^\ddagger\|_F$ vs iteration $t$ of FacTBRK and FacTBREK with outer block sizes $|\mu_t| \in \{1, 5, 10\}$ and inner system block size $|\nu_t| = 1$ on inconsistent tensor linear system (left) and consistent tensor linear system (right). }\label{fig:facTBRKandfacTBREK_samesystem}
\end{figure}

In our first experiment, we generate an inconsistent system according to the process described above and run FacTBRK (Algorithm~\ref{alg:facTBRK}) and FacTBREK (Algorithm~\ref{alg:facTBREK}) with outer system block sizes $|\mu_t| \in \{1, 5, 10\}$ and inner system block size $|\nu_t| = 1$.  Here, the blocks are uniformly sampled from all possible blocks of this size, $T_{\tU} = \{\mu \in \mathcal{P}([m]) : |\mu| = |\mu_t|\}$ and $T_{\tV} = \{\nu \in \mathcal{P}([m_1]) : |\nu| = |\nu_t|\}$.  The results of this experiment are plotted on the left of Figure~\ref{fig:facTBRKandfacTBREK_samesystem}.  We note that in this experiment, as expected, FacTBRK does not converge on this inconsistent system, but FacTBREK does due to its additional projection step.  Additionally, as expected, larger outer block sizes converge faster, but to a more variable final error due to the higher variability of this outer projection step.

In our second experiment, we generate a consistent system according to the process described above and run FacTBRK (Algorithm~\ref{alg:facTBRK}) and FacTBREK (Algorithm~\ref{alg:facTBREK}) with outer system block sizes $|\mu_t| \in \{1, 5, 10\}$ and inner system block size $|\nu_t| = 1$.  Here, the blocks are uniformly sampled from all possible blocks of this size, $T_{\tU} = \{\mu \in \mathcal{P}([m]) : |\mu| = |\mu_t|\}$ and $T_{\tV} = \{\nu \in \mathcal{P}([m_1]) : |\nu| = |\nu_t|\}$.  The results of this experiment are plotted on the right of Figure~\ref{fig:facTBRKandfacTBREK_samesystem}. We note that in this experiment, as expected, both FacTBRK and FacTBREK converge on this consistent system.  FacTBREK does so more slowly than FacTBRK due to its extra projection step which iteratively learns components of $\tY$ that lie in the range of $\tU$. Again, as expected, larger outer block sizes converge faster, but to a more variable final error due to the higher variability of this outer projection step.

\begin{figure}
    \includegraphics[width=0.495\textwidth]{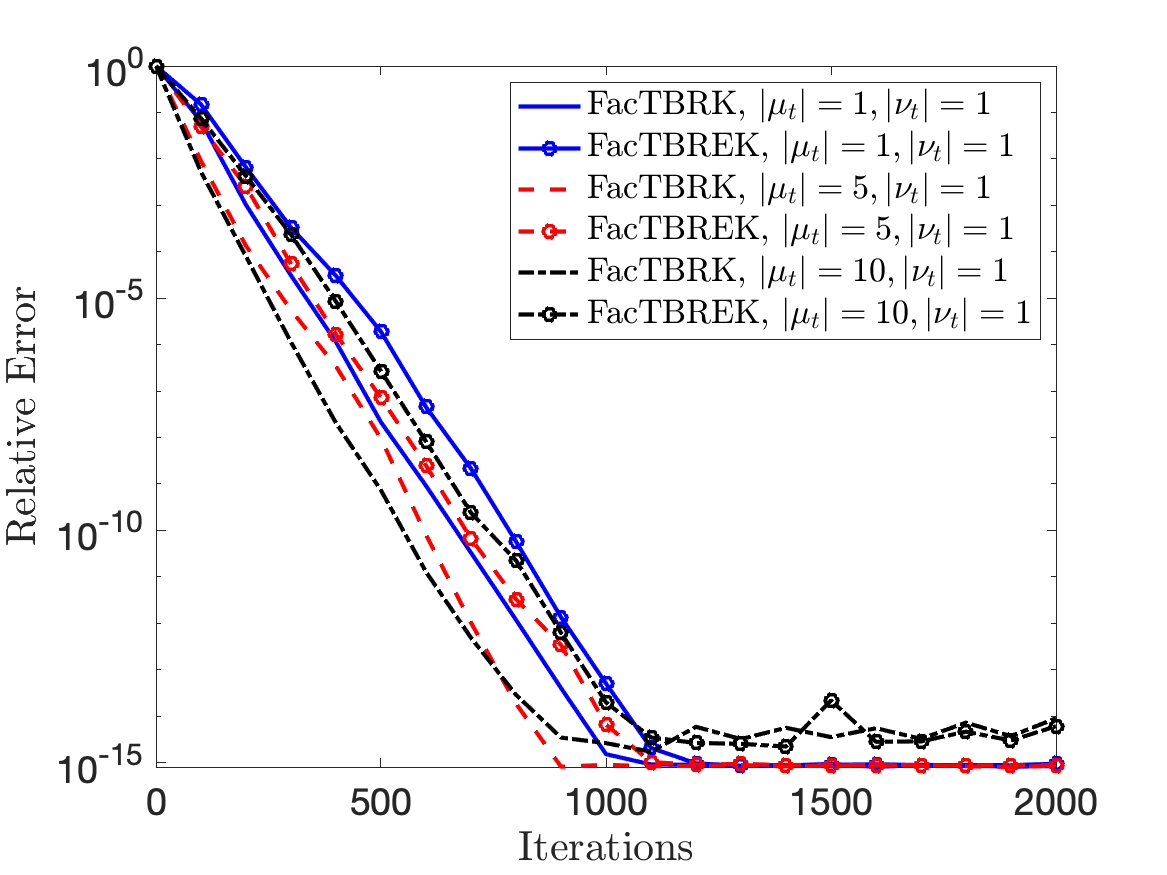}\hfill%
    \includegraphics[width=0.495\textwidth]{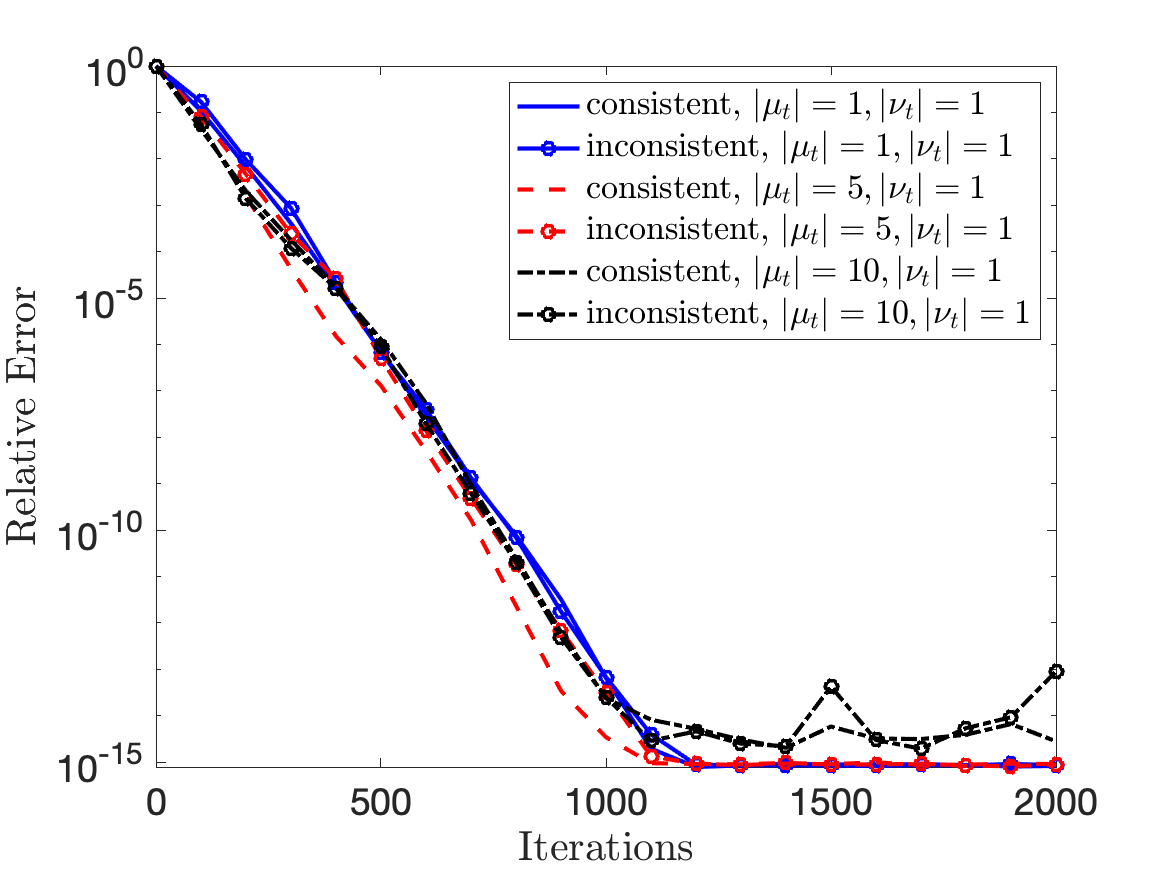}
    \caption{(Left) Relative error $\|\tX^{(t)} - \tX^\ddagger\|_F/\|\tX^\ddagger\|_F$ vs iteration $t$ of FacTBRK on consistent linear system and FacTBREK on inconsistent linear system with outer block sizes $|\mu_t| \in \{1, 5, 10\}$ and inner system block size $|\nu_t| = 1$. (Right) Relative error $\|\tX^{(t)} - \tX^\ddagger\|_F/\|\tX^\ddagger\|_F$ vs iteration $t$ of FacTBREK on consistent and inconsistent linear systems with outer block sizes $|\mu_t| \in \{1, 5, 10\}$ and inner system block size $|\nu_t| = 1$.}\label{fig:facTBRKandfacTBREK_differentsystems}
\end{figure}

In our third experiment, we generate a consistent system according to the process described above and then additionally create a comparable inconsistent system by taking the additional steps of the process described above.  That is, we generate $\tX_{\text{gen}}$ and for the consistent system, build $\tY = \tU \tV \tX_{\text{gen}}$, and for the inconsistent system build $\tY = \tU \tV \tX_{\text{gen}} + 10^{-4} \tY^\perp$.   We run FacTBRK (Algorithm~\ref{alg:facTBRK}) on the consistent system and FacTBREK (Algorithm~\ref{alg:facTBREK}) on the inconsistent system with outer system block sizes $|\mu_t| \in \{1, 5, 10\}$ and inner system block size $|\nu_t| = 1$.  Here, the blocks are uniformly sampled from all possible blocks of this size, $T_{\tU} = \{\mu \in \mathcal{P}([m]) : |\mu| = |\mu_t|\}$ and $T_{\tV} = \{\nu \in \mathcal{P}([m_1]) : |\nu| = |\nu_t|\}$.  The results of this experiment are plotted on the left of Figure~\ref{fig:facTBRKandfacTBREK_differentsystems}.  Due to the high similarity between the consistent and inconsistent systems, we note that the behavior of FacTBRK and FacTBREK are similar to that when run on the same consistent system as in the right plot of Figure~\ref{fig:facTBRKandfacTBREK_samesystem}.  The difference in their convergence is likely mainly due to the extra projection step of FacTBREK which causes slower convergence as the method learns the components of $\tY$ in the range of $\tU$.  Again, as expected, larger outer block sizes converge faster, but to a more variable final error due to the higher variability of this outer projection step.

In our fourth experiment, we generate a consistent system according to the process described above and then additionally create a comparable inconsistent system by taking the additional steps of the process described above.  That is, we generate $\tX_{\text{gen}}$ and for the consistent system, build $\tY = \tU \tV \tX_{\text{gen}}$, and for the inconsistent system build $\tY = \tU \tV \tX_{\text{gen}} + 10^{-4} \tY^\perp$.   We run FacTBREK (Algorithm~\ref{alg:facTBREK}) on both systems with outer system block sizes $|\mu_t| \in \{1, 5, 10\}$ and inner system block size $|\nu_t| = 1$.  Here, the blocks are uniformly sampled from all possible blocks of this size, $T_{\tU} = \{\mu \in \mathcal{P}([m]) : |\mu| = |\mu_t|\}$ and $T_{\tV} = \{\nu \in \mathcal{P}([m_1]) : |\nu| = |\nu_t|\}$.  The results of this experiment are plotted on the right of Figure~\ref{fig:facTBRKandfacTBREK_differentsystems}.  We note that FacTBREK behaves quite similarly on both systems as it is learning (up to randomness due to sampling) the same components of $\tY$ that live in the range of $\tU$.

\subsubsection{Effect of different systems and different block sizes}

In the following examples, we examine the performance of FacTBRK (Algorithm~\ref{alg:facTBRK}) and FacTBREK (Algorithm~\ref{alg:facTBREK}) under two different considerations: (i) whether $\tA$ is under-determined or over-determined, and (ii) the block sizes of the inner and outer systems. To analyze the systems in both the under-determined or over-determined regimes, we also consider the cases when the subsystems containing $\tU$ and $\tV$ are various combinations of under-determined and over-determined, as listed in Table: \ref{tab:facTBRKcases}. Note that some cells in the table are left empty and color-coded black because no combination of the tensors $\tU$ and $\tV$ with the indicated property in the first row can produce a tensor $\tA$ with the desired property in the first column. Also, note that the shaded gray cells indicate cases where our theory shows that convergence is not guaranteed.

\begin{table}[h!]
\caption{Cases of the numerical experiments of FacTBRK (Algorithm~\ref{alg:facTBRK}) and FacTBREK (Algorithm~\ref{alg:facTBREK})algorithms on various sizes of tensors $\tA$, $\tU$ and $\tV$. The dimensions of the $\tens{X}$ were $20\times 10\times 30$, the same for all the cases.  Cases for which Theorem~\ref{theorem:factorized} does not hold are grayed out and these experiments are included in Appendix~\ref{sec:appendix}.  Cases for which the tensors cannot be formed are blacked out.} \label{tab:facTBRKcases}
\small
    \centering
   \begin{tabular}{|c|c|c|c|c|}\hline
   Cases  &\begin{tabular}{@{}c@{}} $\tU$ over-determined \\ $\tV$ under-determined \end{tabular} & \begin{tabular}{@{}c@{}} $\tU$ over-determined \\ $\tV$ over-determined \end{tabular} & \cellcolor{black!25}\begin{tabular}{@{}c@{}} $\tU$ under-determined \\ $\tV$ over-determined \end{tabular} & \cellcolor{black!25}\begin{tabular}{@{}c@{}} $\tU$ under-determined \\ $\tV$ under-determined \end{tabular} \\\hline

1. \begin{tabular}{@{}c@{}} $\tA \in \mathbb{R}^{10\times 20 \times 30}$ \\ under-determined \end{tabular}&
\begin{tabular}{@{}c@{}}$r=5$\\$|\nu_{t}|=[1,3,5]$\\ $|\mu_{t}|=[1,3,5]$ \\ FacTBRK: Figure~\ref{fig:facTBRK_differentsystems_case1_1}\\
FacTBREK: Figure~\ref{fig:facTBREK_differentsystems_case1_1} \end{tabular} & \cellcolor{black}- &
\cellcolor{black!10}\begin{tabular}{@{}c@{}}$r=25$\\$|\nu_{t}|=[1,5,10]$\\ $|\mu_{t}|=[1,5,10]$ \\ FacTBRK: Figure~\ref{fig:facTBRK_differentsystems_case1_2}\\
FacTBREK: Figure~\ref{fig:facTBREK_differentsystems_case1_2}  \end{tabular} &
\cellcolor{black!10}\begin{tabular}{@{}c@{}}$r=15$\\ $|\nu_{t}|=[1,5,15]$\\ $|\mu_{t}|=[1,5,10]$ \\ FacTBRK: Figure~\ref{fig:facTBRK_differentsystems_case1_3}\\
FacTBREK: Figure~\ref{fig:facTBREK_differentsystems_case1_3}
\end{tabular}
\\\hline

2. \begin{tabular}{@{}c@{}} $\tA \in \mathbb{R}^{30\times 20 \times 30}$ \\ over-determined \end{tabular}&
\begin{tabular}{@{}c@{}}$r=15$\\ $|\nu_{t}|=[1,5,10,15]$\\ $|\mu_{t}|=[1,5,10,15]$ \\ FacTBRK: Figure~\ref{fig:facTBRK_differentsystems_case2_1} \\
FacTBREK: Figure~\ref{fig:facTBREK_differentsystems_case2_1} \end{tabular} & \begin{tabular}{@{}c@{}}$r=25$\\$|\nu_{t}|=[1,5,10,20]$\\ $|\mu_{t}|=[1,5,10,20,25]$ \\ FacTBRK: Figure~\ref{fig:facTBRK_differentsystems_case2_4}\\
FacTBREK: Figure~\ref{fig:facTBREK_differentsystems_case2_4} \end{tabular} &
\cellcolor{black!10}\begin{tabular}{@{}c@{}}$r=35$\\$|\nu_{t}|=[1,10,20]$\\ $|\mu_{t}|=[1,10,20]$ \\ FacTBRK: Figure~\ref{fig:facTBRK_differentsystems_case2_2}\\
FacTBREK: Figure~\ref{fig:facTBREK_differentsystems_case2_2} \end{tabular}
& \cellcolor{black}-
\\\hline

\end{tabular}
\end{table}

Different system sizes were chosen to illustrate the convergence results in the six possible scenarios described in Table: \ref{tab:facTBRKcases}. The convergence of FacTBRK (Algorithm~\ref{alg:facTBRK}) and FacTBREK (Algorithm~\ref{alg:facTBREK})algorithms is studied in terms of relative least norm error, and the results are generated according to the process described above in section \ref{sec:Synthetic_Data}.

\begin{figure}[ht]
    \includegraphics[width=0.495\textwidth]{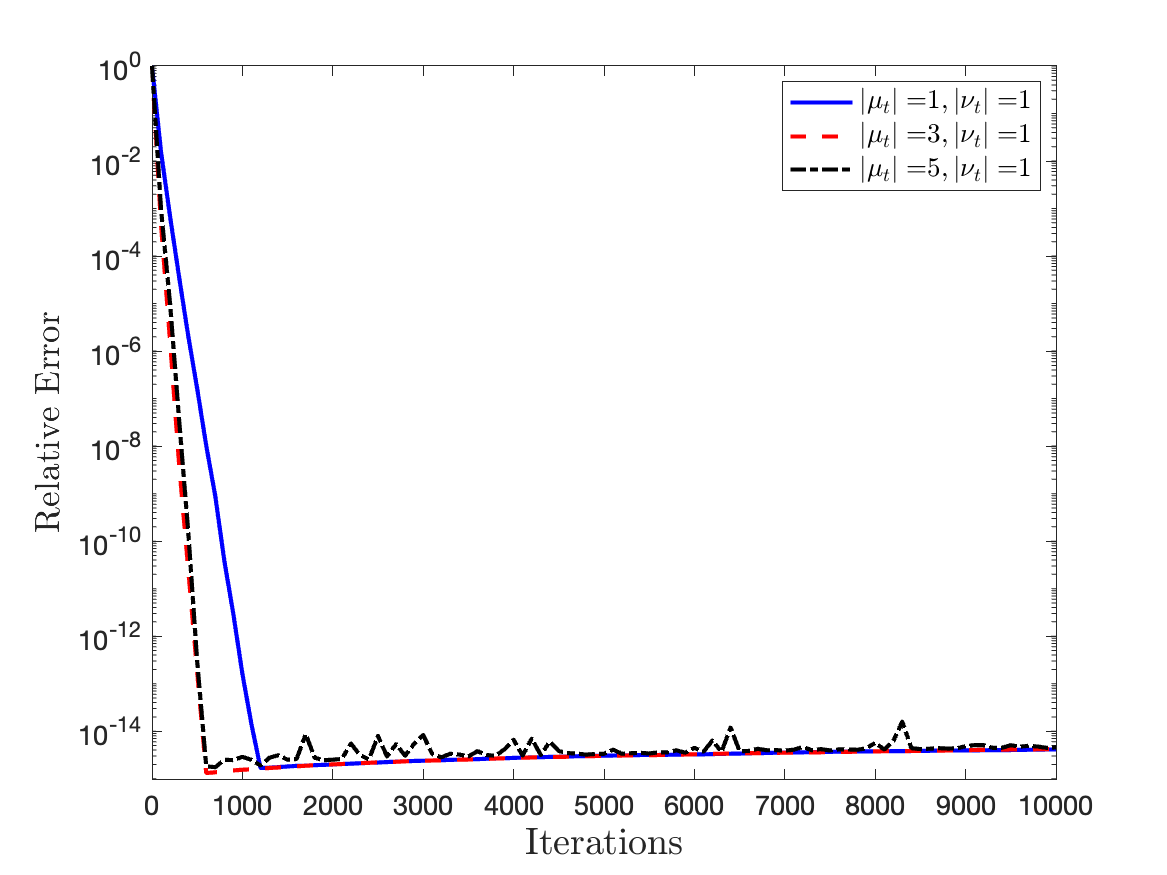}\hfill%
    \includegraphics[width=0.495\textwidth]{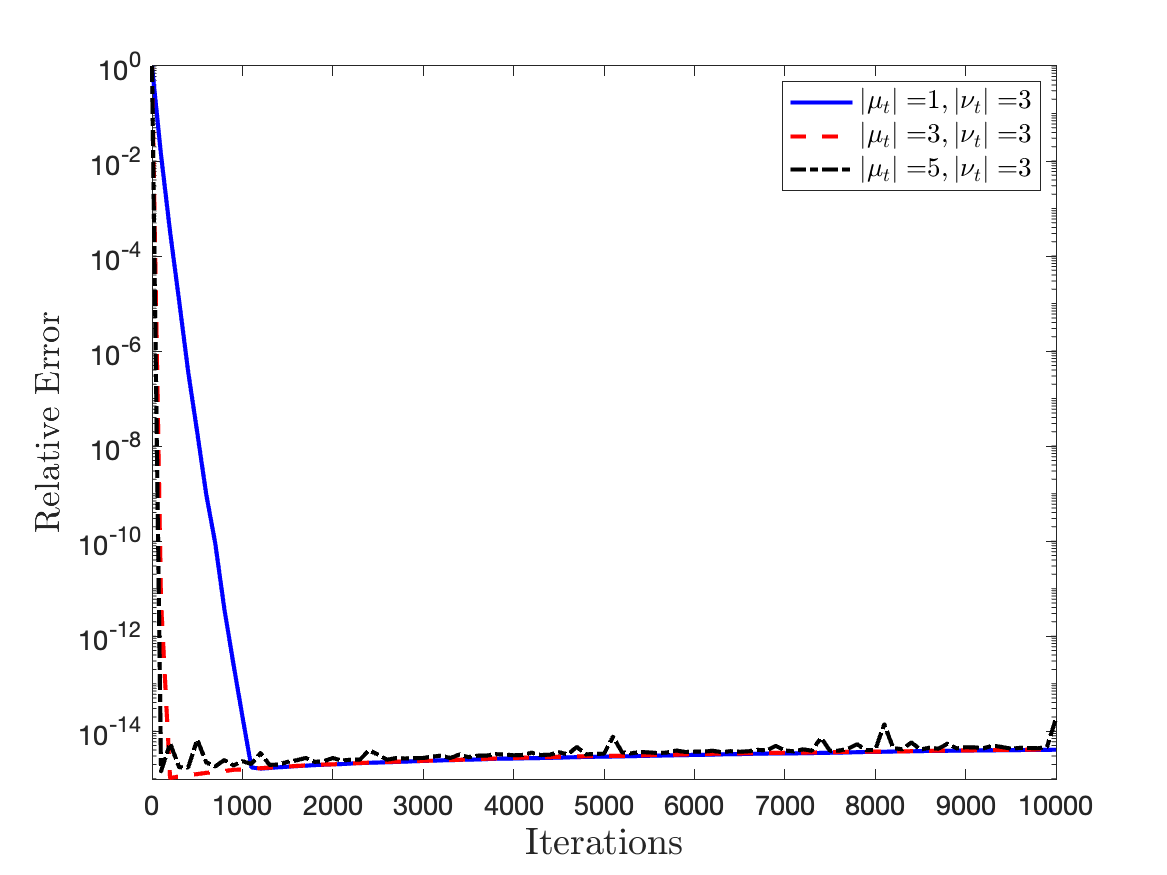}\hfill%
    \includegraphics[width=0.495\textwidth]{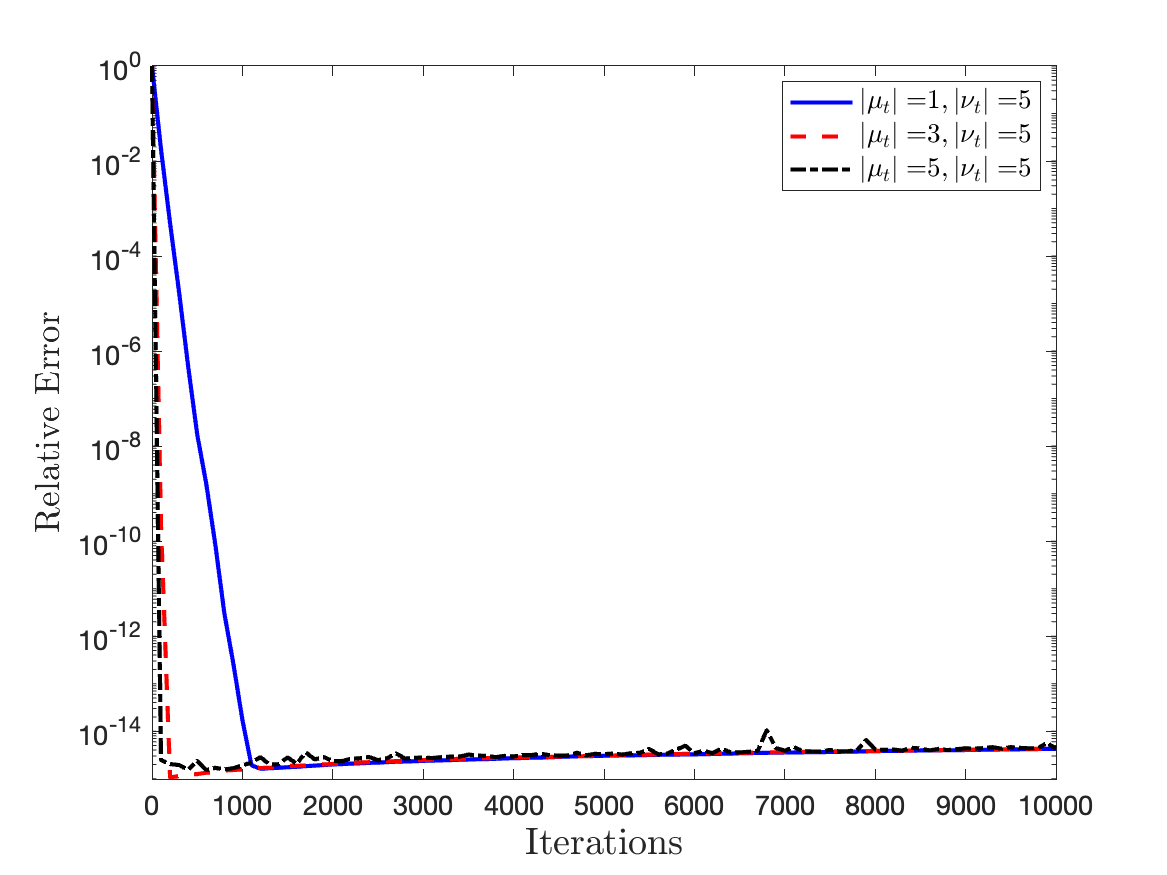}%

    \caption{Relative error $\|\tX^{(t)} - \tX^\ddagger\|_F/\|\tX^\ddagger\|_F$ vs iteration $t$ of FacTBRK on consistent linear system when $\tA$ is under-determined, $\tU$ is over-determined and $\tV$ is under-determined. We consider outer block sizes $|\mu_t| \in \{1, 3, 5\}$ and inner system block size (Upper Left) $|\nu_t| = 1$, (Upper Right) $|\nu_t| = 3$, (Lower) $|\nu_t| = 5$. }\label{fig:facTBRK_differentsystems_case1_1}
\end{figure}

In Figures (\ref{fig:facTBRK_differentsystems_case1_1}, \ref{fig:facTBRK_differentsystems_case2_1} and \ref{fig:facTBRK_differentsystems_case2_4} for FacTBRK, and \ref{fig:facTBREK_differentsystems_case1_1}, \ref{fig:facTBREK_differentsystems_case2_1} and \ref{fig:facTBREK_differentsystems_case2_4} for FacTBREK), we consider the cases when the subsystem with $\tU$ is over-determined, and the subsystem with $\tV$ is either under-determined or over-determined which makes system $\tA$ either under-determined or over-determined as well.
For any given combination of block sizes, both FacTBRK and FacTBREK converge rapidly to the least norm solution in all of the cases. These results are supported by our convergence theory. However, it is notable that, in each case, the convergence speed increases as the outer system block size increases, regardless of the inner block size, and also as the ratio of the inner block size to the outer block size increases the speed of convergence decreases. Also, it should be noted that if the size of the inner blocks is close to the total number of rows of $\tV$ the convergence curves tend to vary more. This happens because the solution of the inner system is too precise, which creates more uncertainty for the outer system.

Notice that when the tensors $\tX,\tU$, and $\tV$ do not meet the assumptions of our main result, our algorithm diverges, as illustrated in Figures \ref{fig:facTBRK_differentsystems_case1_2}, \ref{fig:facTBRK_differentsystems_case1_3},
\ref{fig:facTBRK_differentsystems_case2_2},
\ref{fig:facTBREK_differentsystems_case1_2},
\ref{fig:facTBREK_differentsystems_case1_3},
\ref{fig:facTBREK_differentsystems_case2_2} in Appendix \ref{sec:appendix}. This highlights the scope of our convergence result, as empirical observations suggest divergence in cases beyond the assumptions of our theorem.

\begin{figure}
    \includegraphics[width=0.495\textwidth]{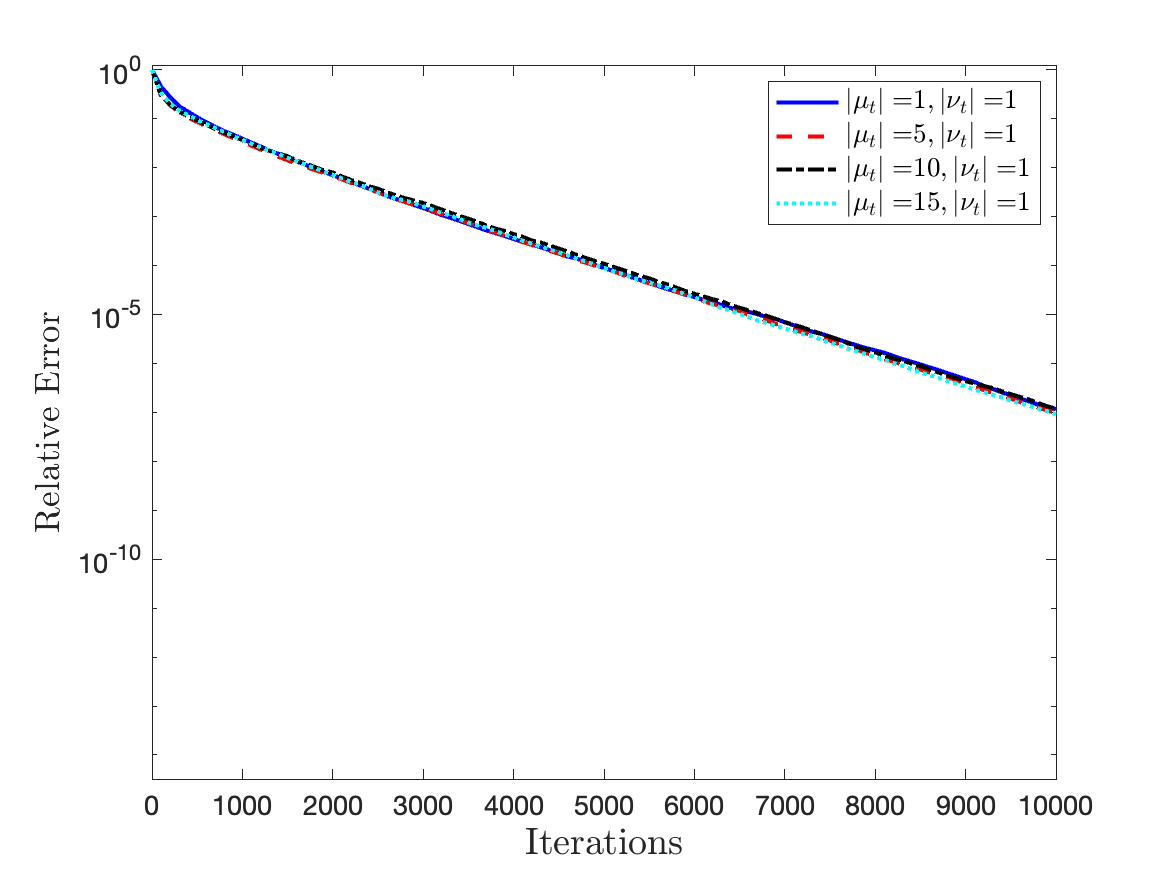}\hfill%
    \includegraphics[width=0.495\textwidth]{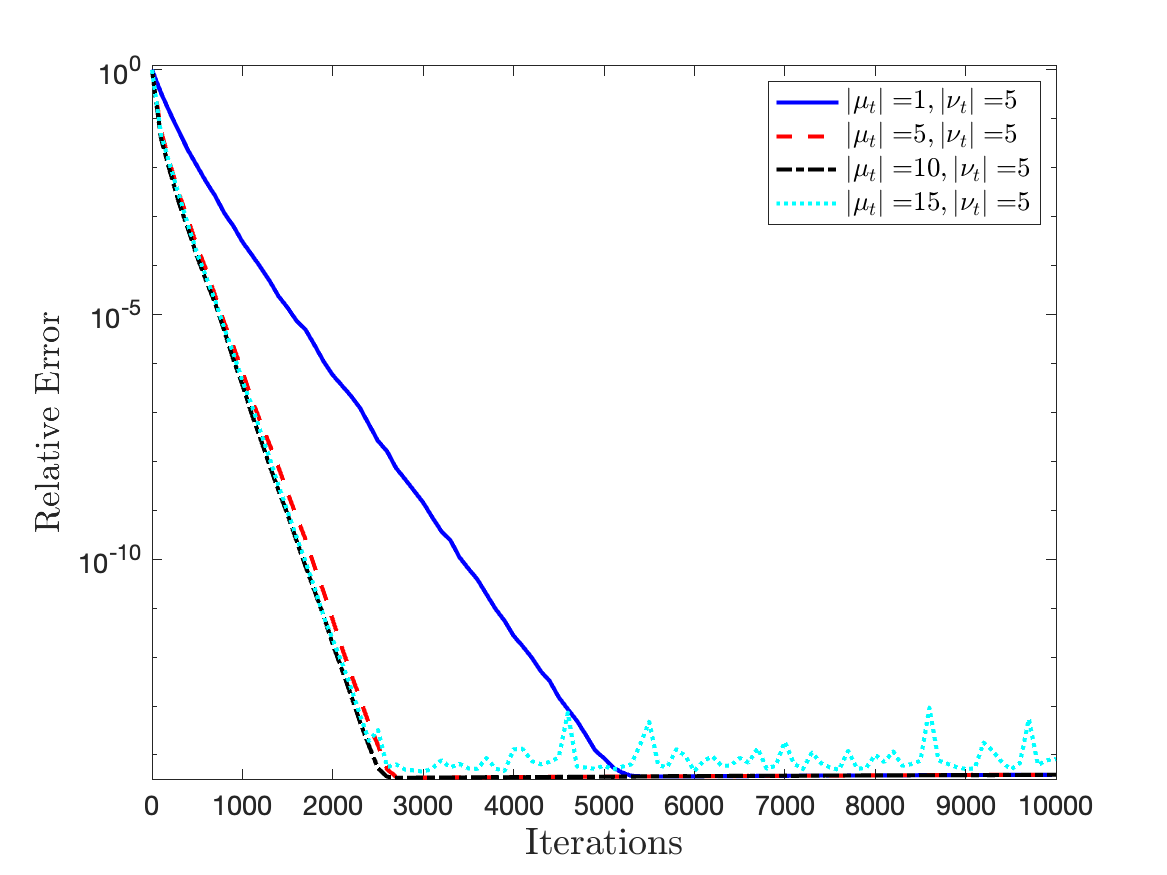}\hfill%
    \includegraphics[width=0.495\textwidth]{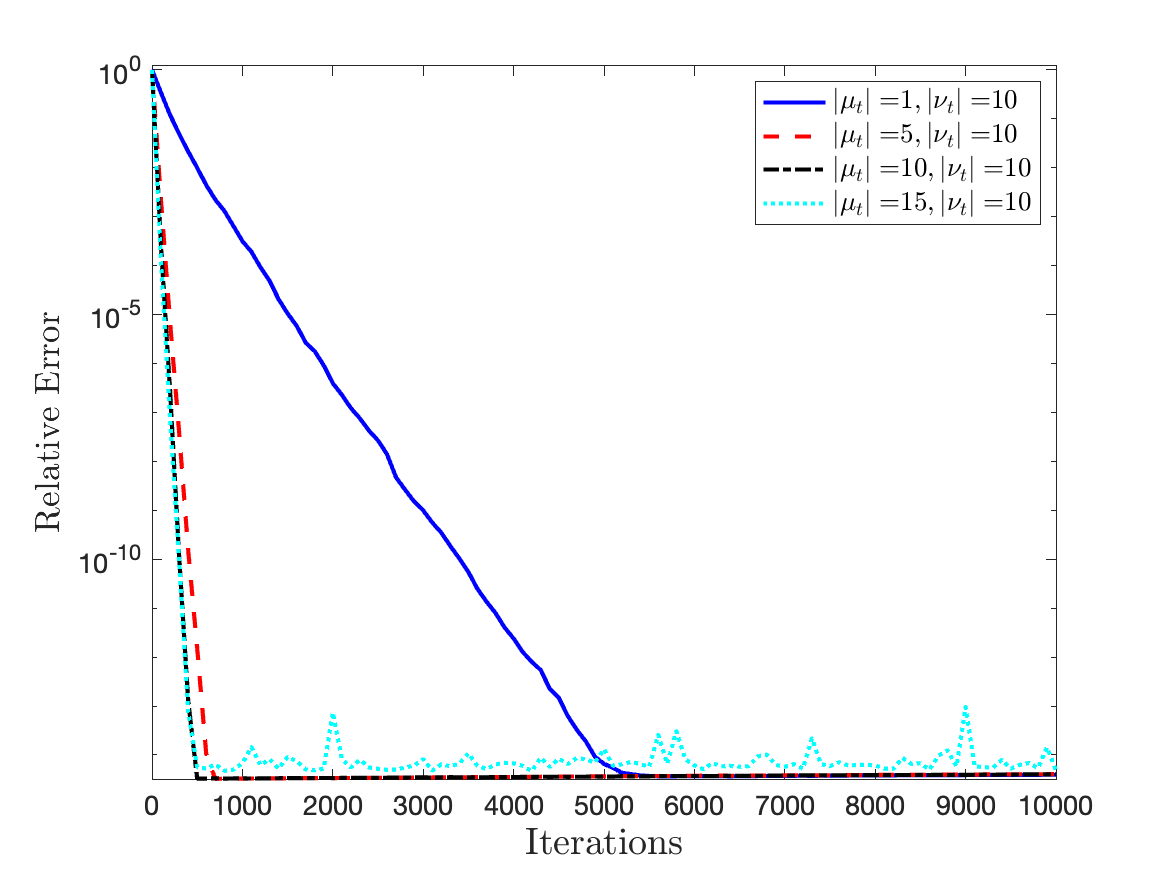}\hfill%
    \includegraphics[width=0.495\textwidth]{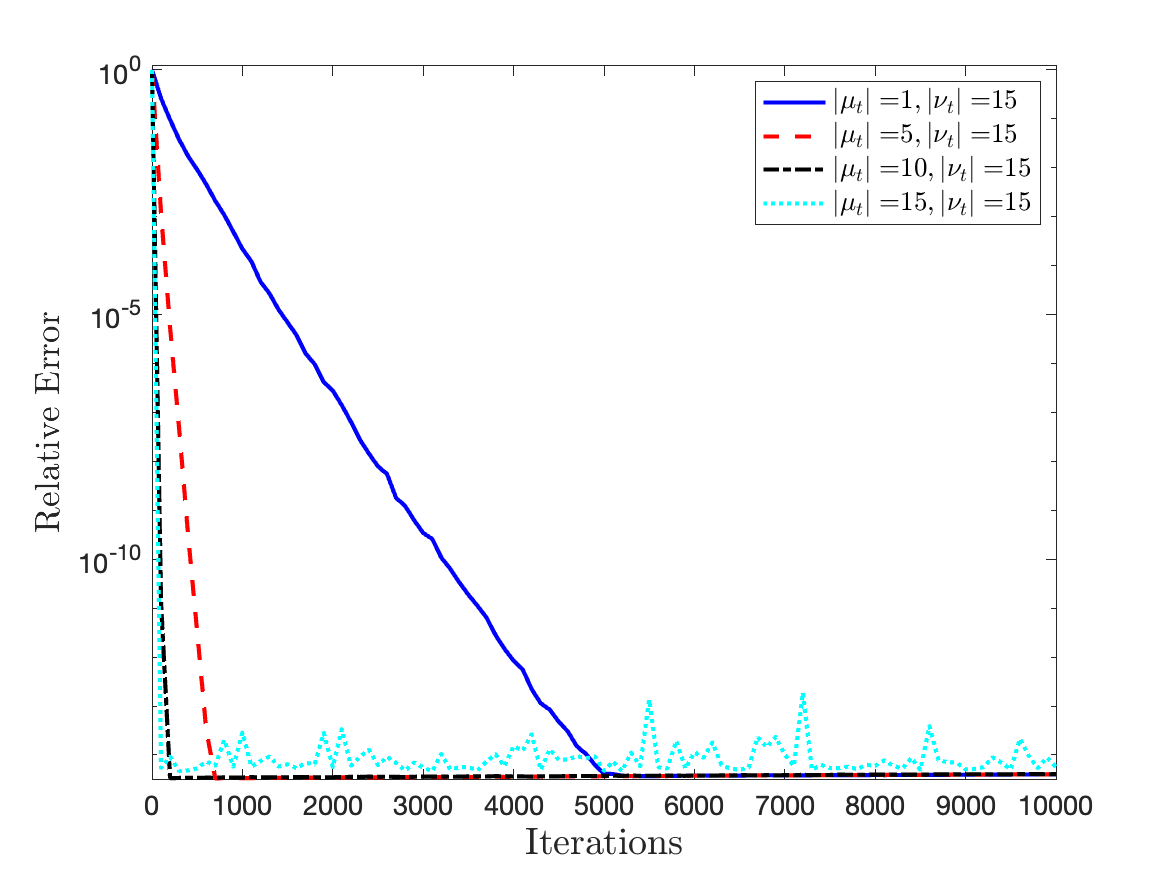}
    \caption{Relative error $\|\tX^{(t)} - \tX^\ddagger\|_F/\|\tX^\ddagger\|_F$ vs iteration $t$ of FacTBRK on consistent linear system when $\tA$ is over-determined, $\tU$ is over-determined and $\tV$ is under-determined. We consider outer block sizes $|\mu_t| \in \{1, 5, 10, 15\}$ and inner system block size (Upper Left) $|\nu_t| = 1$, (Upper Right) $|\nu_t| = 5$, (Lower Left) $|\nu_t| = 10$, (Lower Right) $|\nu_t| = 15$. }\label{fig:facTBRK_differentsystems_case2_1}
\end{figure}

\begin{figure}
    \includegraphics[width=0.495\textwidth]{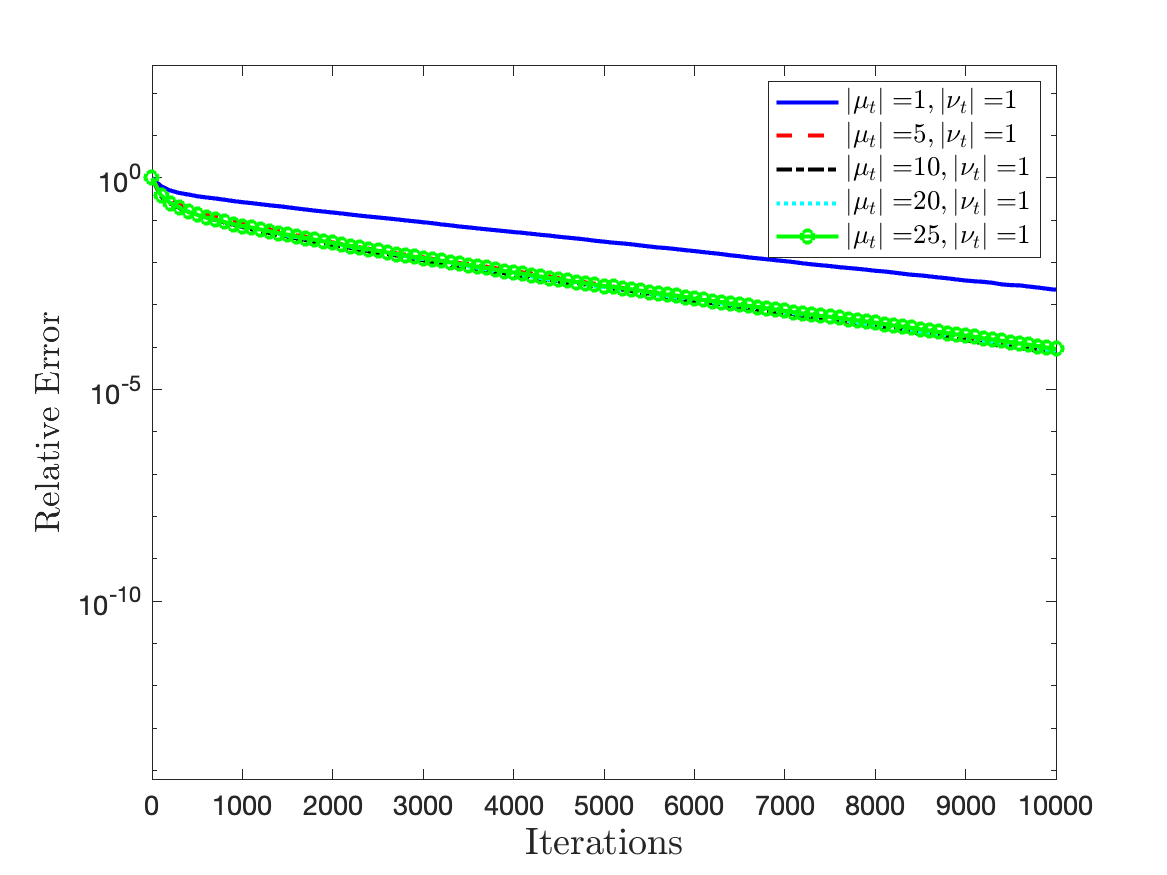}\hfill%
    \includegraphics[width=0.495\textwidth]{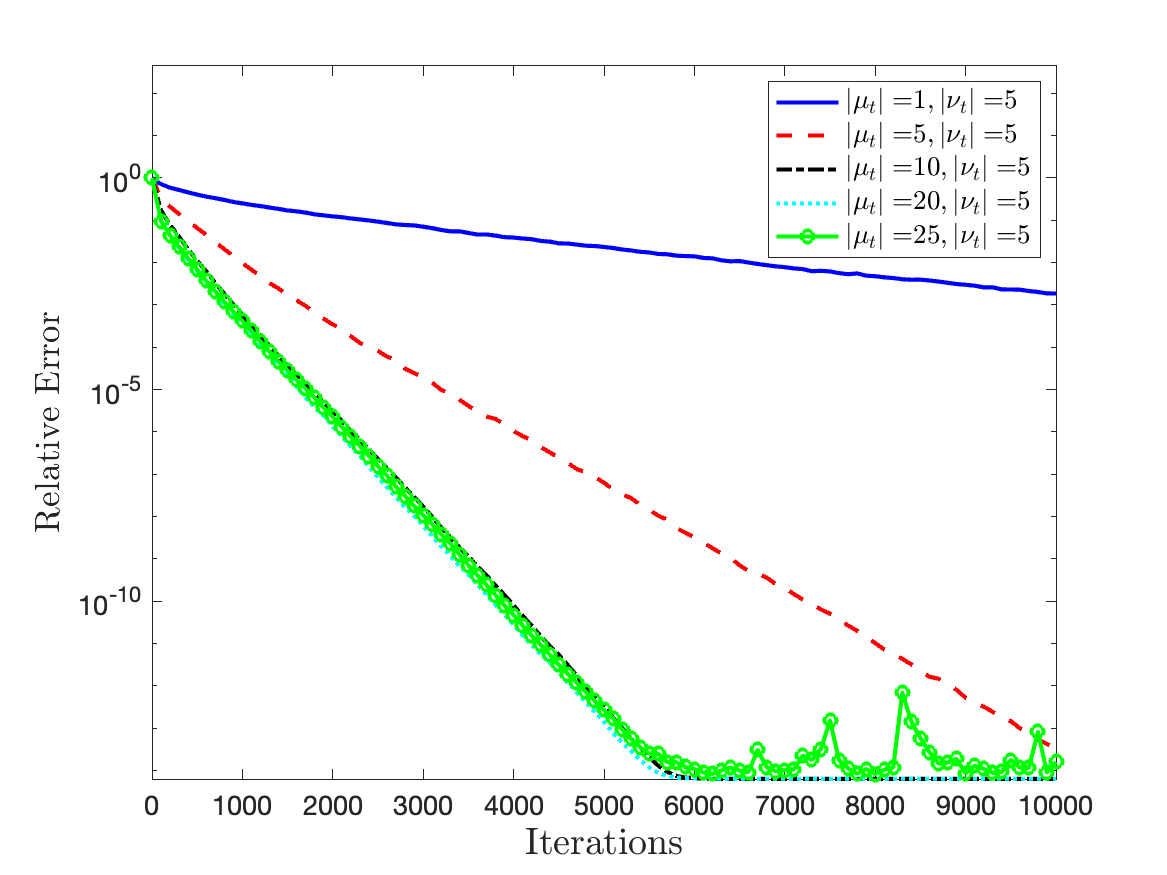}\hfill%
   \includegraphics[width=0.495\textwidth]{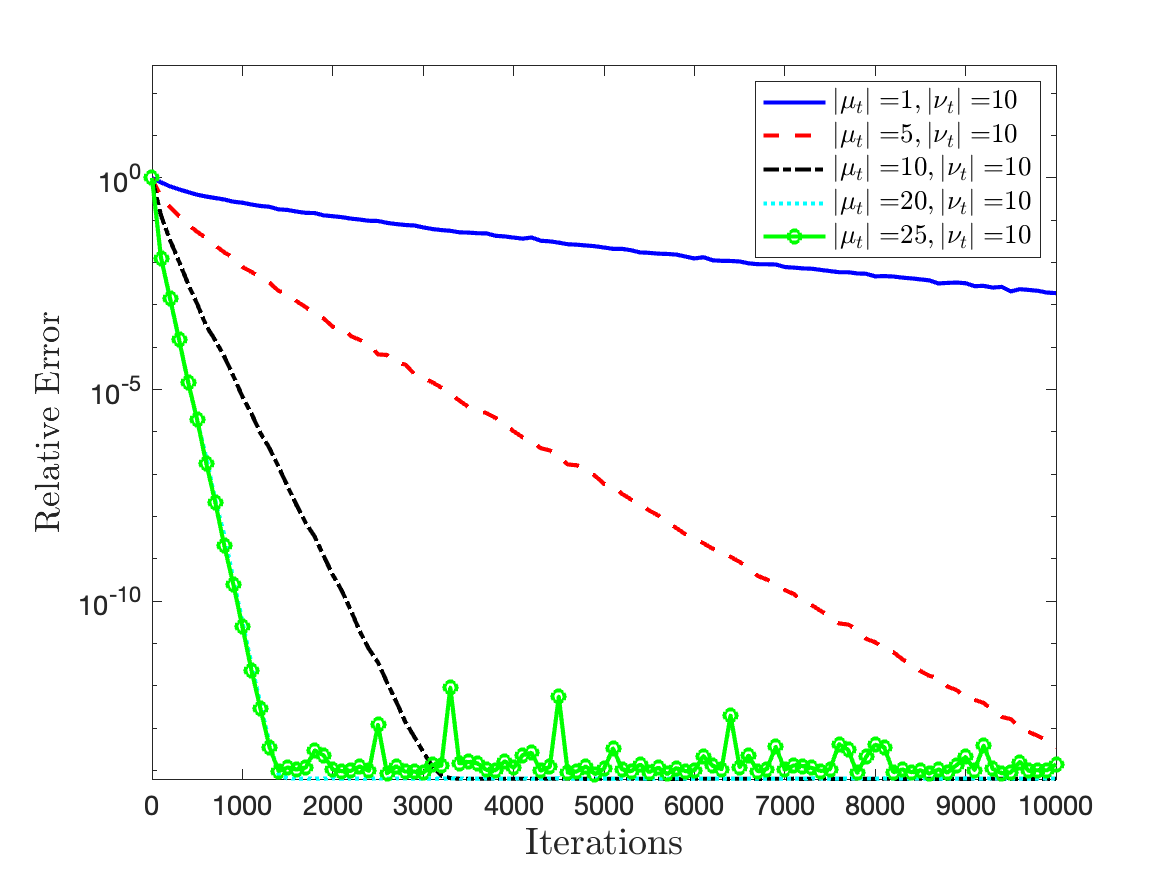}\hfill%
   \includegraphics[width=0.495\textwidth]{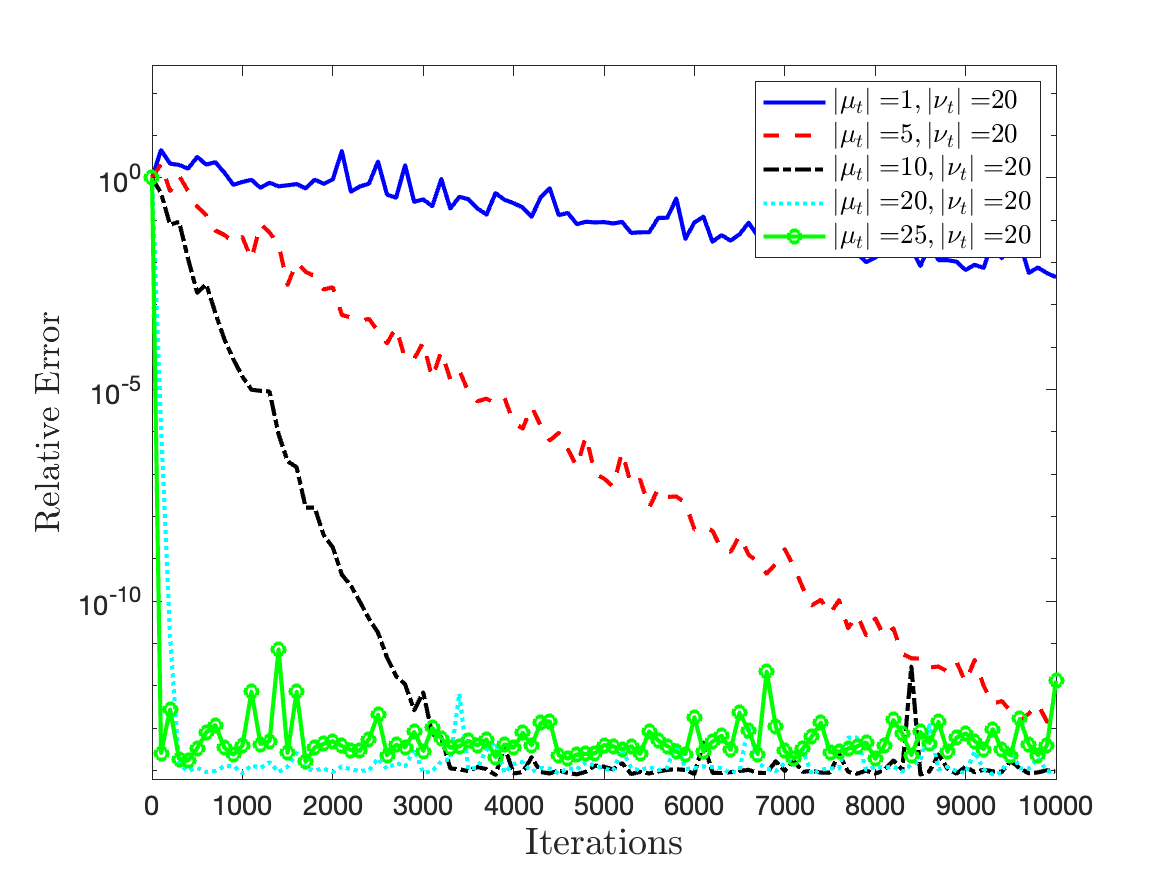}
    \caption{Relative error $\|\tX^{(t)} - \tX^\ddagger\|_F/\|\tX^\ddagger\|_F$ vs iteration $t$ of FacTBRK on consistent linear system when $\tA$ is over-determined and both $\tU$ and $\tV$ are over-determined. We consider outer block sizes $|\mu_t| \in \{1, 5, 10, 20, 25\}$ and inner system block size (Upper Left) $|\nu_t| = 1$, (Upper Right) $|\nu_t| = 5$, (Lower Left) $|\nu_t| = 10$, (Lower Right) $|\nu_t| = 20$. }\label{fig:facTBRK_differentsystems_case2_4}
\end{figure}

\begin{figure}[h]
    \includegraphics[width=0.495\textwidth]{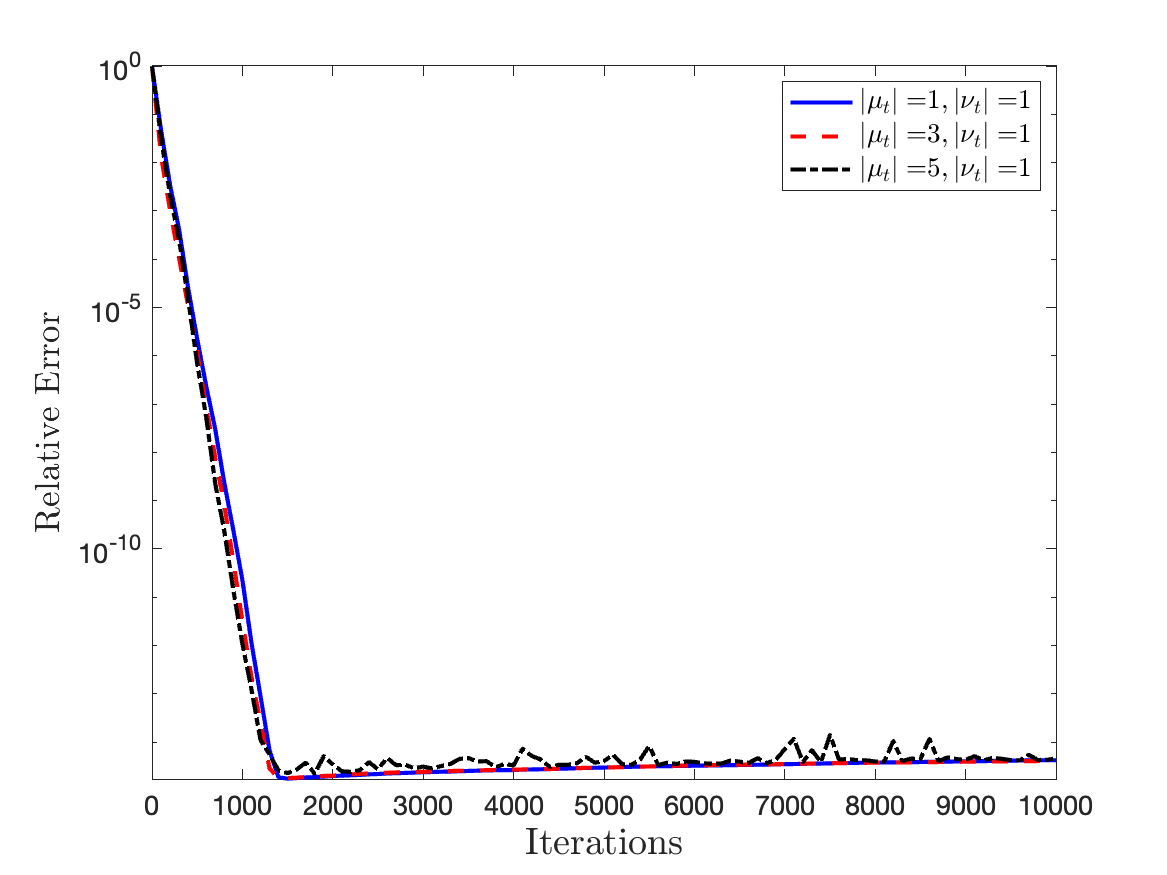}\hfill%
    \includegraphics[width=0.495\textwidth]{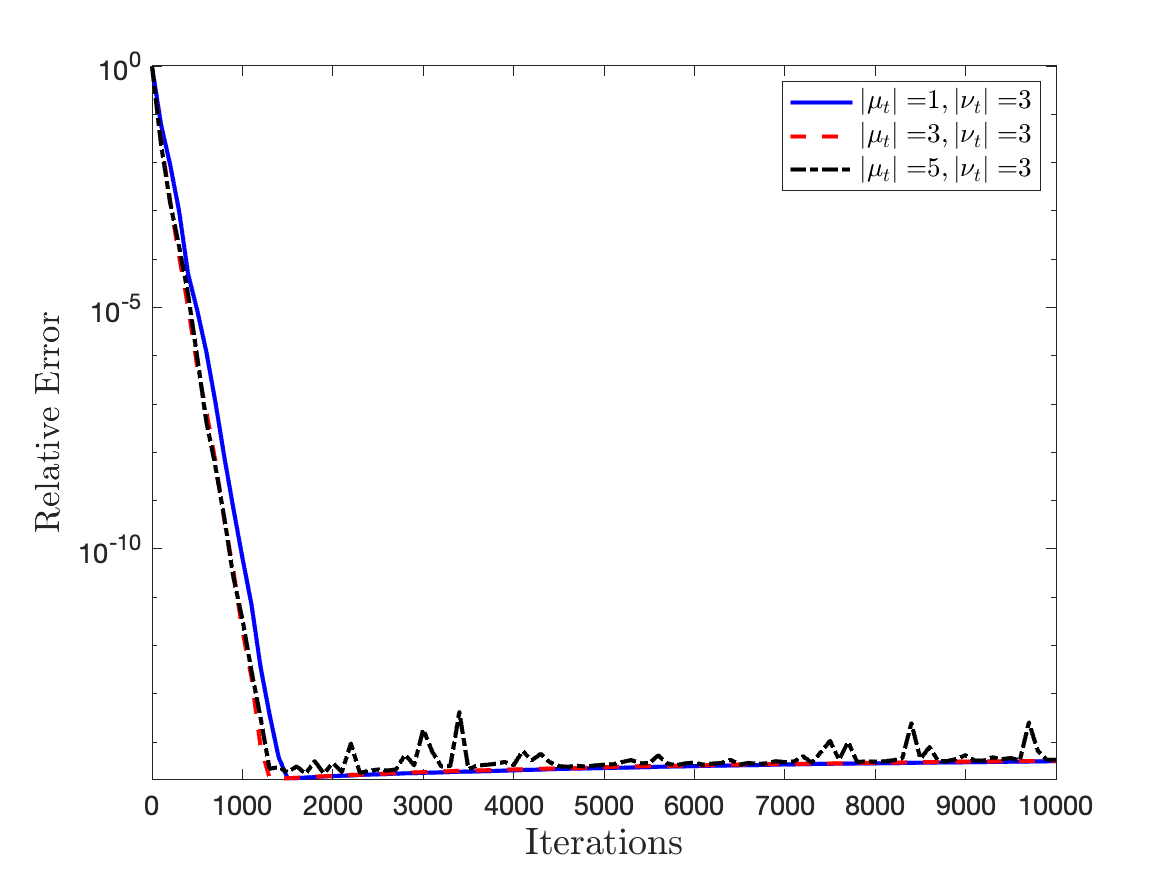}\hfill%
    \includegraphics[width=0.495\textwidth]{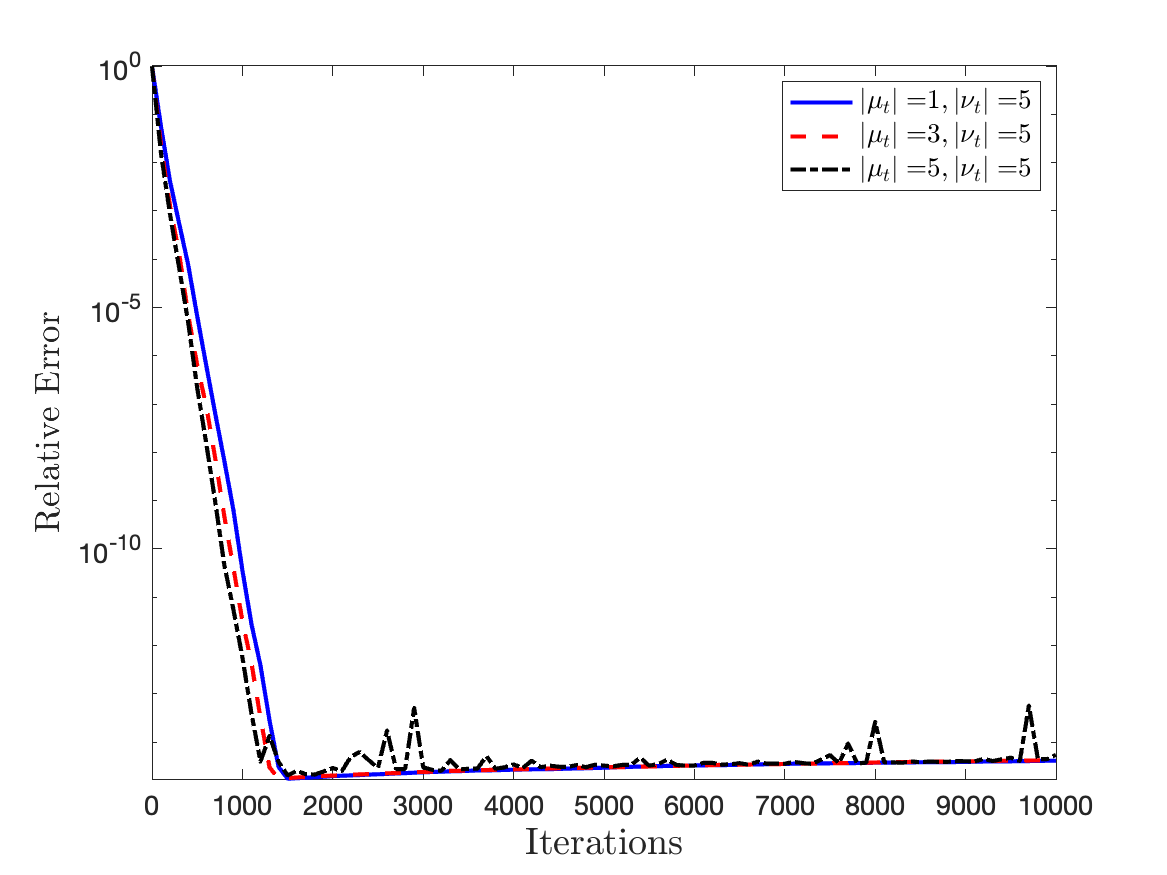}\hfill%

    \caption{Relative error $\|\tX^{(t)} - \tX^\ddagger\|_F/\|\tX^\ddagger\|_F$ vs iteration $t$ of FacTBREK on consistent linear system when $\tA$ is under-determined, $\tU$ is over-determined and $\tV$ is under-determined. We consider outer block sizes $|\mu_t| \in \{1, 3, 5\}$ and inner system block size (Upper Left) $|\nu_t| = 1$, (Upper Right) $|\nu_t| = 3$, (Lower) $|\nu_t| = 5$. }\label{fig:facTBREK_differentsystems_case1_1}
\end{figure}

\begin{figure}[h]
    \includegraphics[width=0.495\textwidth]{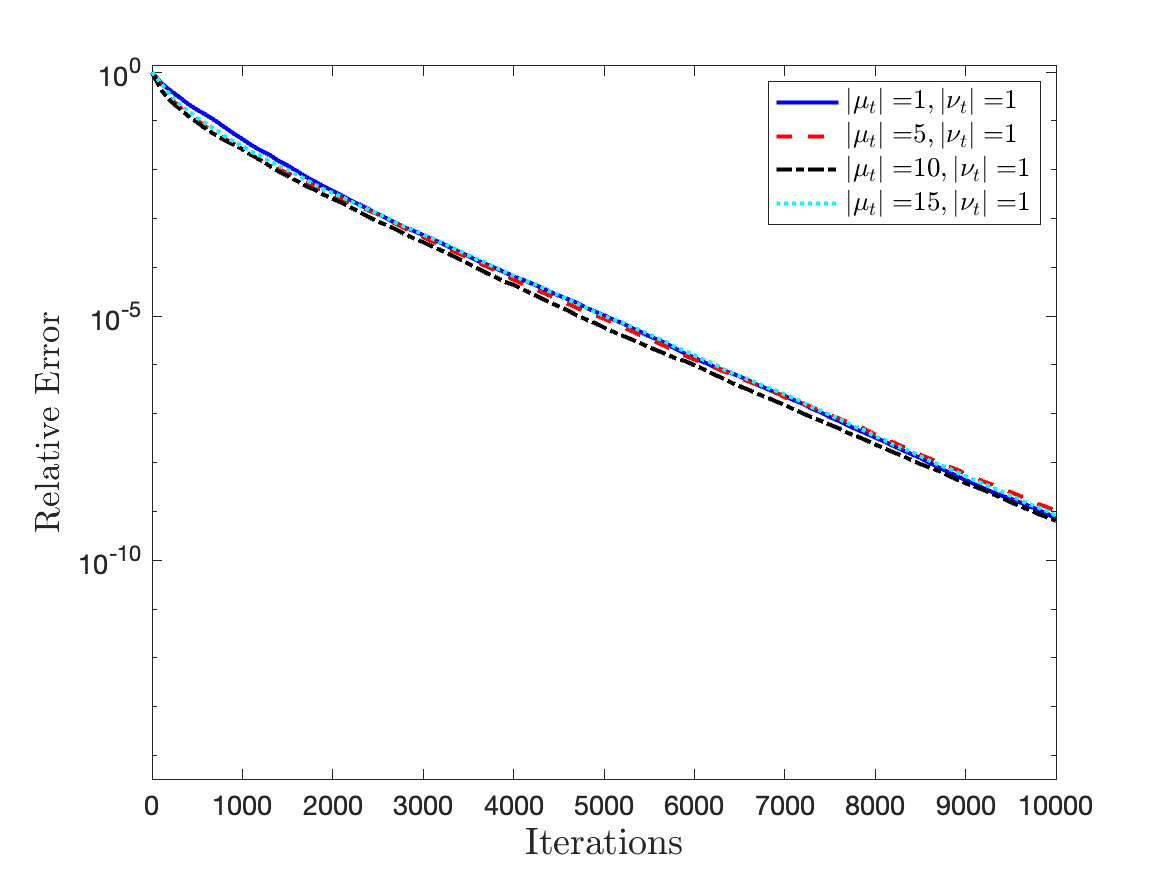}\hfill%
    \includegraphics[width=0.495\textwidth]{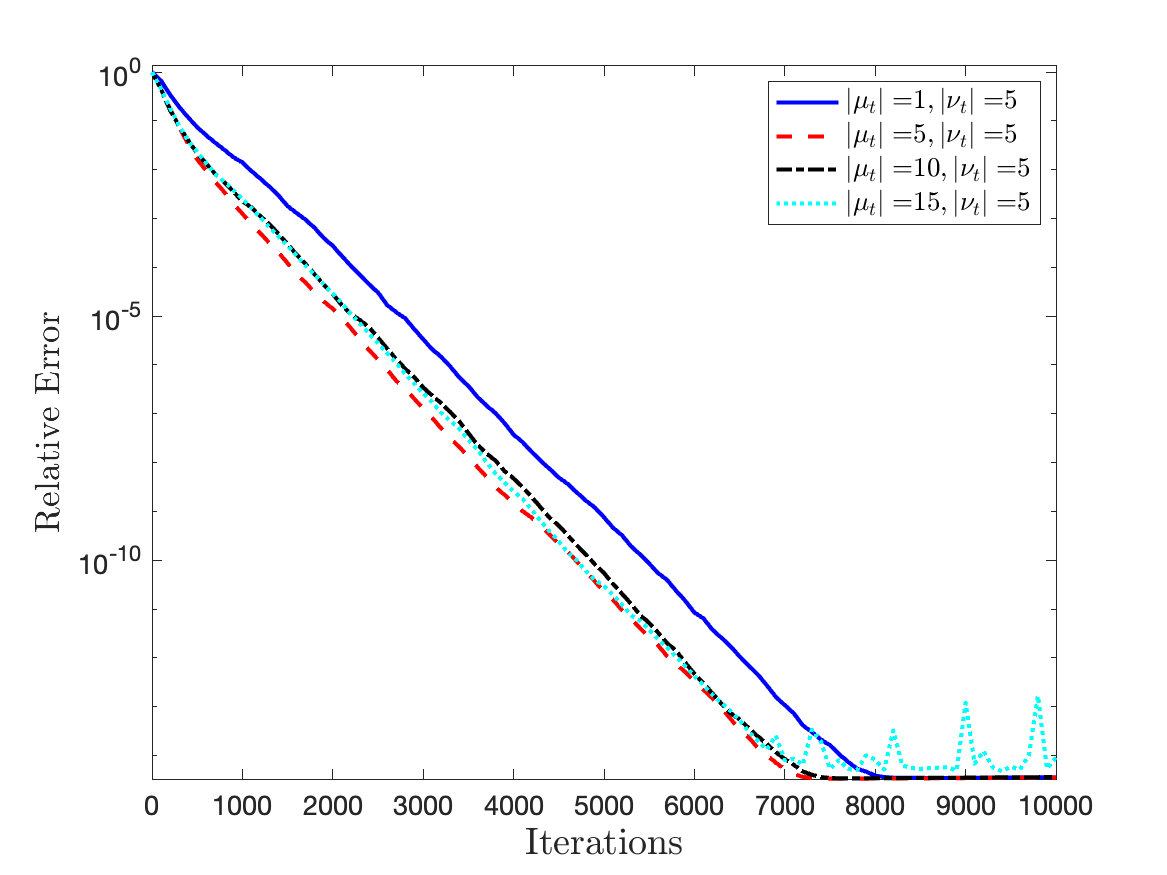}\hfill%
    \includegraphics[width=0.495\textwidth]{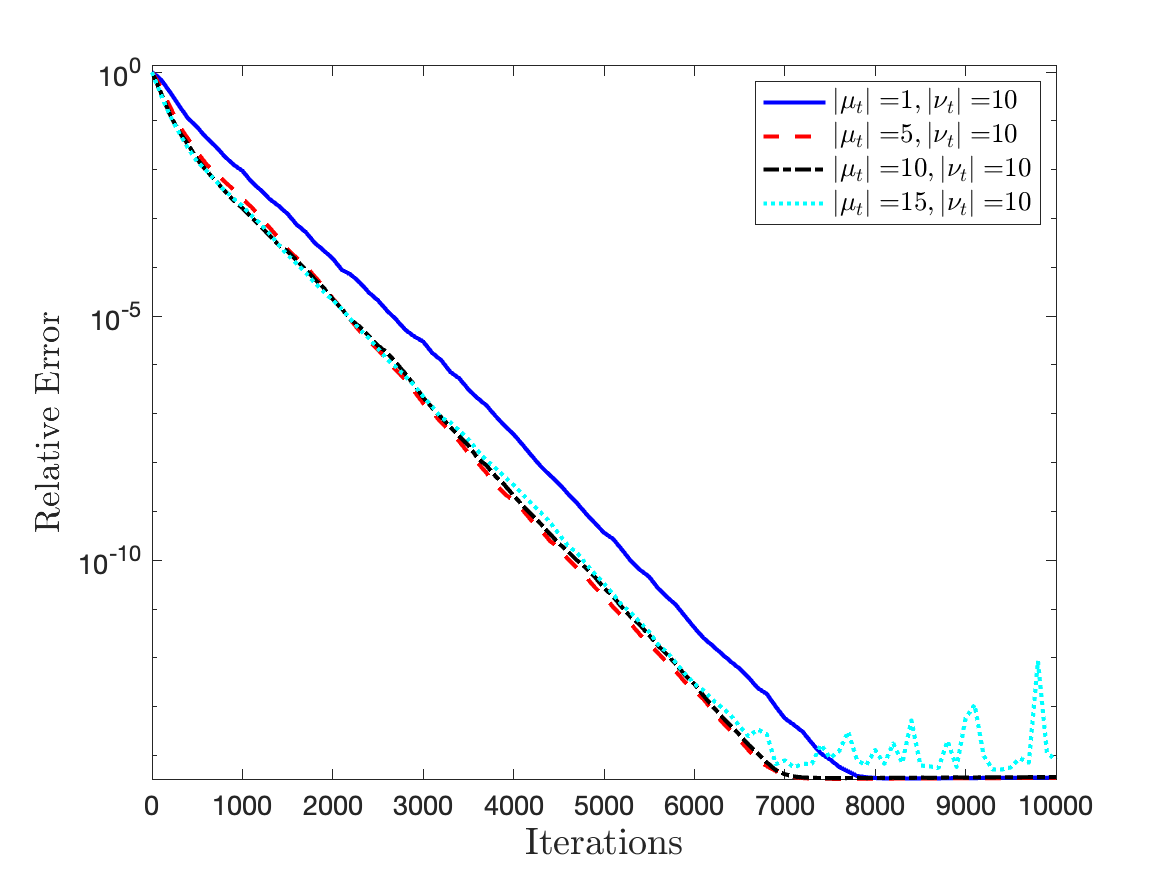}\hfill%
    \includegraphics[width=0.495\textwidth]{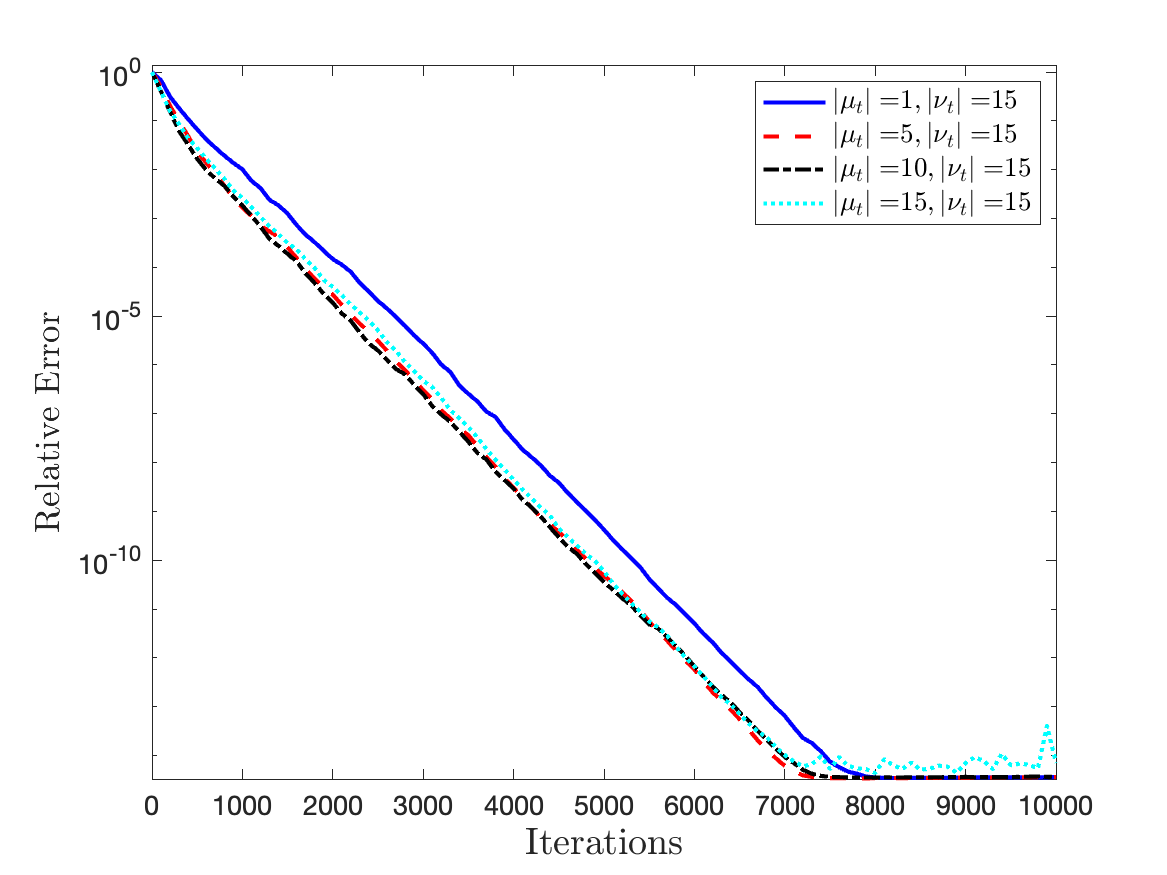}
    \caption{Relative error $\|\tX^{(t)} - \tX^\ddagger\|_F/\|\tX^\ddagger\|_F$ vs iteration $t$ of FacTBREK on consistent linear system when $\tA$ is over-determined, $\tU$ is over-determined and $\tV$ is under-determined. We consider outer block sizes $|\mu_t| \in \{1, 5, 10, 15\}$ and inner system block size (Upper Left) $|\nu_t| = 1$, (Upper Right) $|\nu_t| = 5$, (Lower Left) $|\nu_t| = 10$, (Lower Right) $|\nu_t| = 15$. }\label{fig:facTBREK_differentsystems_case2_1}
\end{figure}

\begin{figure}[h]
    \includegraphics[width=0.495\textwidth]{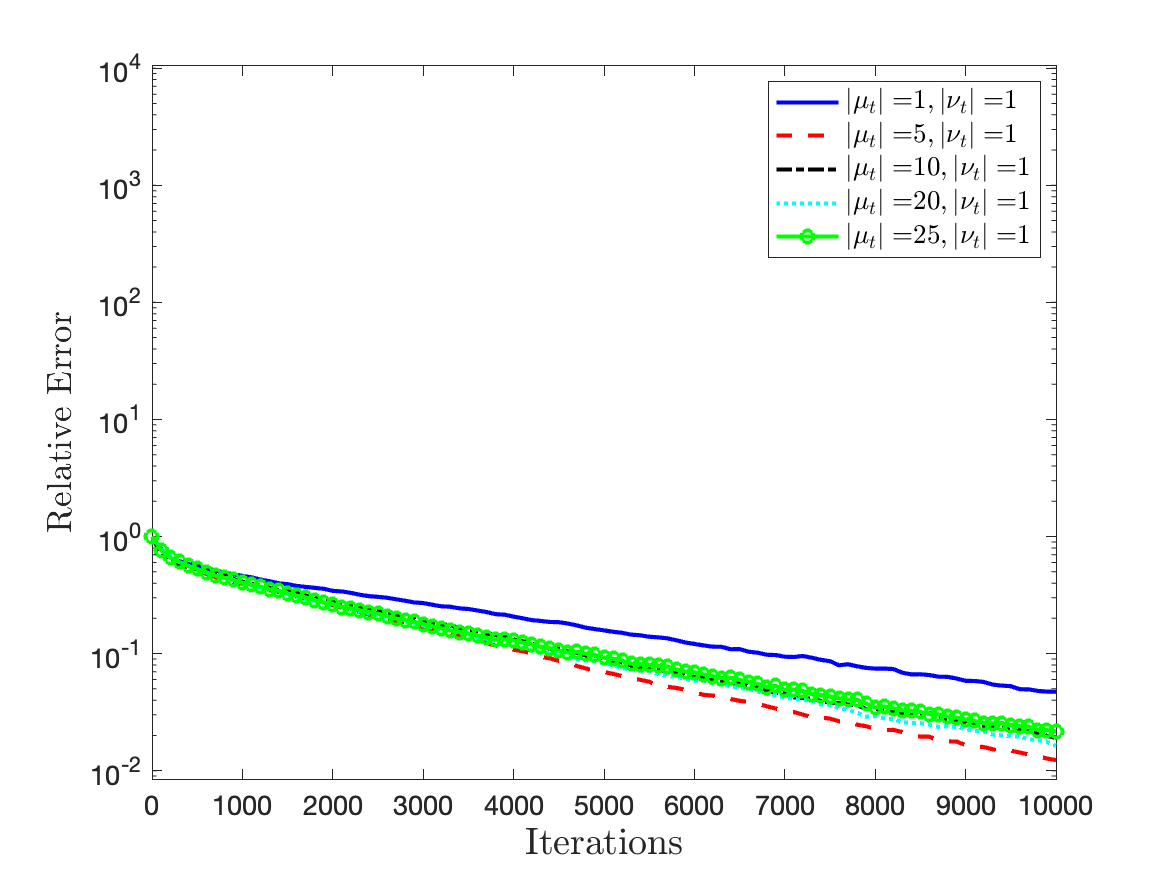}\hfill%
    \includegraphics[width=0.495\textwidth]{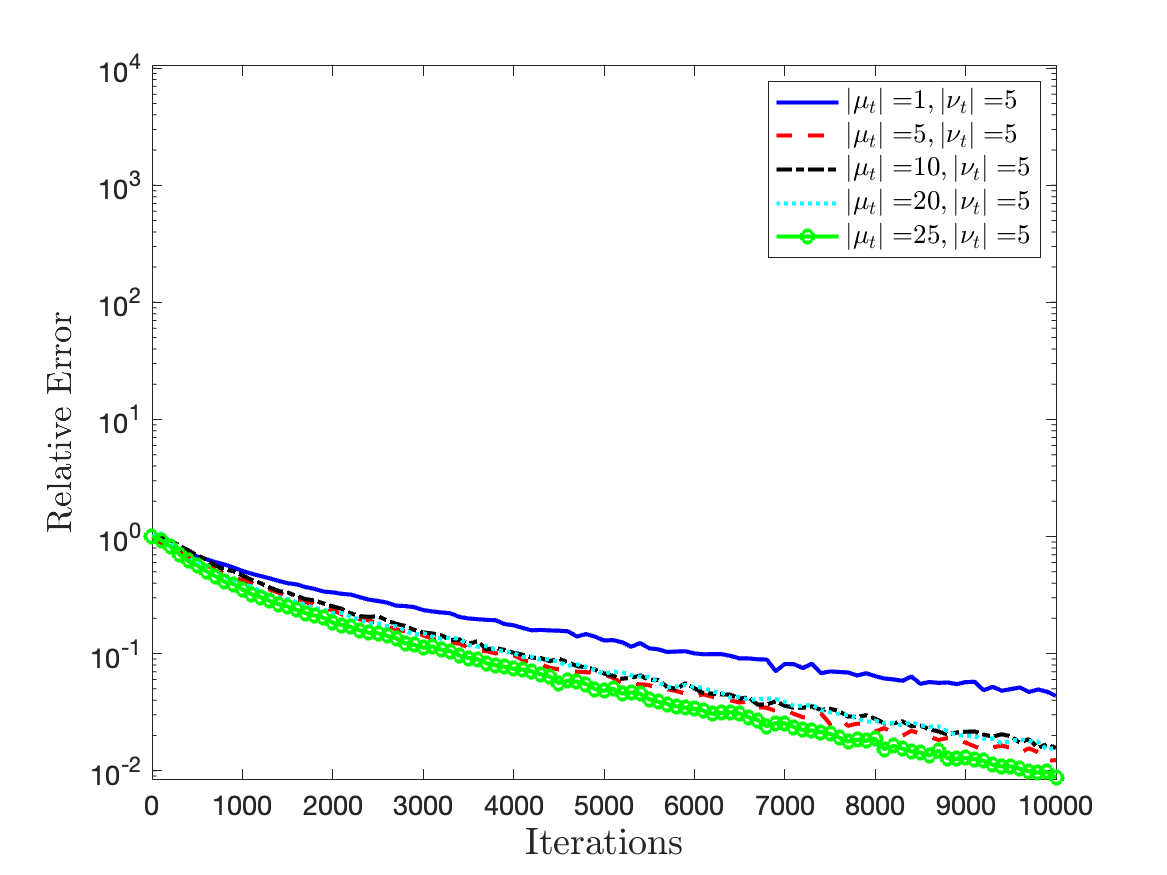}\hfill%
    \includegraphics[width=0.495\textwidth]{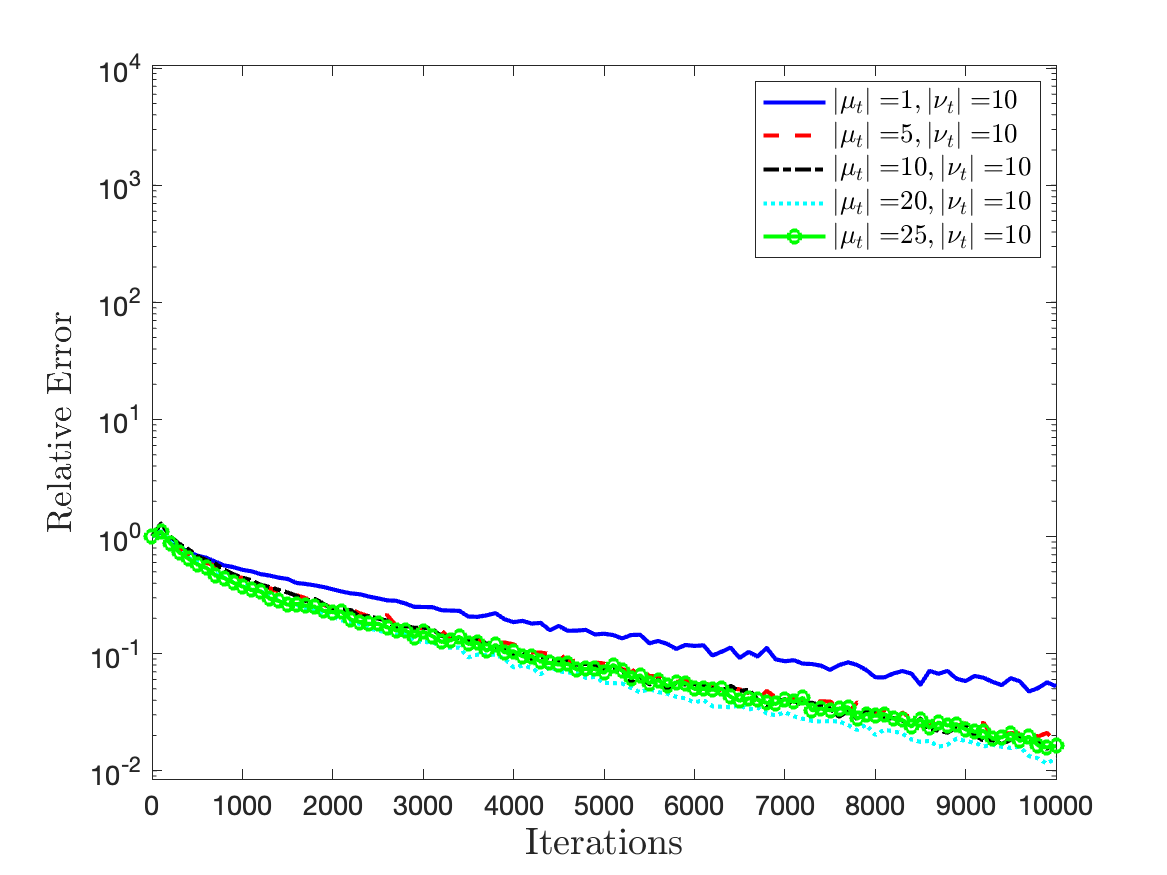}\hfill%
    \includegraphics[width=0.495\textwidth]{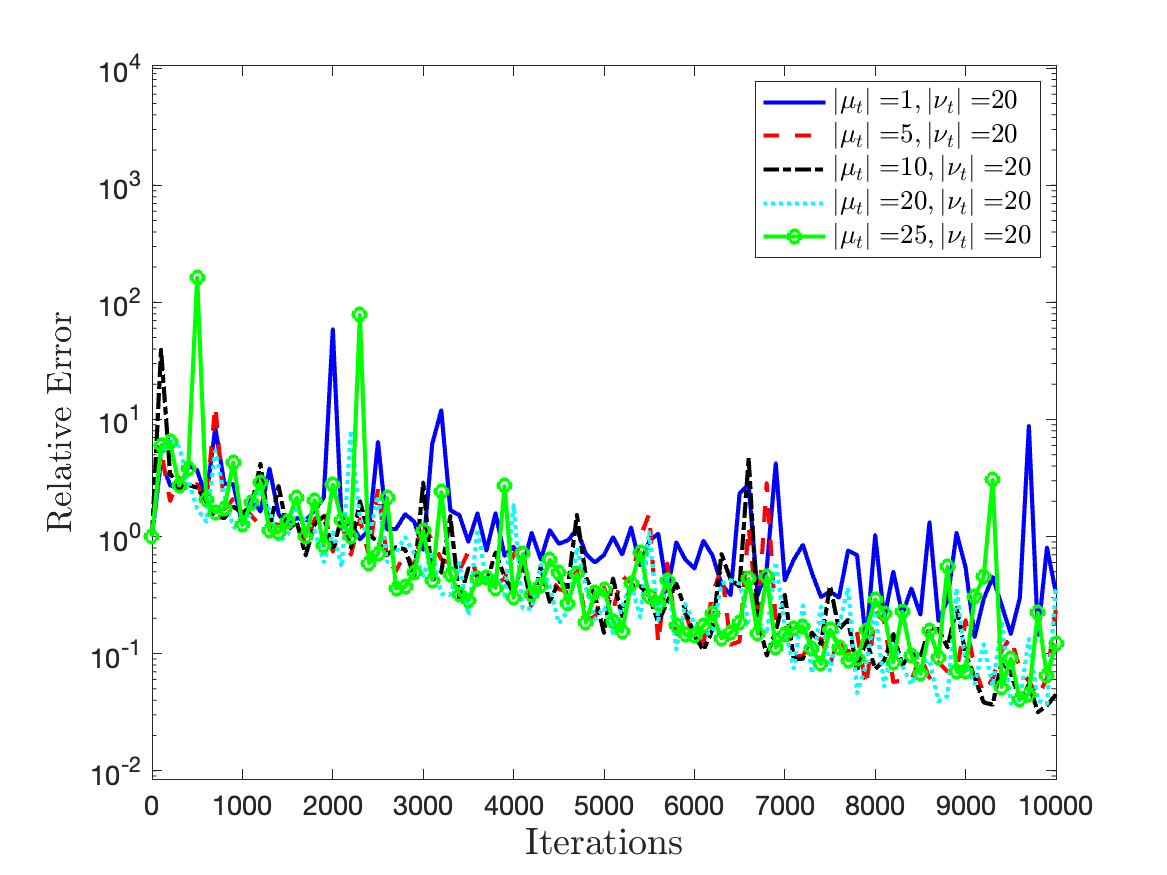}
    \caption{Relative error $\|\tX^{(t)} - \tX^\ddagger\|_F/\|\tX^\ddagger\|_F$ vs iteration $t$ of FacTBREK on consistent linear system when $\tA$ is over-determined, and both $\tU$ and $\tV$ are over-determined. We consider outer block sizes $|\mu_t| \in \{1, 5, 10, 20, 25\}$ and inner system block size (Upper Left) $|\nu_t| = 1$, (Upper Right) $|\nu_t| = 5$, (Lower Left) $|\nu_t| = 10$, (Lower Right) $|\nu_t| = 20$. }\label{fig:facTBREK_differentsystems_case2_4}
\end{figure}

\newpage

\subsubsection{Effect of different types of block sets}

In this section, for both the consistent and inconsistent cases, we compare two types of algorithms. The first type uses FacTBRK and FacTBREK as defined in previous sections, i.e., using blocks of a given size without specifying any row-slice partition of the outer system tensor $\tU$. For the second type of algorithm, at each iteration, we draw a block uniformly at random (and independently of all previous choices) from a partition $T = \{\tau_1,\, \hdots, \tau_m \}$ of the row slice indices of $\tU$. In one case of the second type of algorithm, the $\tau_i$'s have the same size $\mu_t$, and in another case, their size can vary but must average to $\mu_t$ (and cannot be much larger or much smaller than this value).

\begin{figure}
    \includegraphics[width=0.495\textwidth]{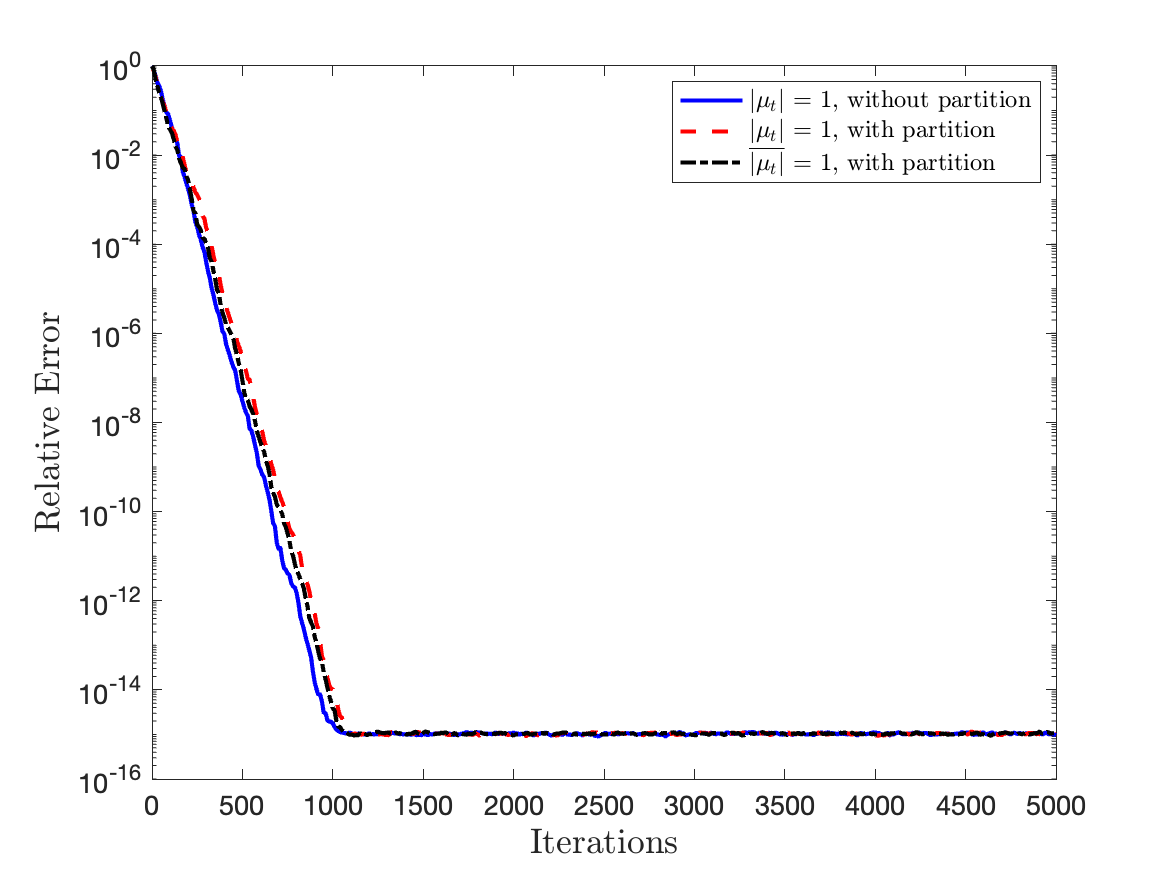}\hfill%
    \includegraphics[width=0.495\textwidth]{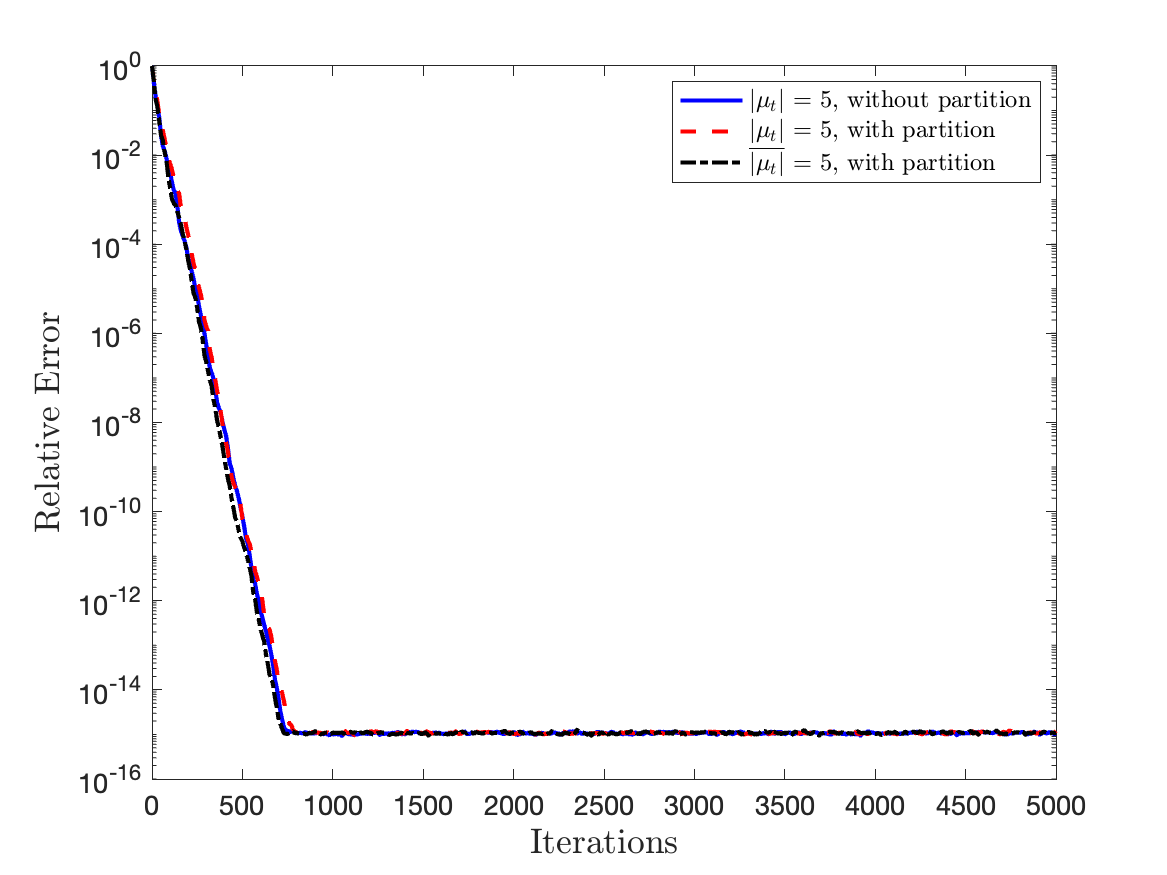}\hfill%
   \includegraphics[width=0.495\textwidth]{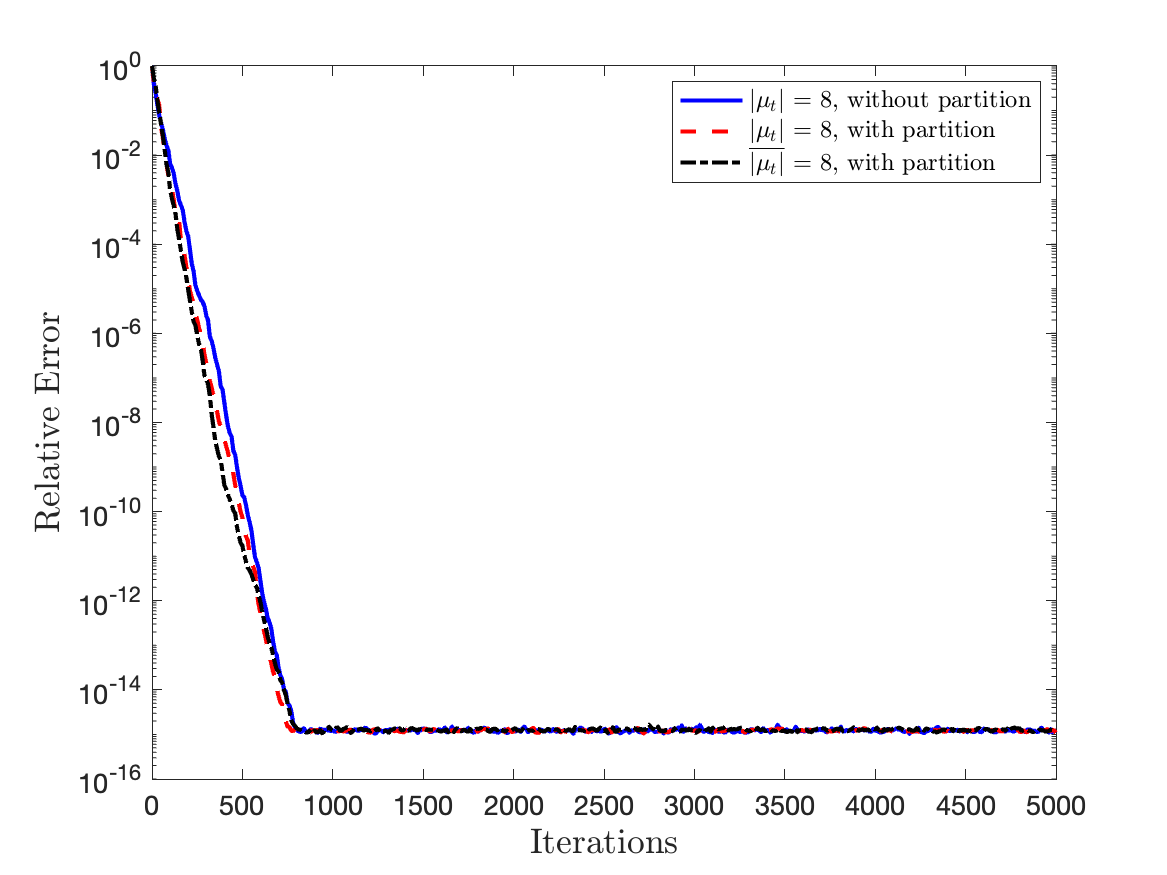}\hfill%
    \includegraphics[width=0.495\textwidth]
    {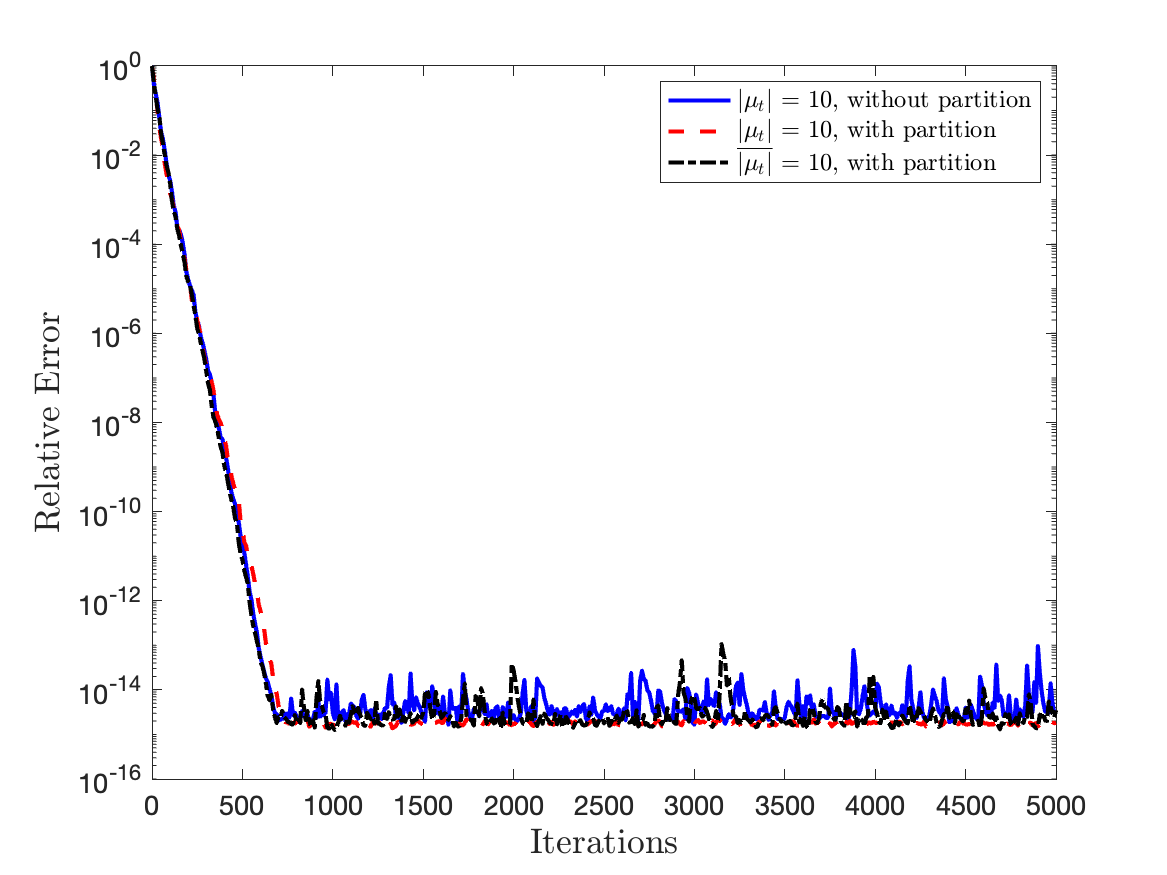}

    \caption{Relative error $\|\tX^{(t)} - \tX^\ddagger\|_F/\|\tX^\ddagger\|_F$ vs iteration $t$ of FacTBRK on a consistent linear system for outer block sizes $|\mu_t| \in \{1, 5, 10\}$ and constant inner system block size $|\nu_t| = 1$ for different types of block partition. The notation $\overline{|\mu_t|}$ means ``average block size $|\mu_t|$''.}\label{fig:facTBRK_typesofblocks}
\end{figure}

The results in Figure \ref{fig:facTBRK_typesofblocks} show that, for a consistent system and inner system block size 1, having different types of block sets for the outer system does not significantly affect the convergence of our FacTBRK algorithm, provided that the block size is not too large. (The reason for the latter restriction is that, heuristically, choosing a large outer block size implies confidence in solving that outer system that is not warranted from using a block of size $\vu_t =$ 1 to solve the inner system. Convergence is not excluded but might come at a much slower rate.) Typically, using partitions of equal-size blocks for $\tU$ or not using any partition at all yields faster and less oscillatory convergence as the size of blocks for the outer system increases. However, we conclude by saying that, in general, these differences are inconsequential in the long run.

\subsubsection{Tensor vs matricized methods}

In these experiments, we compare the behavior of FacTBRK and FacTBREK on the tensor system $\tU \tV \tX = \tY$ to FacTBRK and FacTBREK on the equivalent matricized system $\bcirc(\tU) \bcirc(\tV) \unfold(\tX) = \unfold(\tY)$, for consistent and inconsistent systems.  In each experiment in this section, we generate $\tU \in \mathbb{R}^{40 \times 10 \times 7}, \tV \in \mathbb{R}^{10 \times 5 \times 7}$, and $\tX_{\text{gen}} \in \mathbb{R}^{5 \times 5 \times 7}$.  We generate consistent and inconsistent systems according to the process described above.

\begin{figure}
    \includegraphics[width=0.495\textwidth]{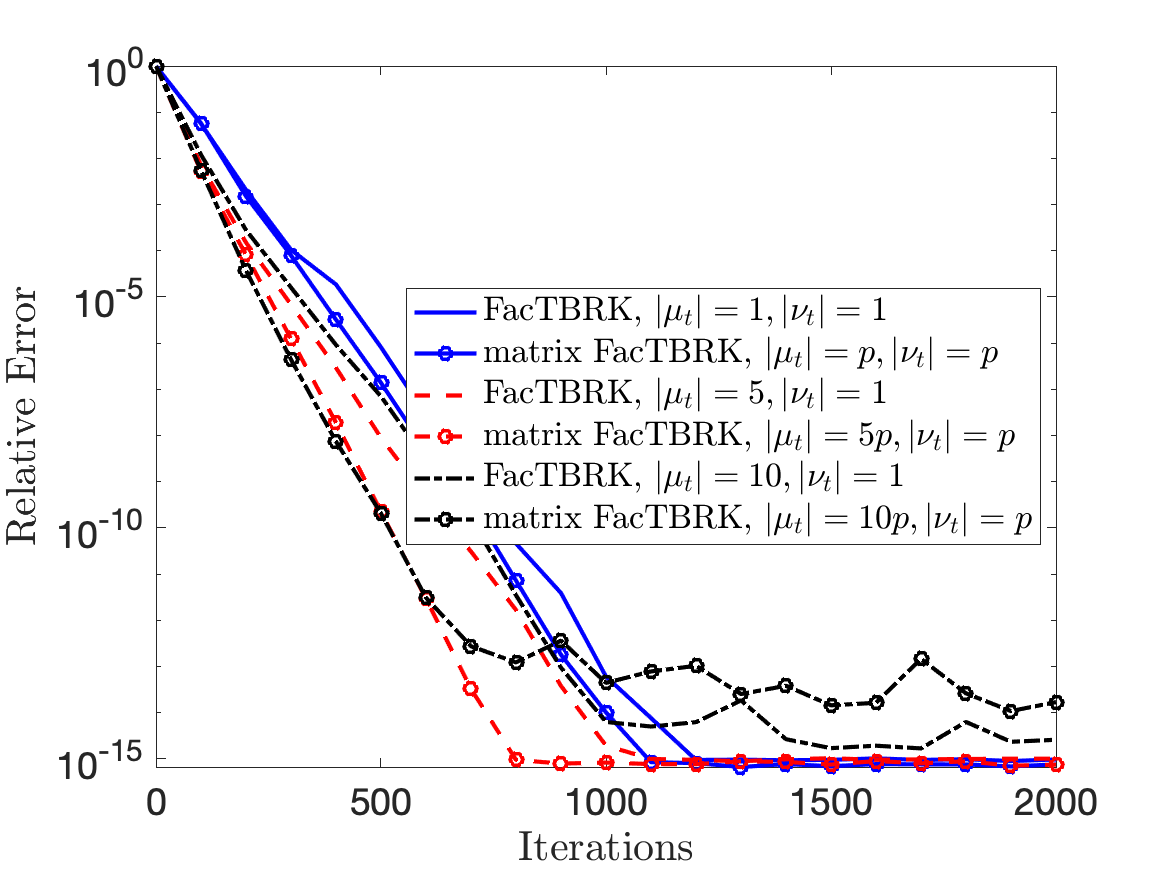}\hfil%
    \includegraphics[width=0.495\textwidth]{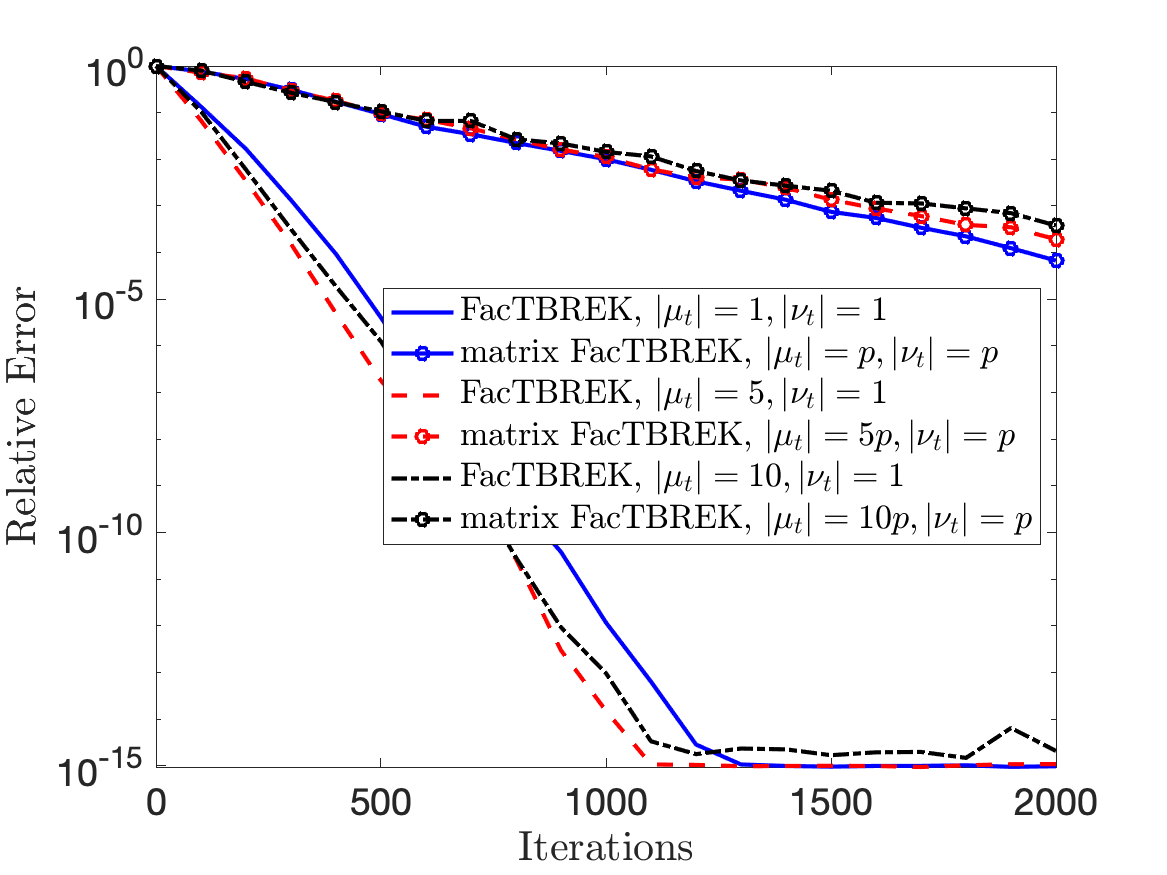}
    \caption{(Left) Relative error $\|\tX^{(t)} - \tX^\ddagger\|_F/\|\tX^\ddagger\|_F$ vs iteration $t$ of FacTBRK on the consistent tensor system $\tU \tV \tX = \tY$ with $|\mu_t| \in \{1, 5, 10\}$ and $|\nu_t| = 1$ and FacTBRK on the equivalent matricized system $\bcirc(\tU) \bcirc(\tV) \unfold(\tX) = \unfold(\tY)$ with  $|\mu_t| \in \{p, 5p, 10p\}$ and  $|\nu_t| = p$ where $p = 7$. (Right) Relative error $\|\tX^{(t)} - \tX^\ddagger\|_F/\|\tX^\ddagger\|_F$ vs iteration $t$ of FacTBREK on the inconsistent tensor system $\tU \tV \tX = \tY$ with $|\mu_t| \in \{1, 5, 10\}$ and $|\nu_t| = 1$ and FacTBREK on the equivalent matricized system $\bcirc(\tU) \bcirc(\tV) \unfold(\tX) = \unfold(\tY)$ with  $|\mu_t| \in \{p, 5p, 10p\}$ and  $|\nu_t| = p$ where $p = 7$.}\label{fig:tensorvsmatricized}
\end{figure}

In our first experiment, we generate a consistent system according to the process described above and run FacTBRK (Algorithm~\ref{alg:facTBRK}) on the tensor system $\tU \tV \tX = \tY$ with outer system block sizes $|\mu_t| \in \{1, 5, 10\}$ and inner system block size $|\nu_t| = 1$ and FacTBRK (Algorithm~\ref{alg:facTBRK}) on the equivalent matricized system $\bcirc(\tU) \bcirc(\tV) \unfold(\tX) = \unfold(\tY)$ with outer system block sizes $|\mu_t| \in \{p, 5p, 10p\}$ and inner system block size $|\nu_t| = p$ where $p = 7$ is the third tensor dimension.  We note that these choices of block sizes imply that the FacTBRK method applied to each system is accessing the same amount of information from the tensors $\tU$ and $\tV$ in each iteration.  Here, the blocks are uniformly sampled from all possible blocks of this size, $T_{\tU} = \{\mu \in \mathcal{P}([m]) : |\mu| = |\mu_t|\}$ and $T_{\tV} = \{\nu \in \mathcal{P}([m_1]) : |\nu| = |\nu_t|\}$.  The results of this experiment are plotted on the left of Figure~\ref{fig:tensorvsmatricized}.  We note that in this experiment, FacTBRK actually converges slightly faster on the matricized system $\bcirc(\tU) \bcirc(\tV) \unfold(\tX) = \unfold(\tY)$ than on the tensor system $\tU \tV \tX = \tY$, which we hypothesize to be due to the fact that this method may sample from a broader set of blocks than FacTBRK on the tensor system.

In our second experiment, we generate an inconsistent system according to the process described above and run FacTBREK (Algorithm~\ref{alg:facTBREK}) on the tensor system $\tU \tV \tX = \tY$ with outer system block sizes $|\mu_t| \in \{1, 5, 10\}$ and inner system block size $|\nu_t| = 1$ and FacTBREK (Algorithm~\ref{alg:facTBREK}) on the equivalent matricized system $\bcirc(\tU) \bcirc(\tV) \unfold(\tX) = \unfold(\tY)$ with outer system block sizes $|\mu_t| \in \{p, 5p, 10p\}$ and inner system block size $|\nu_t| = p$ where $p = 7$ is the third tensor dimension.  We again note that these choices of block sizes imply that the FacTBREK method applied to each system is accessing the same amount of information from the tensors $\tU$ and $\tV$ in each iteration.  Here, the blocks are uniformly sampled from all possible blocks of this size, $T_{\tU} = \{\mu \in \mathcal{P}([m]) : |\mu| = |\mu_t|\}$ and $T_{\tV} = \{\nu \in \mathcal{P}([m_1]) : |\nu| = |\nu_t|\}$.  The results of this experiment are plotted on the right of Figure~\ref{fig:tensorvsmatricized}.  In this case, we note that FacTBREK applied to the tensor linear system $\tU \tV \tX = \tY$ converges significantly faster than FacTBREK applied to the matricized system $\bcirc(\tU) \bcirc(\tV) \unfold(\tX) = \unfold(\tY)$; we hypothesize that this is due to the extended step of FacTBREK learning the components of $\tY$ which contribute to the system inconsistency much more quickly on the tensor system than on the matricized system.

\subsection{Applications to Video Deblurring}
\label{subsec:video_deblurring}
We consider the problem of recovering a true video tensor $\tX \in \R^{m \times n \times p}$ from a blurry video tensor $\tY \in \R^{m \times n \times p}$ with a known blurring operator. We further assume that the video is blurred frame by frame using a circular convolution kernel $\mat{H} \in \R^{m_1 \times n_1}$. Without loss of generality, we can assume that $\mat{H} \in \R^{m \times n}$ by suitably padding the kernel with zeros. Then, the connection between the  slices of the blurry and the true video tensors can be expressed as

 $$ \begin{bmatrix} \mat{H}_1 & \mat{H}_m & \cdots &   \mat{H}_2 \\
 \mat{H}_2 & \mat{H}_{1} & \cdots  &\mat{H}_3\\ \mat{H}_3 &\mat{H}_2 & \cdots  &\vdots \\ \vdots & \vdots &\vdots &\vdots \\ \mat{H}_{m-1} & \mat{H}_{m-2}  &\cdots &\mat{H}_m\\
 \mat{H}_m & \mat{H}_{m-1} & \cdots  & \mat{H_1}\end{bmatrix}
 \begin{bmatrix} \bigg| & \cdots & \bigg| \\ \\ \tt{Vec}\left(\tX_{1::}\right) &\cdots & \tt{Vec}\left(\tX_{m::} \right) \\ \\\bigg| & \cdots & \bigg|\\ \end{bmatrix}= \begin{bmatrix} \bigg| & \cdots & \bigg| \\ \\ \tt{Vec}\left(\tX_{1::}\right) &\cdots & \tt{Vec}\left(\tY_{m::} \right) \\ \\\bigg| & \cdots & \bigg|,\\ \end{bmatrix}$$
where $\mat{H_1} = circ(\vh_i) \in \R^{n \times n}$ is a circulant matrix and $\vh_i \in \R^n$ denotes $i^{th}$ row of $\mat{H}$. Observe that using Definition \ref{def:t-product}, this system can be written as the following t-linear system
$$ \tH \ast \widetilde{\tX} = \ \widetilde{\tY}$$,
where the $i$-th frontal face of the blurring tensor $\tH \in \R^{n \times n \times m}$ is given by $\mat{H}_i$ and $\widetilde{X}$ and $\widetilde{Y}$ are obtained by suitably refolding $\mat{X}$ and $\mat{Y}$ respectively to dimensions $ n \times p \times m $.

\begin{figure}
    \includegraphics[width=0.65\textwidth]{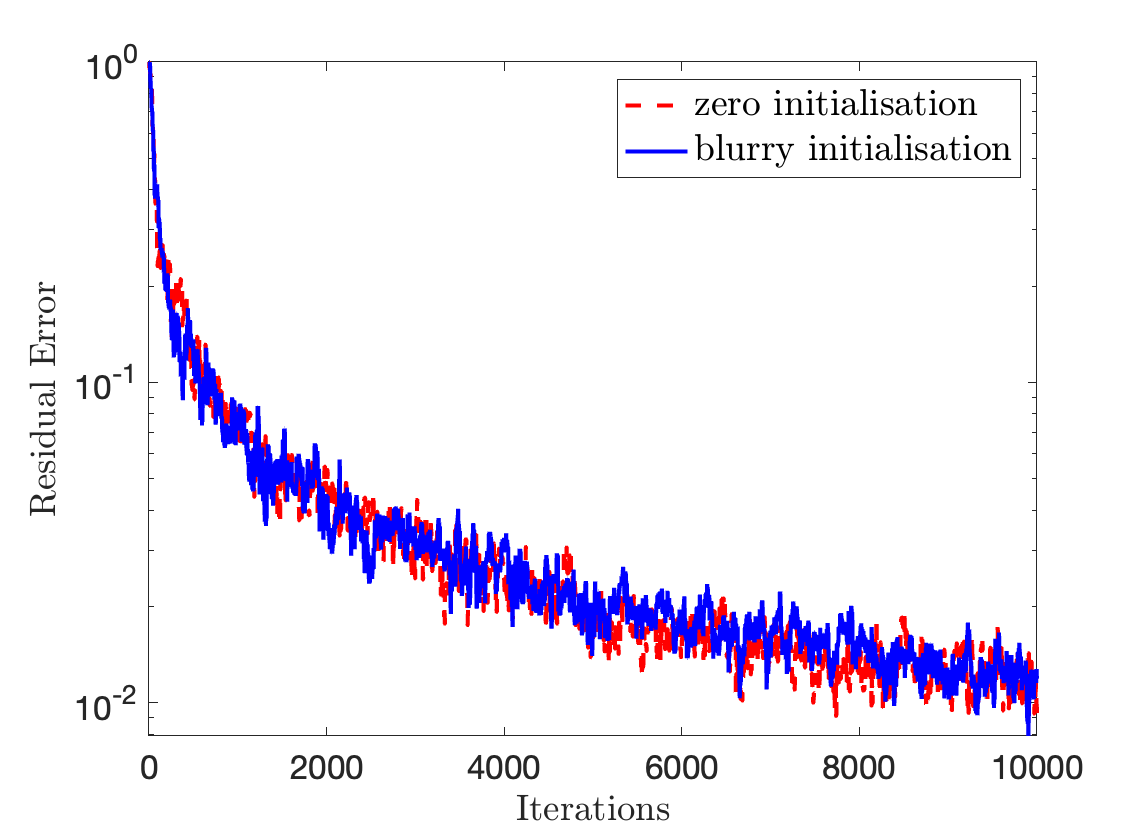}%
    \includegraphics[width=0.31\textwidth]{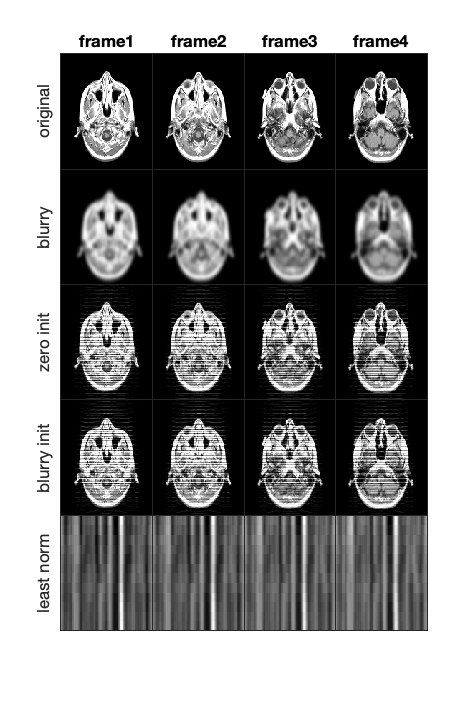}
    \caption{Deblurring  of twice sequentially blurred video data using the FacTRK.  (Left) Residual error of FacTRK iterates for two different initializations. (Right) Frames from original video (top row), twice blurred frames (second row), frames recovered using FacTRK (third and fourth row), and the least norm solution (bottom row).}\label{fig:deblurring}
\end{figure}

In the experiment corresponding to Figure~\ref{fig:deblurring}, we consider the recovery of 12 frames of twice blurred MRI images of size $128 \times 128$. The images were first circularly blurred by using a $5\times5$ Gaussian filter followed by a $5\times5$ averaging filter. This corresponds to solving a t-product system, $\tU\tV\tX = \tY$, where $\tX\in\R^{128\times 12 \times 128}$ represents the original images, $\tY\in\R^{128\times12\times128}$ the blurred images, and $\tU$ and $\tV \in \R^{128 \times 128 \times 128}$ correspond to the Gaussian blurring and average blurring operators, respectively. Using FacTRK, we are able to recover the frames well within a residual error of $10^{-2}$. Further, these recovered images match the original video data qualitatively.

\section{Conclusion}

We have developed two algorithms to solve tensor systems in which the measurement tensor is given as a t-product of two tensors. Our first method, named FacTBRK, obtains an optimal solution for \eqref{eq: full system} when both the outer system and inner systems are consistent and the second method, named FacTBREK, tackles the same problem as before, except that the outer system \eqref{eq:outer_system} can now be inconsistent. The crux of our algorithms are the interlacing Kaczmarz updates, extended from matrix to tensor structure. Our theoretical analysis and numerical simulations demonstrate that our algorithms converge exponentially in expectation to the least-norm least-squares solution of overdetermined or underdetermined linear systems. In particular, we show the strength of our proposed methods by applying them to an image deblurring problem and by illustrating that divergence occurs in cases that are outside the assumptions of our theorem. Moreover, we place our approach in a broader framework by establishing connections between our novel methods and previous randomized Kaczmarz methods, particularly in the non-factorized case and for the matricized version of the linear system.

\section{Acknowledgements}
The initial research for this effort was conducted at the Research Collaboration Workshop for Women in Data Science and Mathematics (WiSDM), held in August 2023 at the Institute for Pure and Applied Mathematics (IPAM). IPAM, AWM, and DIMACS funded the workshop (NSF grant CCF1144502).

This material is based upon work supported by the National Science Foundation under Grant No. DMS-1928930, while several of the authors were in residence at the Mathematical Sciences Research Institute in Berkeley, California, during the summer of 2024.

Several of the authors also appreciate support provided to them at a SQuaRE at the American Institute of Mathematics. The authors thank AIM for providing a supportive and mathematically rich environment.

This material is based upon work supported by the National Science Foundation under Grant No. DMS-1929284 while the author was in residence at the Institute for Computational and Experimental Research in Mathematics in Providence, RI, during the ``Randomized Algorithms for Tensor Problems with Factorized Operations or Data" Collaborate@ICERM.

JH was partially supported by NSF DMS \#2211318. DN was partially supported by NSF DMS \#2011140. KYD was partially supported by NSF \#2232344.
\bibliographystyle{plain}
\bibliography{main}

\appendix
\section{A: Effect of different systems and different block sizes. (Continued)}\label{sec:appendix}

\begin{figure}
    \includegraphics[width=0.495\textwidth]{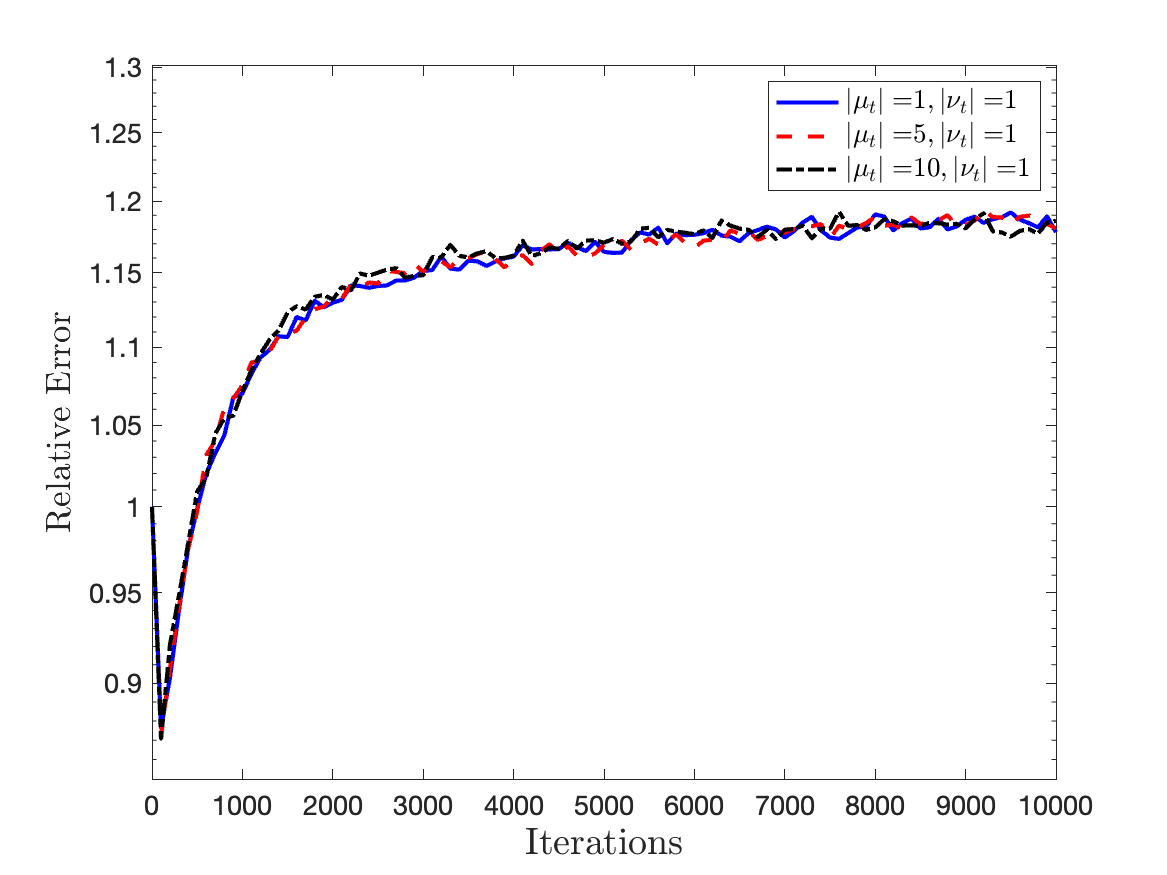}\hfill%
    \includegraphics[width=0.495\textwidth]{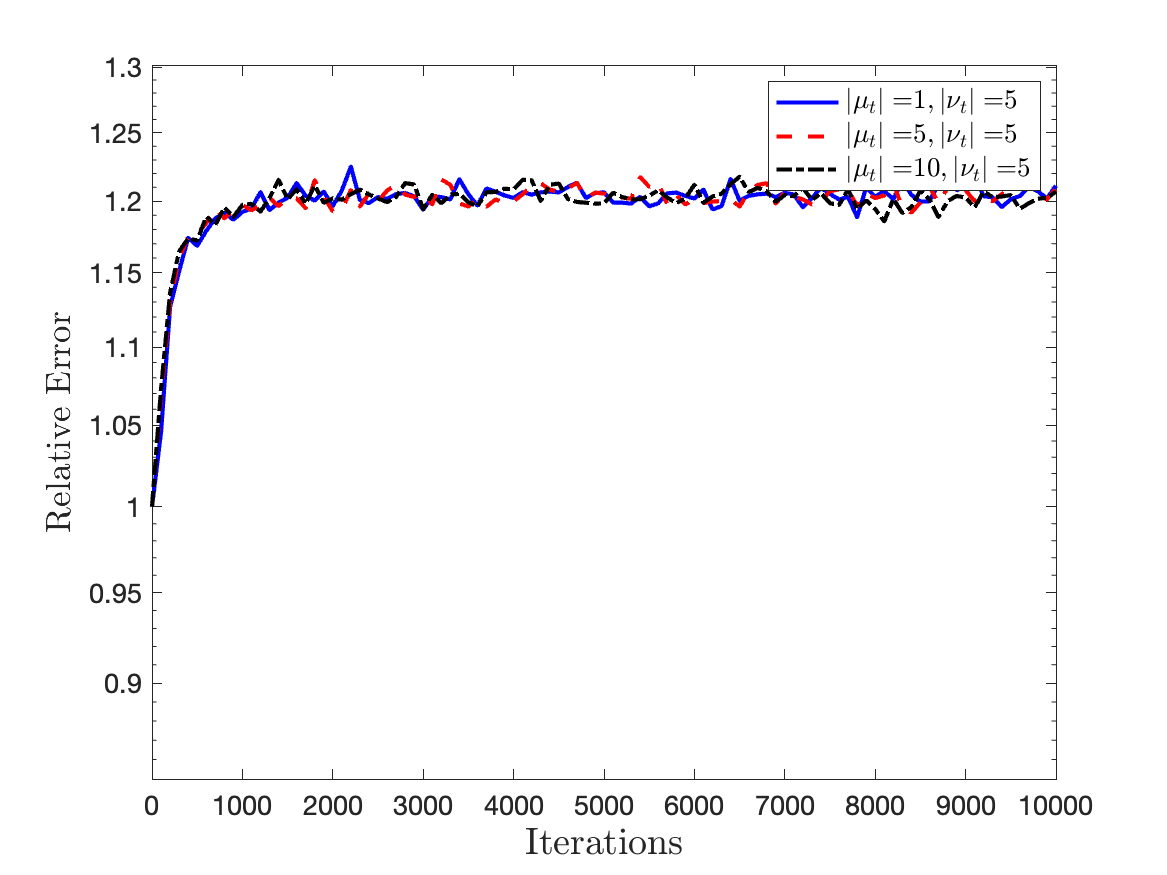}\hfill%
    \includegraphics[width=0.495\textwidth]{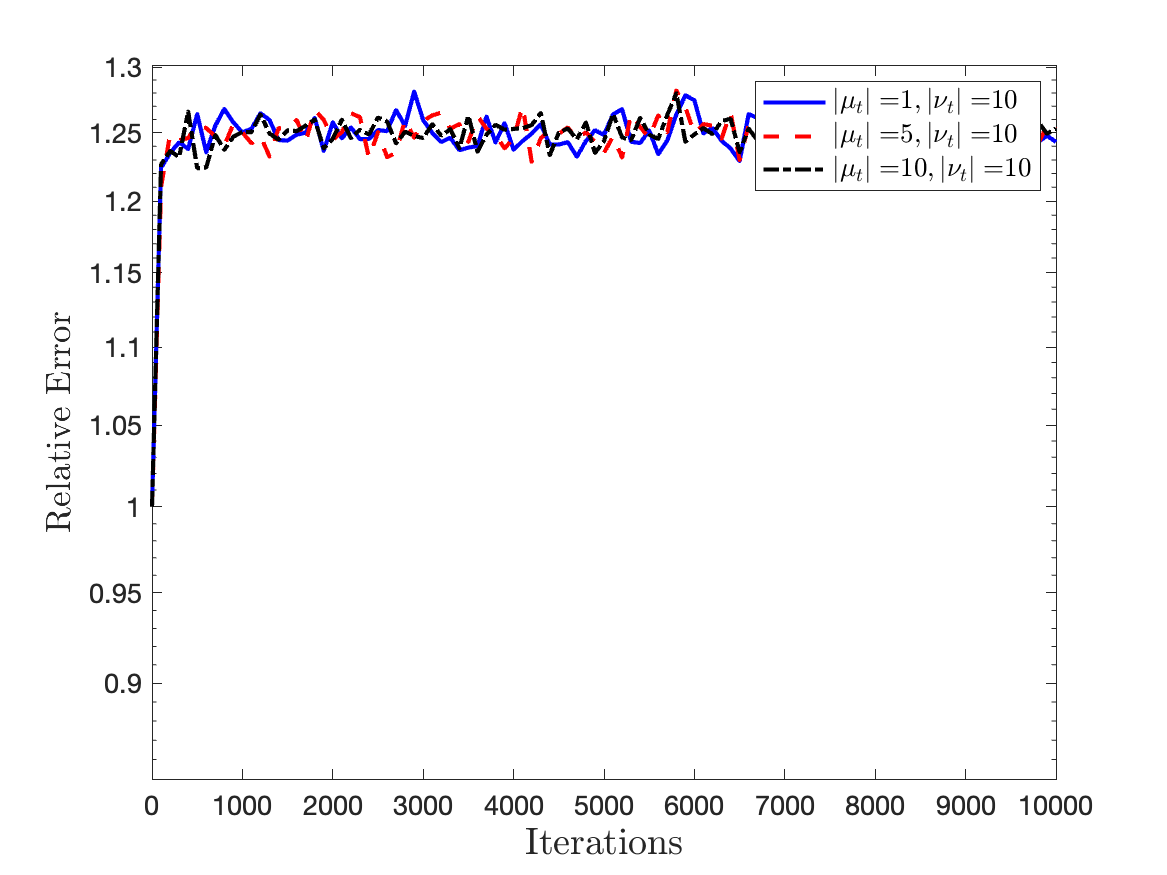}%

    \caption{Relative error $\|\tX^{(t)} - \tX^\ddagger\|_F/\|\tX^\ddagger\|_F$ vs iteration $t$ of FacTBRK on consistent linear system when $\tA$ is under-determined, $\tU$ is under-determined and $\tV$ is over-determined. We consider outer block sizes $|\mu_t| \in \{1, 5, 10\}$ and inner system block size (Upper Left) $|\nu_t| = 1$, (Upper Right) $|\nu_t| = 5$, (Lower) $|\nu_t| = 10$. }\label{fig:facTBRK_differentsystems_case1_2}
\end{figure}

\begin{figure}[h]
    \includegraphics[width=0.495\textwidth]{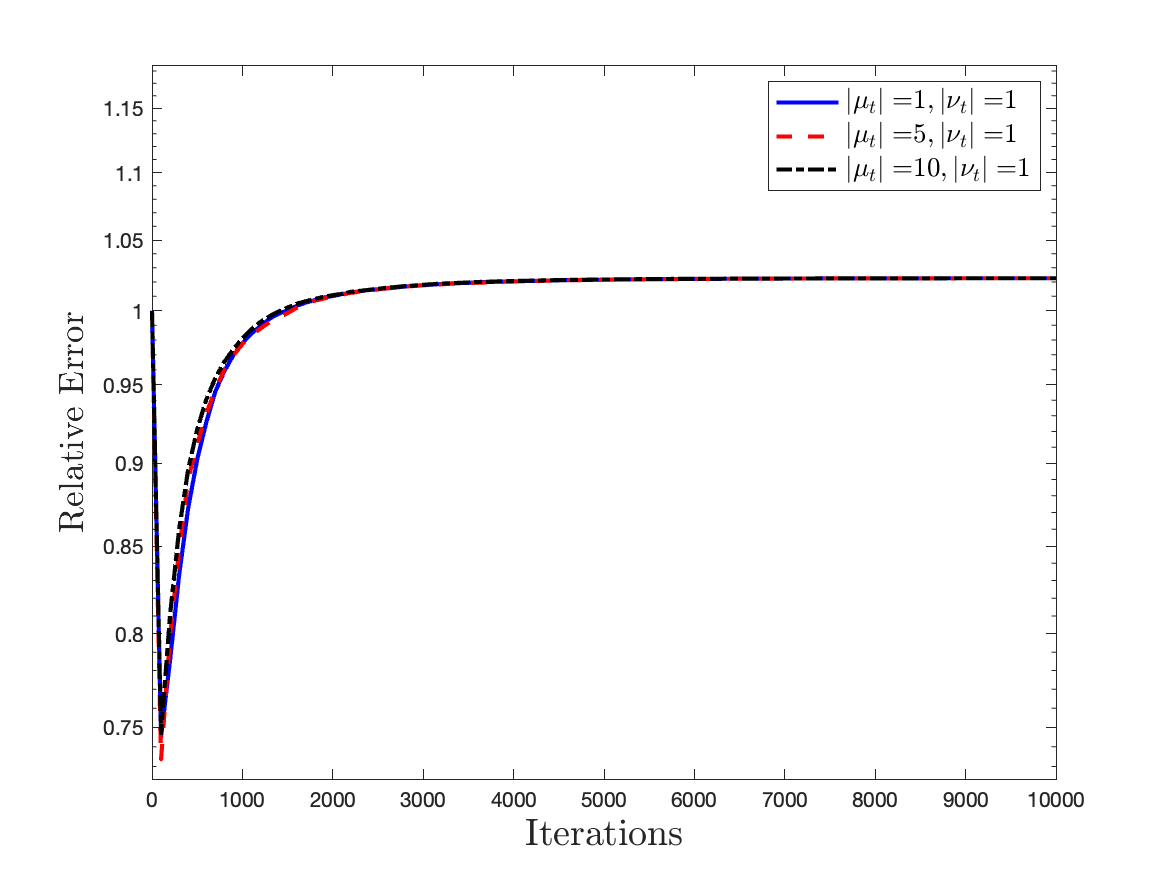}\hfill%
    \includegraphics[width=0.495\textwidth]{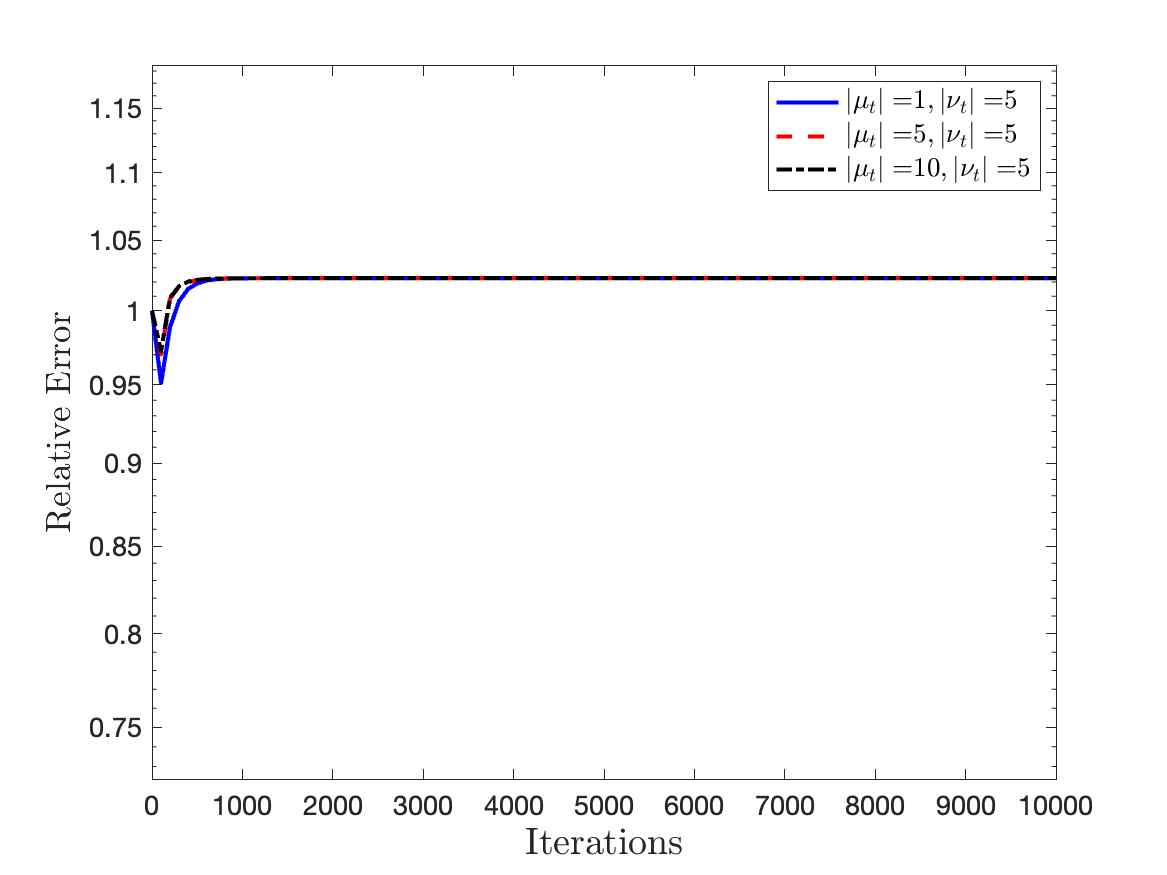}\hfill%
    \includegraphics[width=0.495\textwidth]{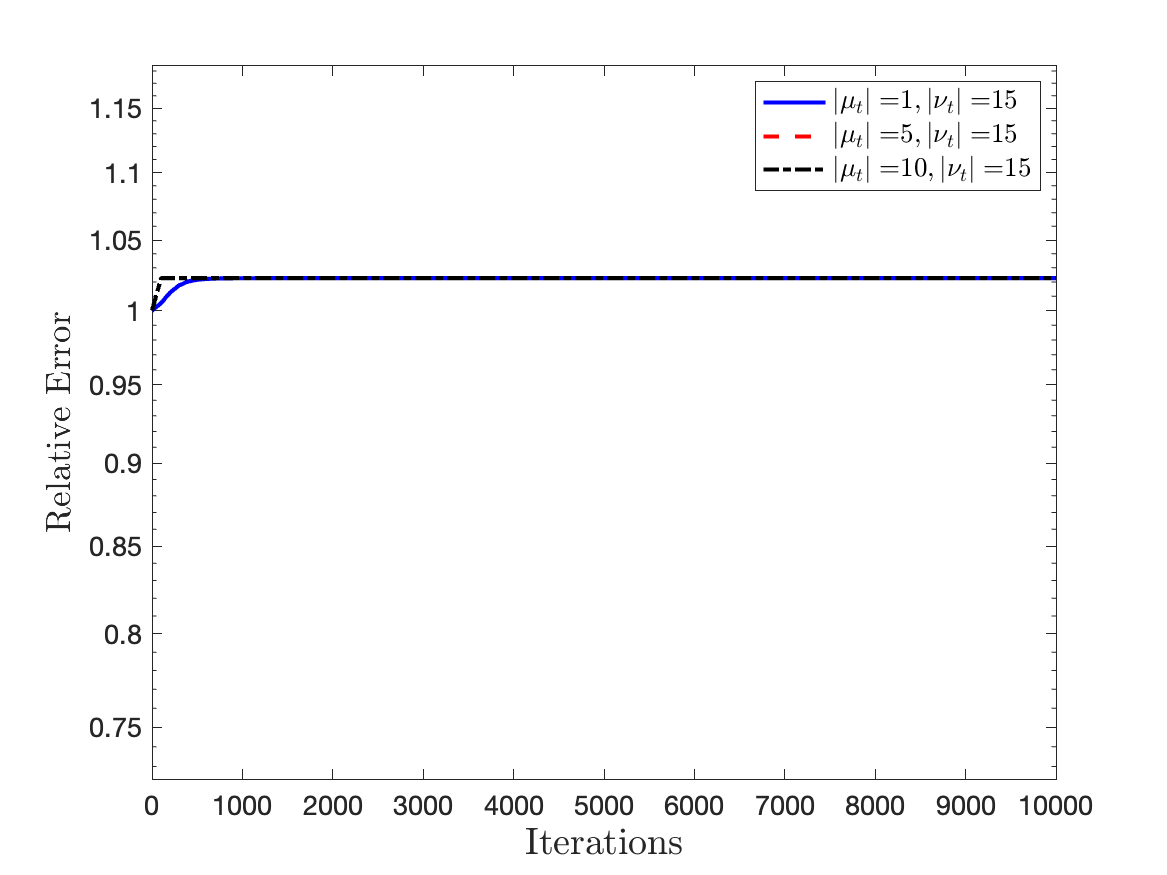}
    \caption{Relative error $\|\tX^{(t)} - \tX^\ddagger\|_F/\|\tX^\ddagger\|_F$ vs iteration $t$ of FacTBRK on consistent linear system when $\tA$ is under-determined and both $\tU$ and $\tV$ are under-determined. We consider outer block sizes $|\mu_t| \in \{1, 5, 10\}$ and inner system block size (Upper Left) $|\nu_t| = 1$, (Upper Right) $|\nu_t| = 5$, and (Lower) $|\nu_t| = 15$. }\label{fig:facTBRK_differentsystems_case1_3}
\end{figure}

\begin{figure}
    \includegraphics[width=0.495\textwidth]{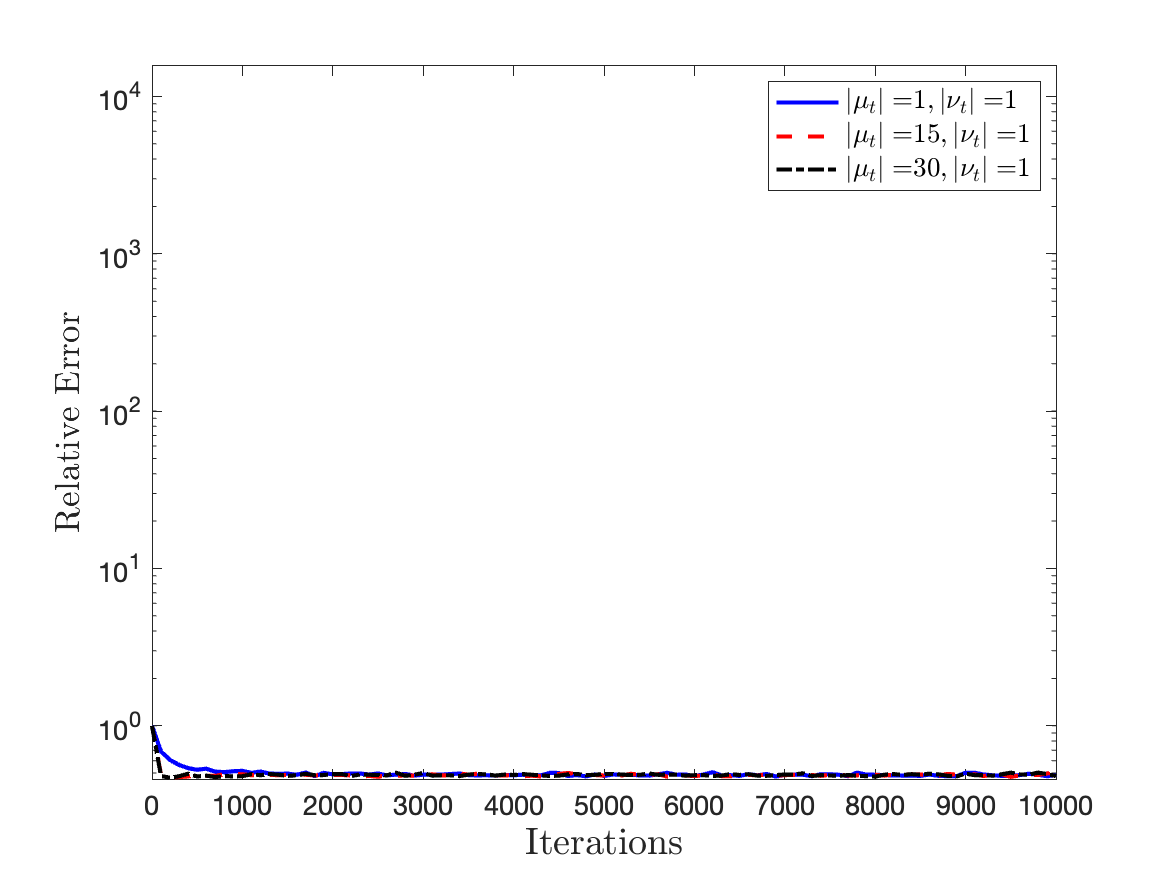}\hfill%
    \includegraphics[width=0.495\textwidth]{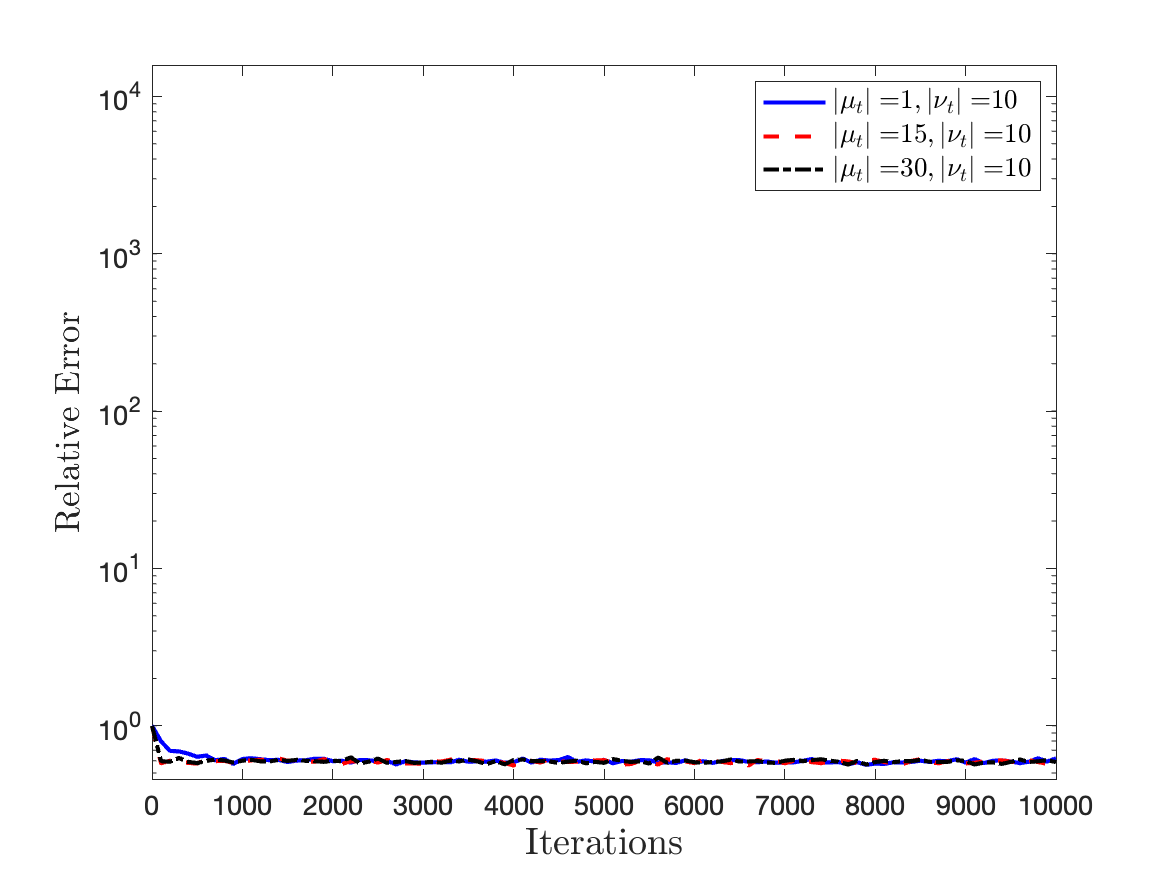}

    \includegraphics[width=0.495\textwidth]{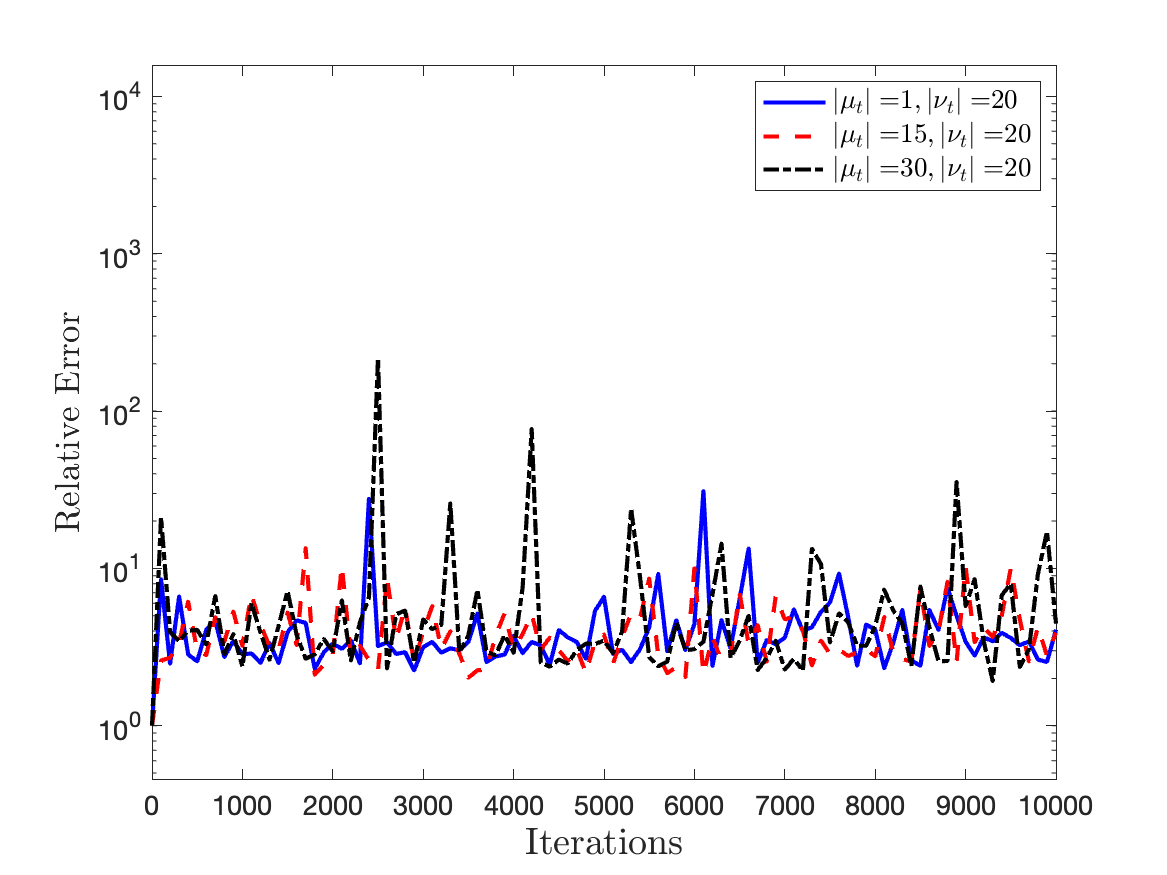}\hfill%

    \caption{Relative error $\|\tX^{(t)} - \tX^\ddagger\|_F/\|\tX^\ddagger\|_F$ vs iteration $t$ of FacTBRK on consistent linear system when $\tA$ is over-determined, $\tU$ is under-determined and $\tV$ is over-determined. We consider outer block sizes $|\mu_t| \in \{1, 10, 20\}$ and inner system block size (Upper Left) $|\nu_t| = 1$, (Upper Right) $|\nu_t| = 10$, (Lower) $|\nu_t| = 20$. }\label{fig:facTBRK_differentsystems_case2_2}
\end{figure}

\begin{figure}
    \includegraphics[width=0.495\textwidth]{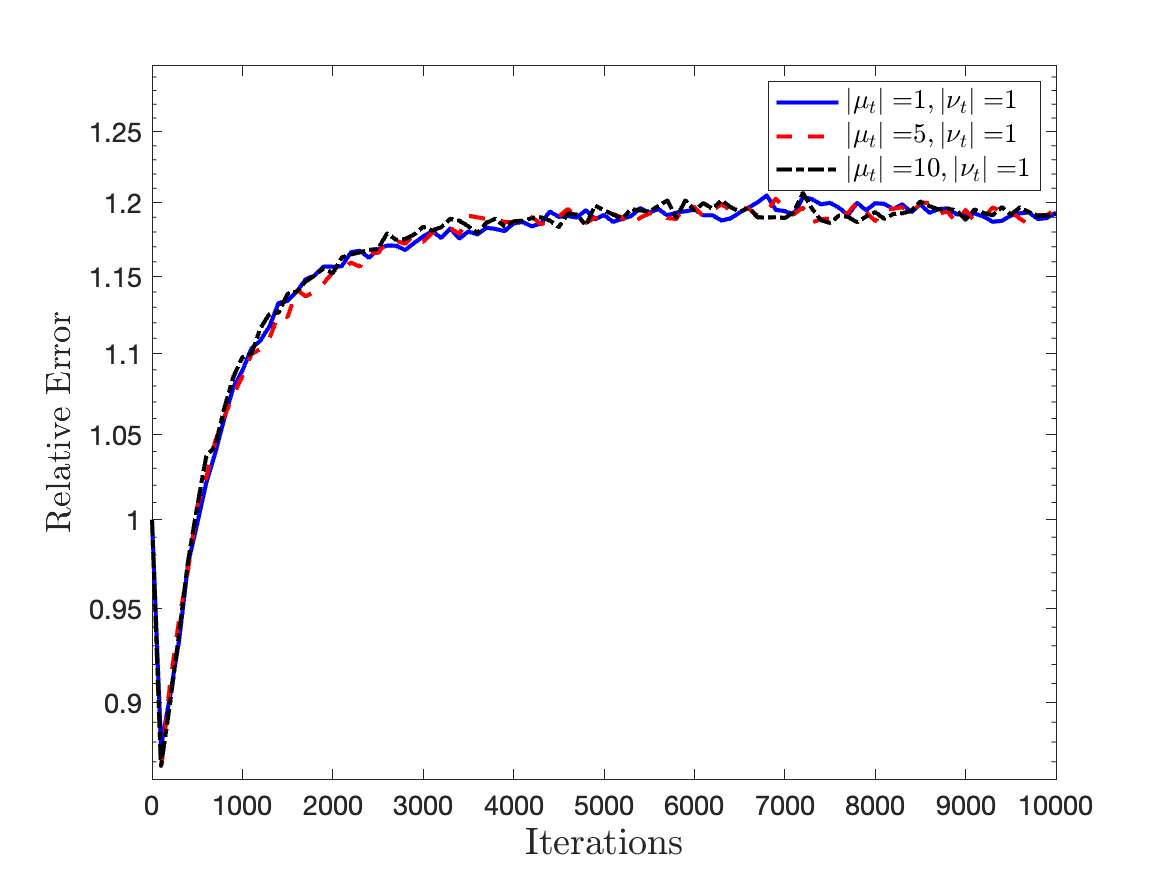}\hfill%
    \includegraphics[width=0.495\textwidth]{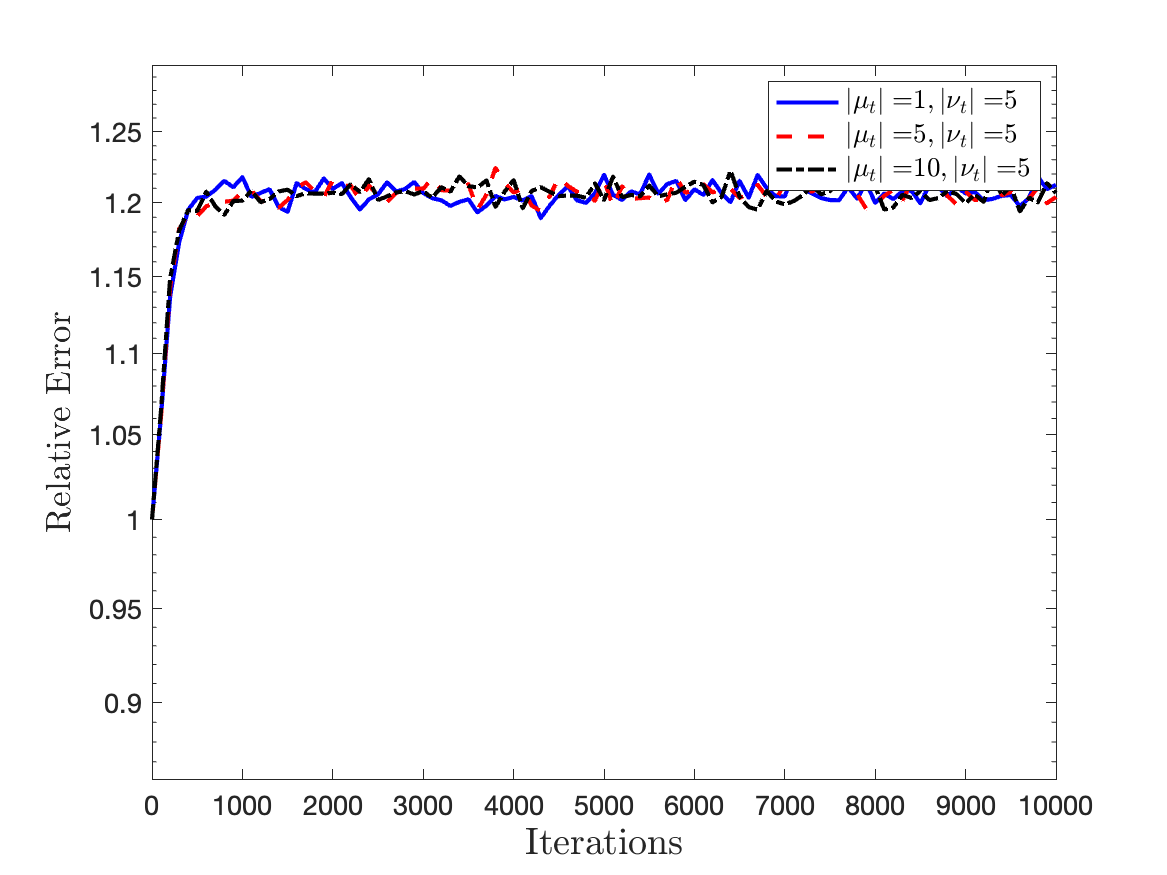}

    \includegraphics[width=0.495\textwidth]{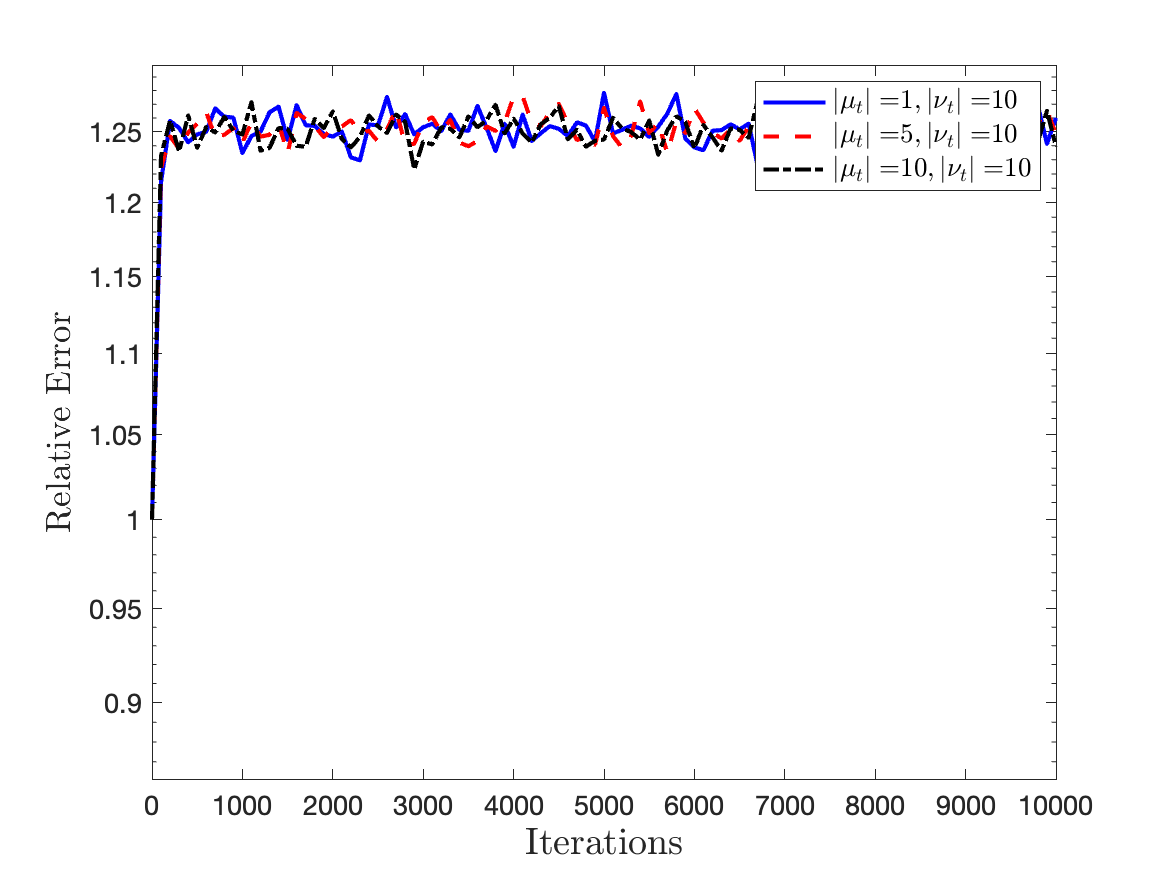}\hfill%

    \caption{Relative error $\|\tX^{(t)} - \tX^\ddagger\|_F/\|\tX^\ddagger\|_F$ vs iteration $t$ of FacTBREK on consistent linear system when $\tA$ is under-determined, $\tU$ is under-determined and $\tV$ is over-determined. We consider outer block sizes $|\mu_t| \in \{1, 5, 10\}$ and inner system block size (Upper Left) $|\nu_t| = 1$, (Upper Right) $|\nu_t| = 5$, (Lower) $|\nu_t| = 10$. }\label{fig:facTBREK_differentsystems_case1_2}
\end{figure}

\begin{figure}
    \includegraphics[width=0.495\textwidth]{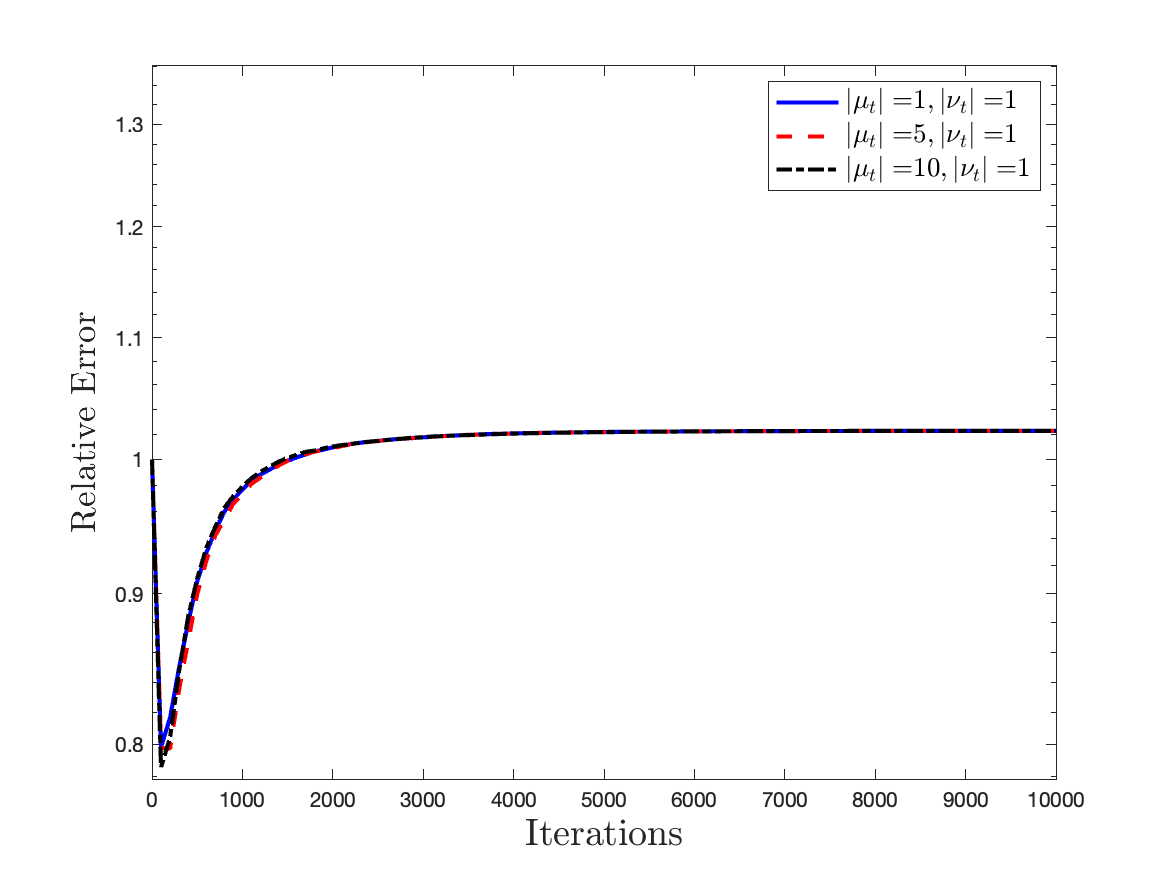}\hfill%
    \includegraphics[width=0.495\textwidth]{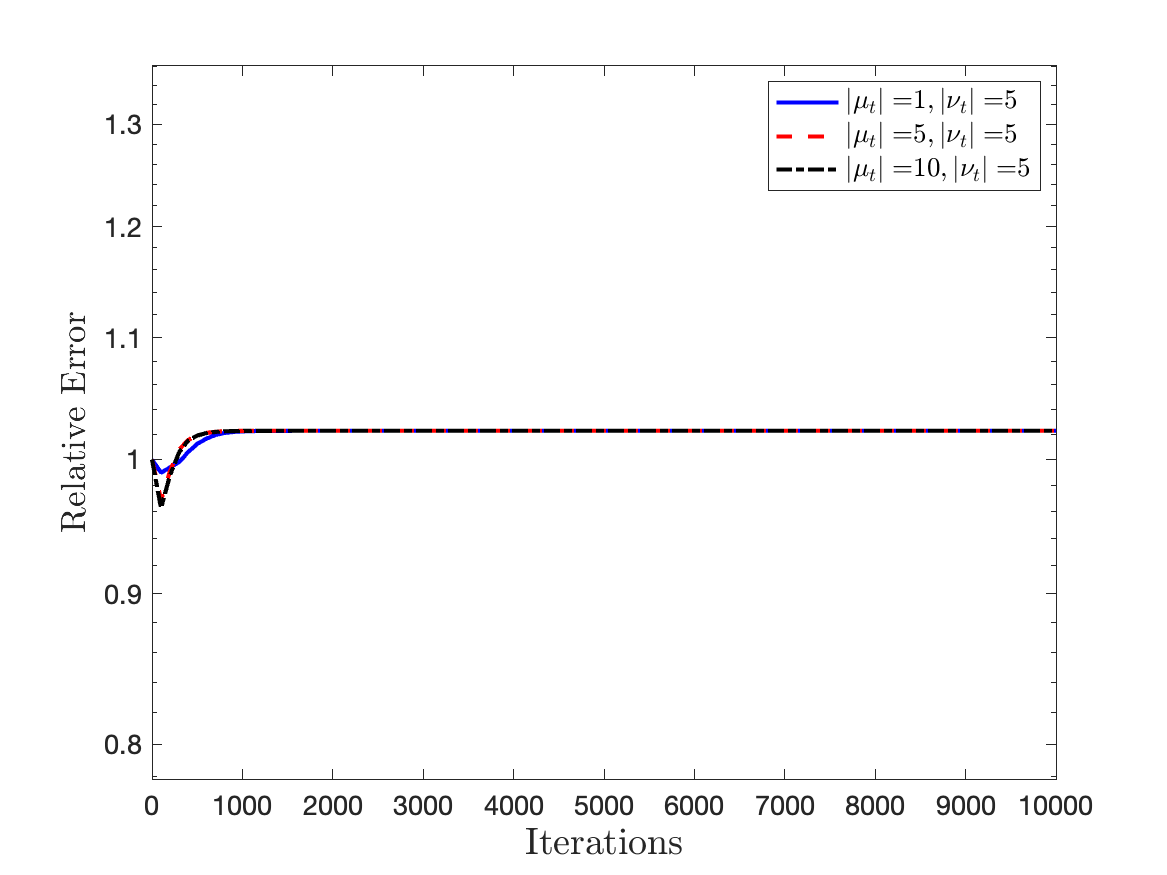}\hfill%
    \includegraphics[width=0.495\textwidth]{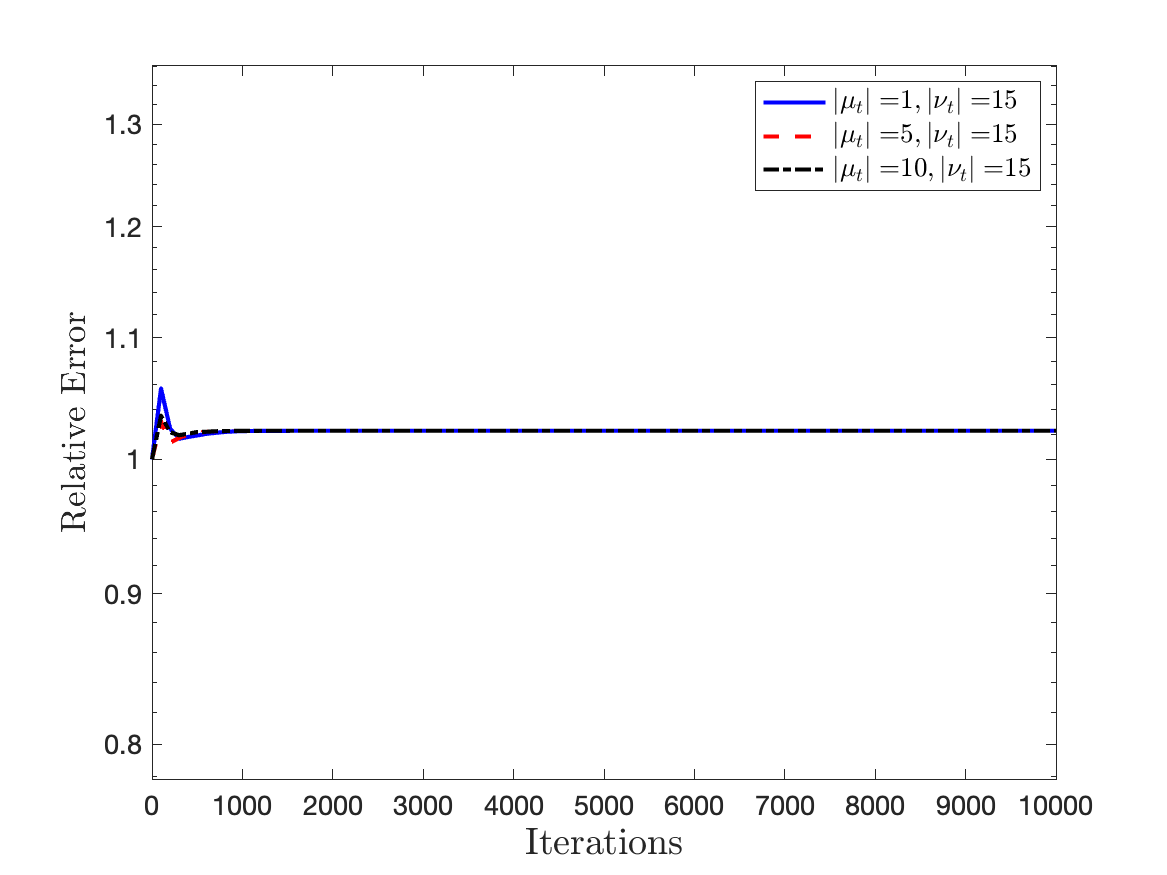}
    \caption{Relative error $\|\tX^{(t)} - \tX^\ddagger\|_F/\|\tX^\ddagger\|_F$ vs iteration $t$ of FacTBREK on consistent linear system when $\tA$ is under-determined, and both $\tU$ and $\tV$ are under-determined. We consider outer block sizes $|\mu_t| \in \{1, 5, 10\}$ and inner system block size (Upper Left) $|\nu_t| = 1$, (Upper Right) $|\nu_t| = 5$, and (Lower) $|\nu_t| = 15$. }\label{fig:facTBREK_differentsystems_case1_3}
\end{figure}

\begin{figure}
    \includegraphics[width=0.495\textwidth]{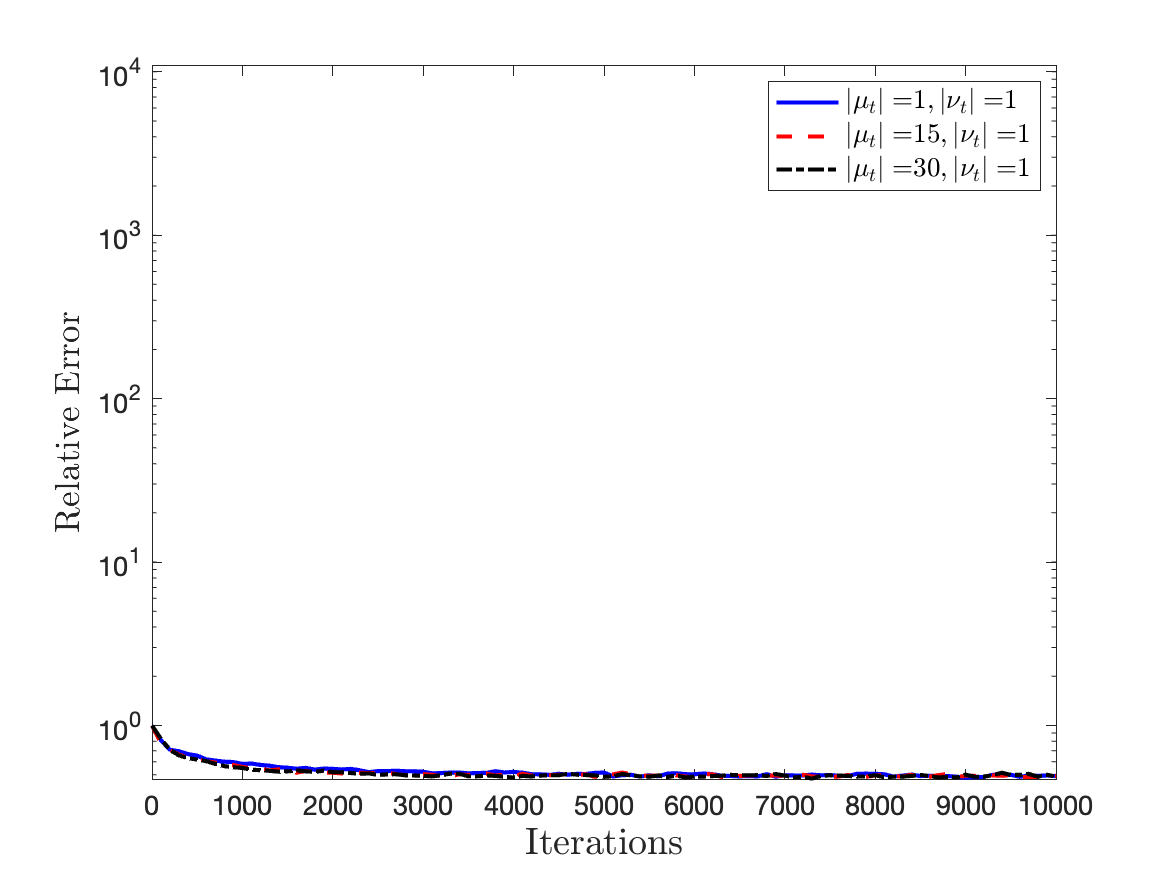}\hfill%
    \includegraphics[width=0.495\textwidth]{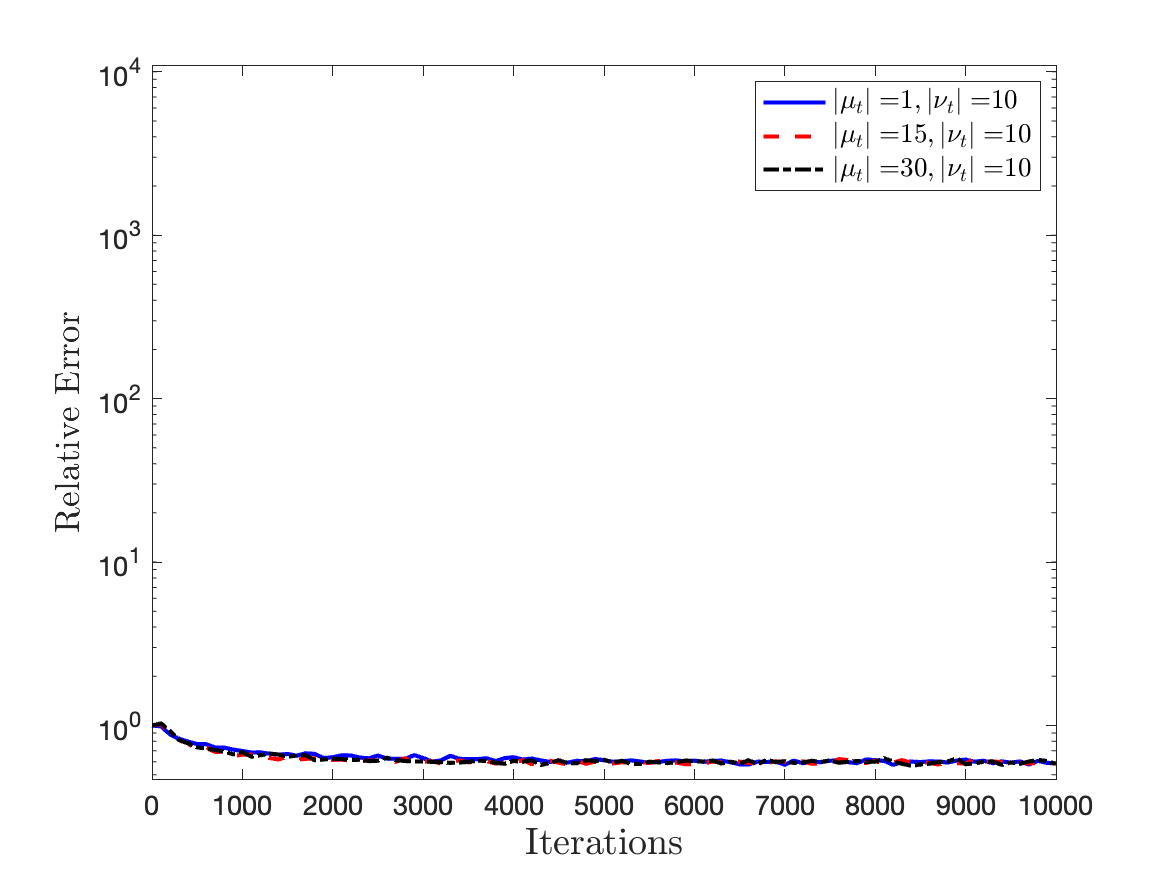}

    \includegraphics[width=0.495\textwidth]{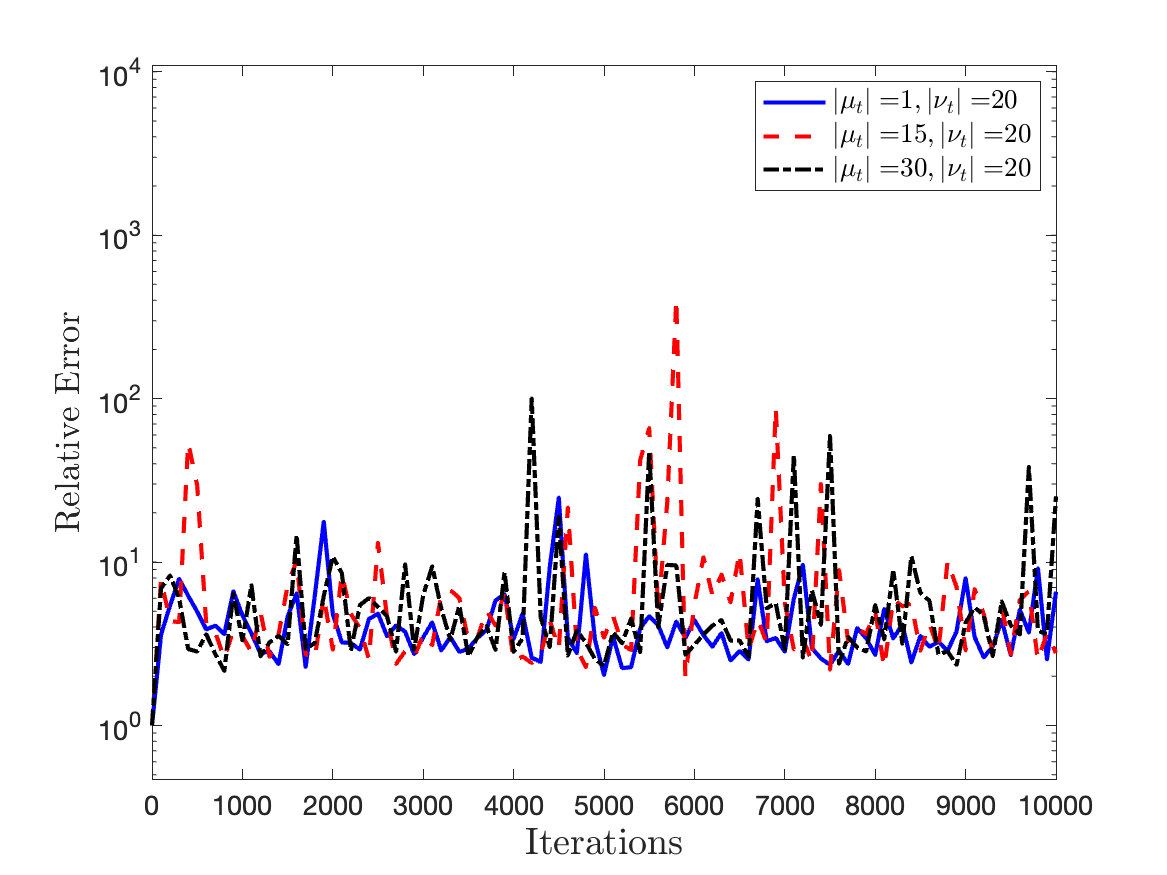}\hfill%

    \caption{Relative error $\|\tX^{(t)} - \tX^\ddagger\|_F/\|\tX^\ddagger\|_F$ vs iteration $t$ of FacTBREK on consistent linear system when $\tA$ is over-determined, $\tU$ is under-determined and $\tV$ is over-determined. We consider outer block sizes $|\mu_t| \in \{1, 10, 20\}$ and inner system block size (Upper Left) $|\nu_t| = 1$, (Upper Right) $|\nu_t| = 10$, (Lower) $|\nu_t| = 20$. }\label{fig:facTBREK_differentsystems_case2_2}
\end{figure}

\end{document}